\documentclass[11pt]{article}
\usepackage{amsmath,amsfonts,amsthm,graphicx}

\RequirePackage{amsmath,amsfonts,amssymb,latexsym} 
\RequirePackage{natbib}
\RequirePackage{authblk}
\RequirePackage{hyperref}
\RequirePackage{hypernat}
\RequirePackage{pst-all}
\RequirePackage{fncylab}
\RequirePackage{multirow}
\RequirePackage{booktabs}
\usepackage{setspace}
\newtheorem{lemma}{\textbf{Lemma}}[section]
\newtheorem{proposition}{\textbf{Proposition}}[section]
\newtheorem{assumption}{\textbf{Assumption}}[section]

\newtheorem{theorem}{\textbf{Theorem}}[section]

\newtheorem{remark}{Remark}[section]

\begin{document}
	\title{Bootstrap Consistency for Quadratic Forms of Sample Averages with Increasing Dimension}
	\author{Demian Pouzo \thanks{Contact: Dept. of Economics at UC Berkeley. 530 Evans Hall \# 3880. Berkeley, CA 94720. E-mail: dpouzo[at]econ.berkeley.edu. I would like to thank Xiaohong Chen, Noureddine El Karoui, Michael Jansson, Jim Powell and Elie Tamer for comments.}}
	\affil{UC Berkeley}
	\maketitle
	
	\begin{abstract}
		This paper establishes consistency of the weighted bootstrap for quadratic forms $\left( n^{-1/2} \sum_{i=1}^{n} Z_{i,n} \right)^{T}\left( n^{-1/2} \sum_{i=1}^{n} Z_{i,n} \right)$ where $(Z_{i,n})_{i=1}^{n}$ are mean zero, independent $\mathbb{R}^{d}$-valued random variables and $d=d(n)$ is allowed to grow with the sample size $n$, slower than $n^{1/4}$. The proof relies on an adaptation of Lindeberg interpolation technique whereby we simplify the original problem to a Gaussian approximation problem. We apply our bootstrap results to model-specification testing problems when the number of moments is allowed to grow with the sample size. 
	\end{abstract}

\section{Introduction}

% What we do: 
Since its introduction by \cite{efron1979} the bootstrap has been widely used as a method for approximating the distribution of statistics. Many papers have extended the original idea in terms, both, of the applicability (see \cite{Horowitz2001} and \cite{Hall1986} for excellent reviews) and of its methodology; of particular interest for us are the bootstrap procedures: \textquotedblleft wild bootstrap" (see \cite{mammen1993}) and more generally the  \textquotedblleft weighted bootstrap" (see \cite{Ma2005}).

In this paper we attempt to expand the applicability of the weighted bootstrap procedure to quadratic forms with increasing dimensions. Namely, we study quadratic forms of the form 
\begin{align}\label{eqn:1-QF}
	\left( \frac{1}{\sqrt{n}} \sum_{i=1}^{n} Z_{i,n} \right)^{T}\left( \frac{1}{\sqrt{n}} \sum_{i=1}^{n} Z_{i,n} \right)
\end{align} where $(Z_{1,n},...,Z_{n,n})$ are independent (among each other) $\mathbb{R}^{d}$-valued random variables with mean zero and general covariance matrix $\Sigma_{n}$. We show that its distribution is well-approximated (under the Kolmogorov distance) by the distribution of
\begin{align}
	\left( \frac{1}{\sqrt{n}} \sum_{i=1}^{n} \omega_{i,n} Z_{i,n} \right)^{T}\left( \frac{1}{\sqrt{n}} \sum_{i=1}^{n} \omega_{i,n} Z_{i,n} \right)
\end{align} where $(\omega_{1,n},...,\omega_{n,n})$ are independent \emph{bootstrap weights}. The novelty in this paper is that we allow for $d=d(n)$ to increase with the sample size.

%Why we do it:
 Studying the asymptotic behavior of quadratic forms, in particular establishing bootstrap consistency, is relevant since many statistics of interest can asymptotically be represented as quadratic forms of (scaled) sample averages. For instance, the likelihood ratio and Wald test statistics are asymptotically represented as quadratic forms of the scores; see \cite{VdV00} Ch. 16, and references therein. \cite{Portnoy_AOS88} establishes such representations for the likelihood ratio test statistics; there $d(n)$ is the dimension of the parameter of interest and is allowed to grow with $n$.  \cite{HmcKVK_AOS09} uses Portnoy's results to show a quadratic approximation result for Owen's (\cite{Owen-90}) empirical likelihood, allowing for $d(n)^{3}/n \rightarrow 0$; see also \cite{Peng-Schick-12}. %(2) Develop bootstrap but not for increasing dim [Check this!!]. (3) rely on \cite{Portnoy_AOS88}. (4) See also Peng and Shick.
 Therefore, by establishing the validity of the bootstrap for general quadratic forms, we propose an alternative method for inference for these statistics.

So as to further illustrate the applicability of our results, in Section \ref{sec:MST} we study a concrete application motivated by the work of \cite{DIN_JOE03} who consider model-specification tests for models defined by a \emph{diverging} number of moment conditions (this quantity determines our $d(n)$). By applying our results, we establish bootstrap consistency results for the distribution of the model-specification test statistics of two ubiquitous estimators in econometrics and statistics: The generalized empirical likelihood (GEL; \cite{Smith-EJ97}) estimator and The generalized method of moments (GMM; \cite{Hansen-ECMA82}) estimator. By employing our bootstrap result we are able to perform inference for \emph{non-optimally weighted} GMM estimators. To our knowledge these results are new.

By letting $d$ to increase with sample size in our general theory, we allow for different asymptotics, a  \textquotedblleft large-$d$ and large-$n$" asymptotics, rather than the standard  \textquotedblleft fixed-$d$ and large-$n$". The former type of asymptotics are more explicit about how the dimension, $d$, can affect the quality of the approximations. That is, even if the dimension does not literally grow with $n$, if, for instance, the model has a large number of parameters (or moment conditions as in our application), doing  \textquotedblleft fixed-$d$ and large-$n$" asymptotics could be misleading, whereas doing  \textquotedblleft large-$d$ and large-$n$" asymptotics could depict a more accurate picture of the behavior for fixed samples; see \cite{Mammen_AOS89} for discussion. Our results can also be applied in cases where there is literally a growing number of parameters. For instance, \cite{ChenPouzo14} study the asymptotic behavior of the quasi-likelihood ratio and Wald test statistics in a semi-parametric conditional moment setup; in particular they show that the statistics are asymptotically equivalent to quadratic forms (\ref{eqn:1-QF}) under a null hypothesis of increasing dimensions (see Appendix A.4 in their paper); our results, in conjunction with theirs, could be applied to establish bootstrap-based inference for the quasi-likelihood ratio and Wald test statistics.\footnote{In Section \ref{sec:MST} we provide more concrete examples of these two cases in the context of our application.}

%XXX Sieve/Series estimators XXX. Our results can be applied to these cases.

% How we do it:
In order to establish our main result of bootstrap consistency, we use Lindeberg interpolation techniques (see \cite{Chatterjee_AOP06}, \cite{Rollin2013} and references therein) to approximate the quadratic forms of $n^{-1/2}\sum_{i=1}^{n} \omega_{i,n} Z_{i,n}$  and $n^{-1/2}\sum_{i=1}^{n} Z_{i,n}$ by the ones for  Gaussian random variables with zero mean and covariance $n^{-1}\sum_{i=1}^{n} Z_{i,n}Z_{i,n}^{T}$ and  $E[Z_{1,n}Z_{1,n}^{T}]$, respectively. 
		
		By proceeding in this manner, we are able to reduce the original problem to a Gaussian approximation problem wherein we need to establish convergence of a Gaussian distribution with zero mean and variance $n^{-1} \sum_{i=1}^{n} Z_{i,n}Z_{i,n}^{T}$ to one with zero mean and variance $E[Z_{1,n}Z_{1,n}^{T}]$. We use Slepian interpolation (\cite{Slepian1962}, \cite{Rollin2013}, \cite{CCK2013} and references therein) to accomplish this. 
		
		Due to the interpolation techniques used here, we need certain restrictions on the higher moments of the random variables. In particular, we impose growth restrictions on the higher moments of the bootstrap weights and the Euclidean norm of $Z_{1,n}$. These conditions essentially restrict the growth rate of $d(n)$. Although the precise growth rate depends on such conditions, the dimensions cannot grow faster than $n^{1/4}$.

A number of papers develop large sample results allowing for increasing dimension. To name a few, \cite{Portnoy_AOS88} establishes the validity of the Wilks phenomenon for the likelihood ratio for exponential families when $d(n)^{3/2}/n \rightarrow 0$. \cite{He-Shao2000} derive the asymptotic distribution for M-estimators when the number of parameters is allowed to grow with the sample size. Recently, a few papers develop this type of results for quadratic forms of the form (\ref{eqn:1-QF}) allowing for increasing dimensions. In particular, \cite{Peng-Schick-12} and \cite{Xu_Zhang_Wu_14} develop a central limit theorem for quadratic forms of sample averages of vectors, allowing for the dimension to grow with $n$; both papers discuss several applications and examples. The results on our paper offer an alternative, bootstrap-based, method for inference for these cases. 

Our paper also contributes to the growing literature of bootstrap results allowing for increasing dimensions. \cite{Mammen_AOS89} derives asymptotic expansion for M-estimators in linear models allowing for increasing dimension and use them to show consistency of a weighted bootstrap. In a different context, \cite{Radulovic-98} uses Lindeberg interpolation methods allowing for increasing dimension to show that the functional bootstrap CLT holds under weaker conditions than equicontinuity; in his paper the restriction over the growth rate is $d(n)^{6}/n \rightarrow 0$. In \cite{CCK-AOS13}, the authors derive a Gaussian weighted bootstrap approximation result for the \emph{maximum} of the sum of high dimensional random vectors; in this specific setup the dimension is allowed to grow very fast, even at an exponential rate. \cite{Zhang-Cheng-14} provide an extension of \cite{CCK-AOS13} to time series. In our paper the object of interest is the $\ell^{2}$-norm of the sum of high dimensional random vectors (as opposed to the $\ell^{\infty}$-norm), so the results in these papers are not directly applicable. Finally, in a recent independent work, \cite{Spoikony-Zhilova-14} study the validity of the weighted bootstrap procedure for the likelihood ratio test statistics in finite samples and model misspecification; their results require $d(n)^{3}/n$ to be  \textquotedblleft small".

\medskip

\textbf{Organization of the Paper.} In Section \ref{sec:preli} we define the problem and impose the required assumptions. Section \ref{sec:main} presents the main Theorem and a discussion of its implications. Section \ref{sec:MST} presents an application to model-specification tests. Section \ref{sec:simul} presents a numerical simulations. Section \ref{sec:proof-boot} presents the proof of the main Theorem. Section \ref{sec:discussion} presents some concluding remarks.  In order to keep the paper short, the proofs of intermediate results are gathered in the appendix.

\medskip

\textbf{Notation.} For any vector $x \in \mathbb{R}^{d}$, we use $||x||^{p}_{p}$ to denote $\sum_{l=1}^{d} |x_{l}|^{p}$ and $x_{[l]}$ to denote the $l$-th coordinate of the vector. $tr\{ A \}$ denotes the trace of matrix $A$.  We use $E_{P}$ to denote the expectation with respect to the probability measure $P$; for conditional distributions $P(\cdot|X)$ we use $E_{P(\cdot|X)}[\cdot]$ or sometimes directly $E_{P}[\cdot |X]$. We use $ X_{n} \precsim Y_{n}$ to denote that $ X_{n} \leq C Y_{n}$ for some universal $C>0$. We use $\partial^{r} f$ to denote the $r$-th derivative of $f$; for the cases of $r=1$ and $r=2$ we use the more standard $f'$ and $f''$ notation. $wpa1-P$ means  \textquotedblleft with probability approaching one under $P$".

\section{Preliminaries}
\label{sec:preli}

 Let $\{ Z_{i,n} \in \mathbb{R}^{d(n)} : i=1,...,n~and~n \in \mathbb{N} \}$ with $(d(n))_{n\in \mathbb{N}}$ being a non-decreasing integer-valued sequence; $d(n)$ could diverge to infinity. For all $n \in \mathbb{N}$, let $Z^{n} \equiv (Z_{1,n},...,Z_{n,n})$ be independent among themselves with $Z_{i,n} \sim \mathbf{P}_{n}$ and $E_{\mathbf{P}_{n}} [ (Z_{i,n}) ] = 0$ and $\Sigma_{n} \equiv E_{\mathbf{P}_{n}} [ (Z_{i,n}) (Z_{i,n})^{T}  ] \in \mathbb{R}^{d(n)\times d(n)}$ positive definite and finite. Henceforth, we will typically omit the sub-index $n$ in $Z_{i,n}$. 
 
 Let $\mathbb{Z}_{n} \equiv n^{-1} \sum_{i=1}^{n} Z_{i}$, and
 \begin{align*}
 E_{\mathbf{P}_{n}} [ (\sqrt{n} \mathbb{Z}_{n}) ( \sqrt{n} \mathbb{Z}_{n} )^{T}   ] = n^{-1} \sum_{i=1}^{n} E_{\mathbf{P}_{n}} [ (Z_{i}) ( Z_{i} )^{T}   ] = \Sigma_{n}.
 \end{align*}
 
 For a given matrix $A \in \mathbb{R}^{ d \times d }$ we denote its eigenvalues as $\{ \lambda_{1}(A),..., \lambda_{d}(A)\} $.
 
 \begin{assumption} %[{ass:data-Z}]
 	\label{ass:data-Z}
% 	(i) $\limsup_{n \rightarrow \infty }\frac{ tr \{ \Sigma^{3}_{n}  \}  }{ ( tr \{ \Sigma^{3}_{n}  \}  )^{3/2}  } =0 $, $\frac{ tr \{ \Sigma_{n}  \}  }{ tr \{ \Sigma^{2}_{n}  \}   } \leq C < \infty $ and $ \sqrt{\frac{tr \{ \Sigma_{n} \}  }{n} } \max \{  E_{\mathbf{P}_{n}} [ ||Z_{1,n}||^{3}_{2} ]  , (tr \{ \Sigma_{n} \} )^{3/2} \}  = o(1) $;
 (i) There exist constants $0 < c \leq C < \infty$ such that $ c \leq \lambda_{l}(\Sigma_{n}) \leq C$ for any  $l=1,...,d(n)$ and $n \in \mathbb{N}$, and\begin{align*}
 	\frac{     \max \{ d(n) (E_{\mathbf{P}_{n}} [ ||Z_{1,n}||^{3}_{2} ] )^{2} , E_{\mathbf{P}_{n}} [ ||Z_{1}||^{4}_{2} ]  , (d(n))^{4} \} }{n} = o(1);
 \end{align*}
%(ii) there exists a $\gamma>0$ such that $n^{-1} E_{\mathbf{P}_{n}} [ ||Z_{1}||^{4}_{2} ] = o(1)  $
%and $\frac{(tr \{ \Sigma_{n} \} )^{2+\gamma} }{n^{\gamma}} E_{\mathbf{P}_{n}} [ ||Z_{1}||^{4+2\gamma}_{2} ] = o(1)  $ (iii) there exists a $\kappa>0$ such that $\frac{(\log(d(n)))^{\kappa/2} d(n)^{2+\kappa} }{ n^{1+\kappa/2} } E_{\mathbf{P}_{n}} [ ||Z_{1}||^{2(2+\kappa)}_{2+\kappa} ] = o(1) $. 
(ii) there exists a $\gamma>0$ such that $\frac{(d(n))^{2+\gamma} }{n^{\gamma}} E_{\mathbf{P}_{n}} [ ||Z_{1}||^{4+2\gamma}_{2} ] = o(1)  $;
(iii) there exists a $\kappa \geq 0$ such that $\frac{(\log(d(n)))^{\kappa/2} d(n)^{2+\kappa} }{ n^{1+\kappa/2} } E_{\mathbf{P}_{n}} [ ||Z_{1}||^{2(2+\kappa)}_{2+\kappa} ] = o(1) $. 
 \end{assumption}

 \subsection{Discussion of the Assumption \ref{ass:data-Z}}

The assumption that $ c \leq \lambda_{l}(\Sigma_{n}) \leq C$  can be somewhat relaxed; for instance, it could be replaced by	$\limsup_{n \rightarrow \infty }\frac{ tr \{ \Sigma^{3}_{n}  \}  }{ ( tr \{ \Sigma^{2}_{n}  \}  )^{3/2}  } =0 $ and $\frac{ tr \{ \Sigma_{n}  \}  }{ tr \{ \Sigma^{2}_{n}  \}   } \leq C < \infty $. The rest of Assumption \ref{ass:data-Z} essentially imposed restrictions on the rate of growth of $d(n)$ relative to $n$. In order to provide sufficient conditions for this part of Assumption \ref{ass:data-Z}, it is convenient to provide bounds in terms of $d(n)$ for the quantities $E_{\mathbf{P}_{n}}[||Z_{1}||_{2}^{q}]$ (for different $q$'s) and $E_{\mathbf{P}_{n}} [ ||Z_{1}||^{2(2+\kappa)}_{2+\kappa} ]$ in the assumption.

%In order to shed more light on the implications of this part, and to provide sufficient conditions for it, is convenient to bound the quantities of the form $E_{\mathbf{P}_{n}}[||Z_{1}||_{2}^{q}]$ (for different $q$'s) and $E_{\mathbf{P}_{n}} [ ||Z_{1}||^{2(2+\kappa)}_{2+\kappa} ]$ in the assumption,  in terms of $d(n)$.

Clearly, if $|Z_{[l],1}| \leq C <\infty$ a.s-$\mathbf{P}_{n}$ for all $l = 1,...,d(n)$ and all $n \in \mathbb{N}$, then $E_{\mathbf{P}_{n}}[||Z_{1}||_{2}^{2q}]  = O(d(n)^{q})$ for any $q>0$.\footnote{Recall that for a vector $x$, $x_{[l]}$ denotes the $l$-th component.} For example, such condition is imposed by \cite{Vershynin-JTP12} in the context of estimation and approximation of covariance matrices of high dimensional distributions.

The next lemma shows that the result still holds if we impose the following (milder) restriction:  $E_{\mathbf{P}_{n}} \left[e^{\lambda Z_{[l],1}^{2} } \right] \leq C < \infty $ for some $\lambda>0$. For instance, if $(Z_{[l],1})^{2}$ is a sub-Gamma random variable (\cite{Boucheron-book13} p. 27), then the condition holds since $E_{\mathbf{P}_{n}} \left[e^{\lambda Z_{[l],1}^{2} } \right] \leq \exp \{  \frac{\lambda^{2} v }{2(1-c\lambda)}   \}$ for any $\lambda \in (0,1/c)$ and some $c>0$. If $Z_{[l],1}$ is sub-Gaussian, then $(Z_{[l],1})^{2}$ is sub-exponential (see \cite{Vershynin-12} Lemma 5.14) and the condition holds by the same argument.

An appealing feature of this result is that it only imposes restrictions on the marginal behavior of the components of the vector $Z_{1}$ and not on its joint behavior.

\begin{lemma}
	Suppose that there exists a $C>0$ and $\lambda> 0$ such $E_{\mathbf{P}_{n}} \left[e^{\lambda Z_{[l],1}^{2} } \right] \leq C$ for all $l=1,...,d(n)$ and all $n \in \mathbb{N}$. Then $E_{\mathbf{P}_{n}}[||Z_{1}||_{2}^{2q}] \precsim d(n)^{q}$ for any $q>0$.
\end{lemma}

\begin{proof}
	Observe that
	\begin{align*}
		E_{\mathbf{P}_{n}}[(||Z_{1}||^{2}_{2}/d(n) )^{q}]  = & \int_{0}^{\infty} \mathbf{P}_{n} \left(  ||Z_{1}||^{2}_{2}/d(n) \geq t^{1/q}   \right) dt \\
		= & q \int_{0}^{\infty} u ^{q-1}\mathbf{P}_{n} \left(  ||Z_{1}||^{2}_{2}/d(n) \geq u   \right) du
	\end{align*}
	since $||Z_{1}||^{2}_{2}/d(n) = d(n)^{-1}\sum_{l=1}^{d(n)} |Z_{[l],1}|^{2}$, by the Markov inequality it follows that for any $\lambda >0$
	\begin{align*}
			E_{\mathbf{P}_{n}}[(||Z_{1}||^{2}_{2}/d(n) )^{q}]  \leq  \left( q \int_{0}^{\infty} u ^{q-1} e^{-\lambda u} du \right) E_{\mathbf{P}_{n}} \left[ e^{ \lambda d(n)^{-1}\sum_{l=1}^{d(n)} |Z_{[l],1}|^{2}} \right] .
	\end{align*}
	By Jensen inequality $E_{\mathbf{P}_{n}} \left[ e^{ \lambda d(n)^{-1}\sum_{l=1}^{d(n)} |Z_{[l],1}|^{2}} \right] \leq d(n)^{-1}\sum_{l=1}^{d(n)}  E_{\mathbf{P}_{n}} \left[ e^{ \lambda |Z_{[l],1}|^{2}} \right] \leq C$. Thus, the desired result follows from the fact that $\left( q \int_{0}^{\infty} u ^{q-1} e^{ - \lambda u} du \right)  = \left( q \lambda^{-q} \int_{0}^{\infty} w ^{q-1} e^{ - w } dw \right)  = q \lambda^{-q}  \Gamma(q) < \infty$ for any $q>0$.
\end{proof}

Under the conditions in the lemma, Assumption \ref{ass:data-Z}(i) boils down to $\frac{d(n)^{4}}{n} = o(1) $. For Assumption \ref{ass:data-Z}(ii) is sufficient to impose $\frac{d(n)^{4+2\gamma}}{n^{\gamma}} = o(1)$; for $\gamma = 2$ it boils down to $\frac{d(n)^{4}}{n} = o(1)$ but for large $\gamma$ it (roughly) becomes $\frac{d(n)^{2}}{n} = o(1)$. Finally, for, say $\kappa = 0$, Assumption \ref{ass:data-Z}(iii) is reduced to $\frac{ d(n)^{2} }{ n } E_{\mathbf{P}_{n}} [ ||Z_{1}||^{4}_{2} ] \precsim  \frac{ d(n)^{4} }{ n } \rightarrow 0$. 

That is, under conditions that bound all (polynomial) moments of the individual components of $Z_{1}$, the dimension is allowed to grow slower than the 4th-root of the sample size.

\subsection{The Bootstrap Weights}

The bootstrap weights are given by $\{
\omega_{in} \in \mathbb{R} : i=1,...,n~and~n \in \mathbb{N} \}$ where, for
any $n \in \mathbb{N}$ and conditional on $Z^{n} = z^{n}$,
$(\omega_{1n},...,\omega_{nn}) \sim \mathbf{P}^{\ast}_{n}(\cdot | z^{n}) $ for some
$\mathbf{P}^{\ast}_{n}(\cdot | z^{n}) $. 

%\begin{assumption}[{ass:boot-w}]\label{ass:boot-w}
% For all $n
%  \in \mathbb{N}$ and $i=1,2,...,n$, (i)
%  $E_{\mathbf{P}^{\ast}_{n}(\cdot | Z^{n})}\left[  \omega_{in} | (\omega_{jn})_{j=1}^{i-1} \right] = 1 $
%  and $E_{\mathbf{P}^{\ast}_{n}(\cdot | Z^{n})}\left[  (\omega_{in} -1 )^{2} | (\omega_{jn})_{j=1}^{i-1} \right] =
%  \sigma^{2}_{n} \leq c < \infty$.
%\end{assumption}

\begin{assumption} %[{ass:boot-w}]
	\label{ass:boot-w}
 For all $n
  \in \mathbb{N}$ and $i=1,2,...,n$, (i) $(\omega_{1n},...,\omega_{nn})$ are independent and
  $E_{\mathbf{P}^{\ast}_{n}(\cdot | Z^{n})}\left[  \omega_{in} \right] = 0 $
  and $E_{\mathbf{P}^{\ast}_{n}(\cdot | Z^{n})}\left[  (\omega_{in} )^{2} \right] = 1$; (ii) there exists a $q \geq \max\{\gamma+2,4\}$, such that $E_{\mathbf{P}^{\ast}_{n}(\cdot | Z^{n})}\left[  |\omega_{in} |^{q} \right]  \leq C_{w} < \infty$ for some constant $C_{w}>0$.
\end{assumption}

Part (i) is standard. Part (ii) is mild considering that the weights are chosen by the researcher.\footnote{Of course, the technique of proof can be applied to the case where the following (stronger) restriction is imposed: $E_{\mathbf{P}^{\ast}_{n}(\cdot | Z^{n})}\left[  \exp \{ \omega_{in}  \} \right] \leq C_{w} < \infty $.}

\section{The Main Result}
\label{sec:main}

We now present the main result of the paper. In what follows, for any measurable function $z^{n} \mapsto f(z^{n})$ we use $|f(Z^{n})| = o_{\mathbf{P}_{n}}(1)$ to denote: For any $\varepsilon>0$, there exists a $N(\varepsilon)$ such that for all $n \geq N(\varepsilon)$, $\mathbf{P}_{n}( |f(Z^{n}) | \geq \varepsilon  ) < \varepsilon$.

Let $\mathbb{Z}^{\ast}_{n} \equiv n^{-1} \sum_{i=1}^{n} \omega_{in} Z_{i}$ be the bootstrap analog of $\mathbb{Z}_{n}$. 

\begin{theorem}\label{thm:boot}
%	Suppose XXX hold. Then\begin{align}
%		\sup_{t \in \mathbb{R}} \left| \mathbf{P}^{\ast}_{n} \left( \frac{(\sqrt{n} \mathbb{Z}^{\ast}_{n})^{T} ( \sqrt{n} \mathbb{Z}^{\ast}_{n} ) - s_{1,n}}{s_{2,n}} \geq t \mid Z^{n}  \right)   -  \mathbf{P}_{n} \left( \frac{(\sqrt{n} \mathbb{Z}_{n})^{T} ( \sqrt{n} \mathbb{Z}_{n} ) - s_{1,n}}{s_{2,n}} \geq t   \right)     \right| = o_{\mathbf{P}}(1)
%	\end{align}
	Suppose Assumption \ref{ass:data-Z} and \ref{ass:boot-w}  hold. Then\begin{align}\label{eqn:main}
		\sup_{t \in \mathbb{R}} \left| \mathbf{P}^{\ast}_{n} \left( ||\sqrt{n} \mathbb{Z}^{\ast}_{n}||^{2}_{2} \geq t \mid Z^{n}  \right)   -  \mathbf{P}_{n} \left( ||\sqrt{n} \mathbb{Z}_{n}||^{2}_{2} \geq t   \right)     \right| = o_{\mathbf{P}_{n}}(1).
	\end{align}
\end{theorem}

\subsection{Comments and discussion}

We now present some remarks and discuss some implications of the preceding Theorem.

\bigskip

\textbf{Heuristics.} We postpone the somewhat long proof of the Theorem to Section \ref{sec:proof-boot}; here we present an heuristic argument. The first step of the proof is to apply Lindeberg interpolation techniques (see \cite{Chatterjee_AOP06} and \cite{Rollin2013} and references therein) to approximate $\sqrt{n} \mathbb{Z}^{\ast}_{n}$ by $\sqrt{n} \mathbb{U}_{n}$ and $\sqrt{n} \mathbb{Z}_{n}$ by $\sqrt{n} \mathbb{V}_{n}$,  where $\mathbb{U}_{n}$ and $\mathbb{V}_{n}$ are Gaussian random variables with zero mean and covariances $n^{-1} \sum_{i=1}^{n} Z_{i}Z_{i}^{T}$ and $E[Z_{1,n}Z_{1,n}^{T}]$  respectively. 

In order to do this, we first approximate the indicator function $x \mapsto 1\{ ||x||^{2}_{2}  \geq t \}$ by  \textquotedblleft smooth" functions $x \mapsto \mathcal{P}_{t,\delta,h}(||x||^{2}_{2})$; the exact expression for $\mathcal{P}_{t,\delta,h}$ is presented in Lemma \ref{lem:Zn-strong-weak} and follows from the suggestion by \cite{Pollard_01} p. 247. The functions are indexed by $(h,\delta)$ where $h$ is \textquotedblleft small'' compared to $\delta$, and the \textquotedblleft smaller" $\delta$ is, the closer the function $\mathcal{P}_{t,\delta,h}$ is to the indicator function; see Lemmas \ref{lem:Zn-strong-weak}, \ref{lem:Vn-strong-weak} and \ref{lem:ZZ-strong-weak}  in the Appendix \ref{app:lemmas-main}.  It is worth to note that  what we mean by $\delta$ to be ``small'' depends on how $||\sqrt{n}\mathbb{V}_{n}||^{2}_{2}$ concentrates mass. Lemma \ref{lem:anticoncentration-V} in the Appendix \ref{app:lemmas-main} establishes an anti-concentration result, wherein we obtain that this random variable puts very little mass in any given interval. Therefore $\delta$ could actually be quite large, of the order of $\sqrt{tr\{ \Sigma^{2}_{n} \}}$.

Second, since $x \mapsto \mathcal{P}_{t,\delta,h}(||x||^{2}_{2})$ belongs to a class of  \textquotedblleft smooth" functions, we show that it suffices to show consistency under the weak norm (as opposed to the norm implied in \ref{eqn:main}).\footnote{The formal definition of the norm is presented in Equation \ref{eqn:WN} in Section \ref{sec:proof-boot}.} This is done in  Lemmas \ref{lem:Delta-P-Pr} and \ref{lem:Delta-Pwz-Pr}. The relevant class of  \textquotedblleft smooth" functions is given by $\mathcal{C}_{M}$, which is the class of functions $f : \mathbb{R} \rightarrow \mathbb{R}$ that are three times continuously differentiable and $ \sup_{x} | \partial^{r} f(x)| \leq (M)^{r}$ and $\sup_{x} |f(x)| \leq 1$.

The following Theorems formalize the aforementioned approximation of $\sqrt{n} \mathbb{Z}^{\ast}_{n}$ by $\sqrt{n} \mathbb{U}_{n}$ and $\sqrt{n} \mathbb{Z}_{n}$ by $\sqrt{n} \mathbb{V}_{n}$ and can be viewed of independent interest since they show that a  \textquotedblleft generalized invariance principle" holds in our setup. Henceforth, we use $\boldsymbol{\Phi}^{\ast}_{n}(\cdot|Z^{n})$ and $\boldsymbol{\Phi}_{n}$ respectively, to denote their probability distributions.

%The second step uses the fact that $\mathcal{P}_{t,\delta,h}$ belongs to a class of  \textquotedblleft smooth" functions, and applies Lindeberg interpolation techniques (see \cite{Chatterjee_AOP06} and \cite{Rollin2013} and references therein) to approximate $\sqrt{n} \mathbb{Z}^{\ast}_{n}$ by $\sqrt{n} \mathbb{U}_{n} \equiv n^{-1/2} \sum_{i=1}^{n} U_{i,n}$ and $\sqrt{n} \mathbb{Z}_{n}$ by $\sqrt{n} \mathbb{V} \equiv n^{-1/2} \sum_{i=1}^{n} V_{i,n}$,  where $(U_{i,n})_{i=1}^{n}$ are independent Gaussian with zero mean and variance $Z_{i}Z_{i}^{T}$ and $(V_{i,n})_{i=1}^{n}$ are independent Gaussian with zero mean and variance $E[Z_{1,n}Z_{1,n}^{T}]$ XXX (1) MATAR $U_{i,n}$ Y SOLO PONER $V_{i,n}$ DONDE SE USA XXX. We use $\boldsymbol{\Phi}^{\ast}_{n}(\cdot|Z^{n})$ and $\boldsymbol{\Phi}_{n}$ respectively, to denote their probability distributions. The following theorems formalize this, and can be viewed of independent interest since they show that a  \textquotedblleft generalized invariance principle" holds in our setup.

\begin{theorem}\label{thm:weak-norm-boot} %[{thm:weak-norm-boot}] 
%	Suppose assumption \ref{ass:boot-w} and \ref{ass:data-Z}. For any $\varepsilon>0$, any $h > 0$, there exist a $N(\varepsilon)$ such that for all $n \geq N(\varepsilon)$
%	\begin{align*}
%	\mathbf{P}_{n} \left( 	\sup_{f \in \mathcal{C}_{h^{-1}}} \left|  E_{\mathbf{P}^{\ast}_{n}} \left[ f \left( || \sqrt{n} \mathbb{Z}^{\ast}_{n}||^{2}_{2}  \right)  |Z^{n} \right] - E_{\boldsymbol{\Phi}^{\ast}_{n}} \left[ f \left( || \sqrt{n} \mathbb{U}_{n} ||^{2}_{2} \right)  | Z^{n} \right]  \right| \geq \varepsilon   \right)  \leq \varepsilon
%	\end{align*}
	Suppose Assumption \ref{ass:data-Z} and \ref{ass:boot-w} hold. For any $h > 0$,
	\begin{align*}
		\sup_{f \in \mathcal{C}_{h^{-1}}} \left|  E_{\mathbf{P}^{\ast}_{n}} \left[ f \left( || \sqrt{n} \mathbb{Z}^{\ast}_{n}||^{2}_{2}  \right)  |Z^{n} \right] - E_{\boldsymbol{\Phi}^{\ast}_{n}} \left[ f \left( || \sqrt{n} \mathbb{U}_{n} ||^{2}_{2} \right)  | Z^{n} \right]  \right| = o_{\mathbf{P}_{n}}(h^{-2}).
	\end{align*}
\end{theorem}

\begin{proof}
	See Appendix \ref{app:lindeberg}.
\end{proof} 

\begin{theorem}\label{thm:weak-norm-ori} %[{thm:weak-norm-ori}] 
	Suppose Assumption \ref{ass:data-Z} and \ref{ass:boot-w} hold. For any $h > 0$,
	\begin{align*}
	\sup_{f \in \mathcal{C}_{h^{-1}}} \left|  E_{\mathbf{P}_{n}} \left[ f \left( || \sqrt{n} \mathbb{Z}_{n}||^{2}_{2}  \right)\right] - E_{\boldsymbol{\Phi}_{n}} \left[ f \left( || \sqrt{n} \mathbb{V}_{n} ||^{2}_{2} \right)  \right]  \right| = o(h^{-2}).
	\end{align*}
\end{theorem}

\begin{proof}
	See Appendix \ref{app:lindeberg}.
\end{proof} 

By using Theorems \ref{thm:weak-norm-boot} and \ref{thm:weak-norm-ori} we have reduced the original problem to a Gaussian approximation problem. That is, we need to establish convergence (under the distance induced by $\mathcal{C}$) of a Gaussian distribution with zero mean and variance $n^{-1} \sum_{i=1}^{n} Z_{i}Z_{i}^{T}$ to one with zero mean and variance $E[Z_{1}Z_{1}^{T}]$. Lemma \ref{lem:slep} in Section \ref{sec:proof-boot} --- which is based in the Slepian interpolation (see \cite{CCK-AOS13}, \cite{CCK2013} and \cite{Rollin2013} and references therein)--- establishes that is enough to show that \begin{align}\label{eqn:LLN-cov}
 d(n) \max_{1 \leq j,l \leq d(n)} \left| n^{-1} \sum_{i=1}^{n} Z_{[j],i} Z_{[l],i} - E_{\mathbf{P}_{n}}[Z_{[j],1} Z_{[l],1}]  \right| = o_{\mathbf{P}_{n}}(1).
\end{align}
% THE RATE IS CORRECT BECAUSE h^{-2} ~ d(n)
In Section \ref{sec:proof-boot}, we show that, employing standard arguments, the expression \ref{eqn:LLN-cov} holds under our assumptions. A similar result is obtained by \cite{CCK-AOS13} without the scaling factor of $d(n)$; their setup, however, is different since the object of interest is $\max_{1\leq j \leq d(n)} |n^{-1/2} \sum_{i=1}^{n} Z_{[j],i}|$ (as opposed to $||n^{-1/2} \sum_{i=1}^{n} Z_{i}||^{2}_{2}$). \footnote{An important consequence of this difference is that, as opposed to our case, \cite{CCK-AOS13} can use a \textquotedblleft smooth maximum function" to approximate their quantity of interest; the approximation error is only of order $\log d$. This, allows them to obtain faster rates for the approximation of the indicator functions with smooth functions. This, in turn, translates into a faster overall rate of convergence ---  $d = o(\exp(n))$ in their case. See \cite{Wasserman_2014} for a discussion and a nice review of these results.}

\bigskip

\textbf{Asymptotic Distribution of $||\sqrt{n} \mathbb{Z}_{n}||^{2}_{2}$.} An implication of the proof of Theorem \ref{thm:boot} and Theorem \ref{thm:weak-norm-ori} is that
\begin{align}\label{eqn:asym-Gauss}
\sup_{t \in \mathbb{R}} \left| \mathbf{P}_{n} \left( \frac{ ||\sqrt{n} \mathbb{Z}_{n}||^{2}_{2} - d(n) }{\sqrt{d(n)}} \geq t   \right)   -  \boldsymbol{\Phi}_{n} \left( \frac{ ||\sqrt{n} \mathbb{V}_{n}||^{2}_{2} - d(n)}{ \sqrt{d(n)} } \geq t   \right)     \right| = o(1).
\end{align}

That is, if $\Sigma_{n} = I_{d(n)}$ then this expression and a direct application of the CLT (when $d(n) \rightarrow \infty$) imply that $\frac{ ||\sqrt{n} \mathbb{Z}_{n}||^{2}_{2} - d(n) }{\sqrt{2 d(n)}}  \Rightarrow N(0,1)$ or, informally, $||\sqrt{n} \mathbb{Z}_{n}||^{2}_{2}$ is approximately chi-square distributed with $d(n)$ degrees of freedom. When $\Sigma_{n} \ne I_{d(n)}$, the last claim is no longer true but it holds that $\frac{ ||\sqrt{n} \mathbb{Z}_{n}||^{2}_{2} - tr\{ \Sigma_{n} \} }{\sqrt{2 tr\{\Sigma^{2}_{n} \}}}$ is approximately distributed as $\sum_{j=1}^{d(n)} \frac{\lambda_{j}(\Sigma_{n})(\chi_{j}-1)}{\sqrt{2\sum_{j=1}^{d(n)} \lambda^{2}_{j}(\Sigma_{n})}}$ with $\chi^{2}_{j}$ drawn from a chi-square with degree one; see \cite{Xu_Zhang_Wu_14} and \cite{Peng-Schick-12} for a discussion regarding these results.

We note that in Theorem \ref{thm:boot} no scaling (by $-d(n)$ and $1/\sqrt{2d(n)}$ or $-tr\{\Sigma_{n} \}$ and $1/\sqrt{2 tr\{ \Sigma^{2}_{n} \}}$) is needed. That is, although the mean and variance of $||\sqrt{n} \mathbb{Z}_{n}||^{2}_{2}$ are  \textquotedblleft drifting" to infinity, the bootstrap still provides a good approximation since the moments of $||\sqrt{n} \mathbb{Z}^{\ast}_{n}||^{2}_{2}$ are mimicking this behavior.   

\bigskip

\textbf{On the Lindeberg Interpolation.} Theorems \ref{thm:weak-norm-boot} and \ref{thm:weak-norm-ori} are based on the following Lindeberg interpolation for quadratic forms.\footnote{This Lindeberg interpolation builds on the approach in \cite{Xu_Zhang_Wu_14}. }

\begin{theorem} %[{thm:LFQ-bound}] 
	\label{thm:LFQ-bound}
	Let $(A_{1},...,A_{n}) \in \mathbb{R}^{d \times n}$ and
	$(B_{1},...,B_{n}) \in \mathbb{R}^{d \times n}$ be random
	matrices independent from each other. Suppose for each $1 \leq i \leq n$, $A_{i}$ has finite
	second moments with $E[A_{i}]=0$, $A_{1},...,A_{n}$ are independent, and $B_{i}$ has finite second moments, with $E[B_{i}]=0$ and 
	$B_{1},...,B_{n}$ are independent. Suppose $E[A_{i}A_{i}^{T}]=E[B_{i}B_{i}^{T}] \equiv C_{i}$. Let $f : \mathbb{R} \rightarrow \mathbb{R}$ be three times
	differentiable and for $r=1,2,3$, $|\partial^{r} f(\cdot) | \leq
	L_{r}(f)$. Then for any $\epsilon>0$ and for any $q>0$ 
	\begin{align*}
	| E[f( ||\sum_{i=1}^{n} A_{i}||^{2}_{2} )] -  E[f( ||\sum_{i=1}^{n} B_{i}||^{2}_{2} )]|	\leq  \mathbf{S}_{n} + L_{2}(f) \left( \frac{ L_{3}(f) }{L_{2}(f)}  \right)^{q} \mathbf{R}_{n}
	\end{align*}
	where $\mathbf{S}_{n} =  \mathbf{S}_{1,n}  + \mathbf{S}_{2,n} $, with 
	\begin{align*}
	\mathbf{S}_{1,n}  = & \sum_{i=1}^{n}| E\left[ f''  \left( || \mathbb{S}_{i:n} ||^{2}_{2}  \right) \right]  E[||B_{i}||^{4}_{2} ] - E [||A_{i}||^{4}_{2}   ] |\\
	\mathbf{S}_{2,n}  = & 4 \sum_{i=1}^{n} |E\left[ f''  \left( || \mathbb{S}_{i:n} ||^{2}_{2}  \right) \mathbb{S}_{i:n}^{T} \right] \left( E [  B_{i} || B_{i}||^{2}_{2} ] -  E [ A_{i} || A_{i}||^{2}_{2}   ] \right)|\\
	\mathbf{R}_{n} = &  \sum_{i=1}^{n}    E \left[ \left(  \mathbb{S}_{i:n}^{T} B_{i}  + ||B_{i}||^{2}_{2} \right)^{2+q}  + \left(  \mathbb{S}_{i:n}^{T} A_{i}  + ||A_{i}||^{2}_{2}   \right)^{2+q} \right] 
	\end{align*} 
	and $\mathbb{S}_{i:n} \equiv \sum_{j=1}^{i-1} A_{j} + 0 + \sum_{j=i+1}^{n} B_{j}$. %\equiv \sum_{j=1}^{n} S_{j}$.
\end{theorem}

\begin{proof}
	See Appendix \ref{app:lindeberg}.
\end{proof}

%A few remarks regarding this theorem are in order. XXX MOVE THIS TO THE DISCUSSION SECTION First, in lemma \ref{lem:LFQ-bound} in the Appendix we provide bounds for $\mathbf{S}_{n}$ (and $\mathbf{R}_{n}$). These bounds only use restrictions imposed on the higher moments of the original data and the bootstrap weights (see assumptions \ref{ass:data-Z}(i)(ii) and \ref{ass:boot-w}). However, it is easy to see that if one would have additional information on the higher moments, one could obtain sharper bounds for $\mathbf{S}_{n}$. For instance, to show Theorem \ref{thm:weak-norm-boot}, we apply theorem \ref{thm:LFQ-bound} with $A_{i} = n^{-1/2} \omega_{i,n} Z_{i} $ and $B_{i} = n^{-1/2} u_{i} Z_{i}$ with $u_{i} \sim N(0,1)$. If we would have that $(\omega_{i,n})_{i=1}^{n}$ were such that $E[|\omega_{i,n}|^{4}]=E[(z)^{4}]$ with $z \sim N(0,1)$, then $\mathbf{S}_{1,n} = 0$ (a similar observation applies to $\mathbf{S}_{2,n}$). These bounds in $\mathbf{S}_{n}$, in turn, may translate to faster rates of the bootstrap approximation. XXX

It is worth pointing out that the interpolation compares the quantities $\sum_{i=1}^{n} A_{i}$ with $\sum_{i=1}^{n} B_{i}$ by comparing  \textquotedblleft one component at a time". This comparison is essentially divided into two parts. First, we compare $||\mathbb{S}_{i:n} + A_{i} ||^{2}_{2}$ and $||\mathbb{S}_{i:n}+  B_{i}||^{2}_{2}$, which are real-valued quantities. Second, we exploit the smoothness of the \emph{univariate} function $f$ to bound its variation using Taylor's approximation. Loosely speaking, the first step reduces a $d(n)$-dimensional problem to an univariate one. An alternative approach would be to consider interpolations for \emph{multivariate} functions (e.g. \cite{Chatterjee-Meckes-2008}) of the form $g : \mathbb{R}^{d(n)} \rightarrow \mathbb{R}$ with $g(x) \equiv f(||x||^{2}_{2})$. As can be seen from the derivations in \cite{Chatterjee-Meckes-2008}, the remainder term will also require bounds on higher derivatives of $g$ (and thus $f$), but of the form $\sup_{x \ne y} \frac{\left \Vert Hess(g)(x) - Hess(g)(y) \right \Vert_{op} }{||x-y||_{2}}$. \footnote{$Hess(g)$ is the Hessian of the function and $||.||_{op}$ is the operator norm. Other type of bounds could be found in \cite{Raic2014} based on Hilbert-Schmidt norm.} Which approach is better depends largely on what type of restrictions over the class of test functions are natural in the problem at hand. For us, $||\partial^{r} f||_{L^{\infty}} < \infty$ is a natural assumption, but in other applications it could be too strong. 

More generally, this discussion illustrates the relationship between restrictions in the class of test functions ($\mathcal{C}$) and the bounds on higher order moments and ultimately the rate of growth of $d(n)$.

\bigskip

\textbf{Bootstrap P-Value.} For any $\alpha \in (0,1)$ and $Z^{n} \in \mathbb{R}^{d(n)}$, let $t_{n}(\alpha,Z^{n}) \equiv \inf\{ t :  \mathbf{P}^{\ast}_{n} \left( ||\sqrt{n} \mathbb{Z}^{\ast}_{n}||^{2}_{2} \leq t \mid Z^{n}  \right)  \geq \alpha  \}$. Due to the distribution consistency result proven in Theorem \ref{thm:boot}, we can approximate the $\alpha$-th quantile of the distribution of $ ||\sqrt{n} \mathbb{Z}_{n}||^{2}_{2}  $ by $t_{n}(\alpha,Z^{n})$, in the sense that\begin{align*}
   \mathbf{P}_{n} \left( ||\sqrt{n} \mathbb{Z}_{n}||^{2}_{2} \geq t_{n}(\alpha,Z^{n}) - \eta  \right) \leq \alpha + o(1)
\end{align*}
for any $\eta>0$. If $t_{n}(\alpha,Z^{n})$ is a continuity point of $\mathbf{P}^{\ast}_{n} \left( \cdot | Z^{n} \right)$, then\begin{align*}
	\mathbf{P}^{\ast}_{n} \left( ||\sqrt{n} \mathbb{Z}^{\ast}_{n}||^{2}_{2}  \geq t_{n}(\alpha,Z^{n}) \mid Z^{n}  \right)  = \alpha,
\end{align*} and the first display becomes $\mathbf{P}_{n} \left( ||\sqrt{n} \mathbb{Z}_{n}||^{2}_{2} \geq t_{n}(\alpha,Z^{n}) \right) = \alpha + o(1)$. Hence, Theorem \ref{thm:boot} can be used to construct valid p-values based on the bootstrap.

\section{An application to model specification tests for GEL and GMM estimators}\label{sec:MST}

In this section we apply our results to construct bootstrap-based specification tests for models with increasing number of moment restrictions. We do this for two estimators: generalized method of moment (GMM; see \cite{Hansen-ECMA82}) estimator and generalized empirical likelihood (GEL; see \cite{Smith-EJ97}) estimator. Both estimators are widely used in econometrics and statistics and encompass a wide range of commonly used estimators such as Z-estimators (\cite{VdV00} Ch. 5), and empirical likelihood estimator (\cite{Owen-BIO88}), respectively.\footnote{See \cite{Imbens-JBES02} for additional examples and a discussion. See also \cite{Hall} for a review for GMM.}

 In models characterized by moment conditions, model-specification tests (MST) allow us to check whether the moment conditions match the data well or not. In this setup with increasing moment restrictions, MST has been studied by \cite{DIN_JOE03} (DIN, henceforth); see also \cite{Djong-Bierens-ET94}. They show that the MST statistic is asymptotically a quadratic form of scaled sample averages; however, they rely on inferential methods build on expressions akin to \ref{eqn:asym-Gauss}. Instead, by applying our Theorem \ref{thm:boot}, we can use the weighted bootstrap method to approximate the asymptotic distribution of MST statistics; thus complementing their results by providing an alternative way of constructing asymptotic p-values. Moreover, as explained below, by not relying on CLT-type results to approximate the limiting distribution, we are able to provide valid asymptotic inference for a larger class of GMM estimators than the one considered in DIN.

% Moreover, by doing this we not only provide an alternative inferential procedure for their existing cases, but also, as we argue below, expand the class of GMM estimators for which we can apply the model-specification tests suggested therein.

The setup closely follows that of DIN and is as follows. Suppose $(X_{i})_{i=1}^{n}$ is an i.i.d. sample of real-valued random variables with $X_{i} \sim \mathbf{P}_{n}=\mathbf{P}$. The model we consider is one where the true parameter of interest, $\theta_{0} \in Int(\Theta)$ --- with $\Theta$ a compact subset of $\mathbb{R}^{q}$--- is uniquely identified by the following set of moment conditions
\begin{align*}
	E_{\mathbf{P}}[g(X,\theta_{0})] = 0
\end{align*}
where $g : \mathbb{R} \times \mathbb{R}^{q} \rightarrow \mathbb{R}^{d}$ is known to the researcher.

The main feature of this setup is that it allows $d\equiv d(n)$ to grow with the sample size. In many cases this departure from the standard theory is of relevance. For example, in many models the identifying condition is given by a conditional moment restriction, $E_{\mathbf{P}}[\rho(Y,\theta_{0})|W]$ --- where $\rho$ maps into $\mathbb{R}^{J}$ with $J$ fixed --- and the researcher converts it to a series of unconditional moment restrictions $E_{\mathbf{P}}[\rho(Y,\theta_{0}) \otimes q^{K(n)}(W)]$ where $q^{K(n)}(w) = (q_{1}(w),...,q_{K(n)}(w))$ are basis functions such as Fourier series, P-splines, etc; this is the case considered in DIN (see also \cite{Djong-Bierens-ET94} and references therein). For this case $x = (y,w)$, $d(n) = J K(n)$ and $g(x,\theta) = \rho(y,\theta) \otimes q^{K(n)}(w)$. 

An alternative motivation to consider increasing $d$ would be cases where although the number of moments is fixed, it could be large relative to the sample size and thus treating it as a diverging sequence could deliver more accurate asymptotics. As pointed out by \cite{K-M_JOE93} one example of this could be the panel data model  in \cite{AB-RESTUD91} where $x = (y_{1},...,y_{T})$ and the components of the vector $g(x,\theta)$ are given by $((y_{t}-y_{t-1}) - \theta (y_{t-1}-y_{t-2})) y_{t-s} $ for $s=1,...,t-1$ and $t=3,...,T$. Here, for a panel of length $T$, the number of instruments/moments is given by $d=(T-2)(T-1)/2$. \footnote{For instance for $T=4$, $d=d(n)=3$ and for $T=5$, $d=d(n)=6$. In cases where  $d(n) = o(n^{1/4})$, these values imply that, roughly speaking, the number of observations should be larger than 82 and 1300, resp. It is also worth to point out that in case where $T$ is large, one can simply include fewer lags $y_{t-s}$ in $((y_{t}-y_{t-1}) - \theta (y_{t-1}-y_{t-2})) y_{t-s} $ and thus reduce $d$.}

The next assumptions impose some regularity conditions on $g$. These restrictions are standard in the literature and can be somewhat relaxed (e.g. see  \cite{DIN_JOE03} and references therein).

\begin{assumption}\label{ass:Omega}
	 $\Omega \equiv E_{\mathbf{P}}[g(X,\theta_{0})g(X,\theta_{0})^{T}]$ exists with $C^{-1} \leq \lambda_{l}(\Omega) \leq C$ for all $l = 1,...,d$ for some $C\geq 1$.
\end{assumption}

For instance, for the case where $g = \rho \otimes q^{K(n)} $ (for simplicity, let $J=1$) it suffices to assume that $E_{\mathbf{P}}[\rho(Y,\theta_{0})^{2} | W ]$ and that the eigenvalues of $E_{\mathbf{P}}[q^{K(n)}(W,\theta_{0})q^{K(n)}(W,\theta_{0})^{T}]$ are both bounded bounded and bounded away from zero a.s.-$\mathbf{P}$ 
%and that the same holds for the eigenvalues of $E_{\mathbf{P}}[q^{K(n)}(W,\theta_{0})q^{K(n)}(W,\theta_{0})^{T}]$. 
These assumptions are standard; see \cite{DIN_JOE03} for a discussion. \footnote{These assumptions are also standard in the context of series-based estimators; see \cite{Chen20075549}.}

Let $\mathcal{N}$ be an open neighborhood of $\theta_{0}$.

\begin{assumption}\label{ass:example-1}
	For all $n$: (i) $E_{\mathbf{P}}\left[ \sup_{\theta \in \mathcal{N}}  ||g(X,\theta)||^{2 (2+\gamma)}_{2}   \right] \precsim d(n)^{2+\gamma}$ for some $\gamma \geq 0$; (ii) $\theta \mapsto g(X,\theta)$ is continuously differentiable a.s.-$\mathbf{P}$; (iii) $E_{\mathbf{P}}[\sup_{\theta \in \mathcal{N}}||\nabla_{\theta} g(X,\theta)||^{2\beta}_{2} ] \precsim d(n)^{\beta}$ for some $\beta \geq 1$; (iv) there exists a measurable $x \mapsto \delta_{n}(x)$ such that $||\nabla_{\theta} g(X,\theta) - \nabla_{\theta} g(X,\theta_{0})||_{2} \precsim \delta_{n}(X) ||\theta - \theta_{0}||_{2}$ for all $\theta \in \mathcal{N}$ a.s.-$\mathbf{P}$, and $E_{\mathbf{P}}[\delta_{n}(X)^{2}] \precsim d(n)$.\footnote{The notation $\nabla_{\theta}g(x,\theta)$ means the gradient with respect to $\theta$ of the function $g$; it is a $q \times d$ matrix. For any matrix, $A$, $||A||_{2}$ is defined as the operator norm.}
\end{assumption}

For instance, for the case $g = \rho \otimes q^{K}$ for many basis functions such as splines and Fourier series it holds that $\sup_{w} ||q^{K}(w)||_{2} \precsim \sqrt{K}$.\footnote{Other series like power series typically present $\sup_{w}||q^{K}(w)||_{2} \precsim K$, or more generally one can think of $\sup_{w} ||q^{K}(w)||_{2} \precsim \zeta(K)$ for some function $\zeta$. These cases can be accommodated in our theory, at the expense of further restricting the rate of growth of $d(n)$. } Thus, the previous assumption holds provided that $E[\sup_{\theta \in \mathcal{N}} ||\rho(Y,\theta)||_{2}^{2(2+\gamma)}|W]$ and $E_{\mathbf{P}}[\sup_{\theta \in \mathcal{N}}||\nabla_{\theta} \rho(Y,\theta)||^{2\beta}_{2} | W]$ are bounded by a constant $C$, and $||\nabla_{\theta} \rho(Y,\theta) - \nabla_{\theta} \rho(Y,\theta_{0})||_{2} \precsim \delta(Y) ||\theta - \theta_{0}||_{2}$ with $E_{\mathbf{P}}[\delta(Y)^{2}|W] \leq C$, a.s.-$\mathbf{P}$, for some $C>0$,.\footnote{These restrictions are analogous to Assumptions 4-6 in \cite{DIN_JOE03}.}

The GMM estimator is given by $\hat{\theta}_{GMM,n} = \arg\min_{\theta \in \Theta} \hat{Q}_{GMM,n}(\theta) $ where \begin{align*}
 \hat{Q}_{GMM,n}(\theta) \equiv	n^{-1} \sum_{i=1}^{n} g(X_{i},\theta)^{T} \hat{W}_{n} n^{-1} \sum_{i=1}^{n} g(X_{i},\theta) 
\end{align*} with $\hat{W}_{n} \in \mathbb{R}^{d \times d}$ is a (possibly random) positive definite matrix. The following mild condition is required 

\begin{assumption}\label{ass:example-W}
	There exists a $W \in \mathbb{R}^{d(n) \times d(n)}$ positive definite and a $C\geq 1$ such that $|| \hat{W}_{n} - W ||_{2} =  o_{\mathbf{P}}(d(n)^{-1/2})$ and $C^{-1} \leq \lambda_{l}(W) \leq C$ for all $l = 1,...,d(n)$ and $n \in \mathbb{N}$.
\end{assumption}

The bootstrap analog is given by $\hat{\theta}^{\ast}_{GMM,n} = \arg\min_{\theta \in \Theta} \hat{Q}^{\ast}_{GMM,n}(\theta)$ where 
\begin{align*}
\hat{Q}^{\ast}_{GMM,n}(\theta) = n^{-1} \sum_{i=1}^{n} \omega_{i,n}g(X_{i},\theta)^{T} \hat{W}_{n} n^{-1} \sum_{i=1}^{n} \omega_{i,n}g(X_{i},\theta)   .
\end{align*}

These formulas give raise to the following MST statistic: $\hat{T}_{GMM,n} \equiv n \hat{Q}_{GMM,n}(\hat{\theta}_{GMM,n})$ and its bootstrap version $\hat{T}^{\ast}_{GMM,n} \equiv n \hat{Q}^{\ast}_{GMM,n}(\hat{\theta}^{\ast}_{GMM,n})$. 

In order to simplify the exposition we directly impose that $(\omega_{i,n})_{i \leq n}$ satisfy Assumption \ref{ass:boot-w} and also that they are uniformly bounded; this last assumption is not necessary for the results but imposing it greatly simplifies the technical derivations in our proofs.

It is worth to point out that DIN only considers GMM estimators with $W = \Omega^{-1}$ because they rely on CLT-type approximations for inference (e.g., see their Theorem 6.3). Since our result allow us to focus on bootstrap-based inference, the weighting matrix $W$ does not need to coincide with $\Omega^{-1}$; in fact it can simply be chosen as $\hat{W} = W = I$. 
That is, our results provide valid asymptotic inference for MST statistics for a larger class of GMM estimator, one with $W \ne \Omega^{-1}$. %For instance, DIN (p. 64) study an instrumental variable regression estimator, but restrict themselves to homoscedastic errors; using our methods we can relax this restriction.

The GEL estimator is given by
\begin{align*}
	&\hat{\theta}_{GEL,n} = \arg\min_{\theta \in \Theta} \hat{Q}_{GEL,n}(\theta),\\
	& where~\hat{Q}_{GEL,n}(\theta) \equiv \sup_{\lambda \in \Lambda(\theta)} \sum_{i=1}^{n} s(\lambda^{T}g(X_{i},\theta)) 
\end{align*}
where $s : \mathcal{V} \subseteq \mathbb{R} \mapsto \mathbb{R}$ is concave and twice continuously differentiable with Lipschitz second derivative, $\mathcal{V}$ includes a neighborhood of 0, and $\Lambda(\theta) \equiv \{ \lambda \in \mathbb{R}^{d} :  \lambda^{T}g(X,\theta) \in \mathcal{V},~a.s.-\mathbf{P}   \}$. The function $s$ can be chosen to encompass several estimators of interest such as empirical likelihood ($s(\cdot) = \ln(1-\cdot)$), exponential tilting ($s(\cdot) = -\exp(\cdot)$; \cite{ISJ_ECMA98} and \cite{KS_ECMA97}) and continuously updating GMM ($s(\cdot) = -0.5(1 + \cdot)^{2}$; \cite{HHY_JBES96}). Henceforth, to simplify the presentation we assume the following normalization $s'(0)=s''(0)=-1$.
%where $s : \mathbb{V} \subseteq \mathbb{R} \mapsto \mathbb{R}$ is concave and twice continuously differentiable with Lipschitz second derivative, $\mathbb{V}$ includes a neighborhood of 0, and $\Lambda(\theta) \equiv \{ \lambda \in \mathbb{R}^{d} :  \lambda^{T}g(X,\theta) \in \mathbb{V},~a.s.-\mathbf{P}   \}$. The function $s$ can be chosen to encompass several estimators of interest such as empirical likelihood ($s(\cdot) = \ln(1-\cdot)$), exponential tilting ($s(\cdot) = -\exp(\cdot)$) and continuously updating GMM ($s(\cdot) = -0.5(1 + \cdot)^{2}$); see \cite{DIN_JOE03} for a more thorough discussion. Henceforth, to simplify the presentation we assume the following normalization $s'(0)=s''(0)=-1$.

Analogously to GMM, we have the following MST statistic for GEL: $\hat{T}_{GEL,n} \equiv 2 \left\{ \hat{Q}_{GEL,n}(\hat{\theta}_{GEL,n}) - ns(0) \right\}$ and its bootstrap version $\hat{T}^{\ast}_{GEL,n} \equiv 2 \left\{ \hat{Q}^{\ast}_{GEL,n}(\hat{\theta}^{\ast}_{GEL,n}) - n s(0) \right\}$, where $\hat{\theta}^{\ast}_{GEL,n} =  \arg\min_{\theta \in \Theta} \hat{Q}^{\ast}_{GEL,n}(\theta)$ and $\hat{Q}^{\ast}_{GEL,n}$ is defined as $\hat{Q}_{GEL,n}$ but with $\omega_{i,n} g(x_{i},\cdot)$ instead of $g(x_{i},\cdot)$.\footnote{Abusing notation we still denote $\Lambda(\theta)$ as the set for the bootstrap case.}

The next assumption is a high level condition. Part (i) ensures existence of a minimizer for $\lambda$ and part (ii) imposes convergence rates on the GMM and GEL estimators. Because our main goal is to establish the asymptotic behavior of the MST statistics, we directly impose this assumption to ease the exposition.

\begin{assumption}\label{ass:lambda-GEL}
	(i) $\hat{\lambda}^{\ast}_{n} = \arg \max_{\lambda \in \Lambda(\hat{\theta}^{\ast}_{GEL,n})} \sum_{i=1}^{n} s(\lambda^{T}\omega_{i,n} g(X_{i},\hat{\theta}^{\ast}_{GEL,n}))$ exists wpa1-$\mathbf{P}$ and $||\hat{\lambda}^{\ast}_{n}||_{2} = O_{\mathbf{P}^{\ast}_{n}(\cdot|Z^{n})}(\sqrt{d(n)/n})$, wpa1-$\mathbf{P}$; (ii) $\hat{\theta}^{\ast}_{j,n} = \theta_{0} + O_{\mathbf{P}^{\ast}_{n}(\cdot|Z^{n})}(n^{-1/2})$ wpa1-$\mathbf{P}$ and $\hat{\theta}_{j,n} = \theta_{0} + O_{\mathbf{P}}(n^{-1/2})$ for $j \in \{GEL,GMM\}$.
\end{assumption}

The derivation of both parts of this assumption from more primitive conditions can be obtained from the results in DIN and references therein; in particular in Lemma A.10 and Theorems 5.4 and 5.6.

The following lemma establishes that the test statistics for both estimators are asymptotically equivalent to a quadratic form on sample averages of $g$. 

\begin{lemma}\label{lem:quad-approx}
	Suppose Assumptions \ref{ass:Omega}, \ref{ass:example-1}, \ref{ass:example-W} and \ref{ass:lambda-GEL} hold. Also, suppose that $\frac{ d(n)^{\max\{2+4/\gamma,4\}}}{n}=o(1)$. Then
		\begin{align*}
		\hat{T}_{GMM,n} = & \left ( \frac{1}{\sqrt{n}} \sum_{i=1}^{n} g(X_{i},\theta_{0})   \right)^{T} W \left ( \frac{1}{\sqrt{n}}  \sum_{i=1}^{n} g(X_{i},\theta_{0})   \right) + o_{\mathbf{P}}(\sqrt{d(n)})\\
		\hat{T}_{GEL,n} = & \left ( \frac{1}{\sqrt{n}}  \sum_{i=1}^{n} g(X_{i},\theta_{0})   \right)^{T} \Omega^{-1} \left (\frac{1}{\sqrt{n}}  \sum_{i=1}^{n} g(X_{i},\theta_{0})   \right) + o_{\mathbf{P}}(\sqrt{d(n)})
		\end{align*}
%	\begin{align*}
%		&\hat{T}_{GMM,n} = \left \Vert n^{-1/2} \sum_{i=1}^{n} g(X_{i},\theta_{0})   \right \Vert^{2}_{W} + o_{\mathbf{P}}(\sqrt{d(n)})\\
%		&~and~\hat{T}_{GEL,n} = \left \Vert n^{-1/2} \sum_{i=1}^{n} g(X_{i},\theta_{0})   \right \Vert^{2}_{\Omega^{-1}} + o_{\mathbf{P}}(\sqrt{d(n)})
%	\end{align*}
and
%		\begin{align*}
%		&\hat{T}^{\ast}_{GMM,n} = \left \Vert \frac{1}{\sqrt{n}}  \sum_{i=1}^{n} \omega_{i,n} g(X_{i},\theta_{0})   \right \Vert^{2}_{W} + o_{\mathbf{P}^{\ast}_{n}(\cdot|Z^{n})}(\sqrt{d(n)})\\
%		&~and~\hat{T}^{\ast}_{GEL,n} = \left \Vert \frac{1}{\sqrt{n}}  \sum_{i=1}^{n} \omega_{i,n} g(X_{i},\theta_{0})   \right \Vert^{2}_{\Omega^{-1}} + o_{\mathbf{P}^{\ast}_{n}(\cdot|Z^{n})}(\sqrt{d(n)})
%		\end{align*}
		\begin{align*}
		\hat{T}^{\ast}_{GMM,n} =& \left( \frac{1}{\sqrt{n}}  \sum_{i=1}^{n} \omega_{i,n} g(X_{i},\theta_{0})   \right)^{T} W\left( \frac{1}{\sqrt{n}}  \sum_{i=1}^{n} \omega_{i,n} g(X_{i},\theta_{0})   \right) \\
		& + o_{\mathbf{P}^{\ast}_{n}(\cdot|Z^{n})}(\sqrt{d(n)}),\\
		\hat{T}^{\ast}_{GEL,n} = & \left( \frac{1}{\sqrt{n}}  \sum_{i=1}^{n} \omega_{i,n} g(X_{i},\theta_{0})   \right)^{T}\Omega^{-1} \left( \frac{1}{\sqrt{n}}  \sum_{i=1}^{n} \omega_{i,n} g(X_{i},\theta_{0})   \right) \\
		& + o_{\mathbf{P}^{\ast}_{n}(\cdot|Z^{n})}(\sqrt{d(n)})
		\end{align*}
		wpa1-$\mathbf{P}$.
\end{lemma}

\begin{proof}
	See Appendix \ref{app:appl}.
\end{proof}

This Theorem establishes that the test statistics, asymptotically, behave as quadratic forms of (properly scaled) sample averages. Thus, our result in Theorem \ref{thm:boot} can be applied to these cases with $Z_{i} \equiv g(X_{i},\theta_{0}) W^{1/2}$ or $Z_{i} \equiv g(X_{i},\theta_{0}) \Omega^{-1/2}$ . The next Theorem formalizes this claim in this particular setting.

%XXX MOVER TO DISCUSSION SECTION: It is worth to point out that, although this is a particular example, the previous lemma illustrates a general feature present in several test statistics, namely that they asymptotically behave as quadratic forms of (properly scaled) sample averages. Thus, our result in theorem \ref{thm:boot} can be applied to these cases. The next theorem formalizes this claim in this particular setting. XXX

%The next theorem follows from the previous lemma and theorem \ref{thm:boot}.

\begin{theorem}\label{thm:MST}
	Suppose Assumptions \ref{ass:Omega}, \ref{ass:example-1}, \ref{ass:example-W} and \ref{ass:lambda-GEL} hold. Also, suppose that $\frac{ d(n)^{\max\{2+4/\gamma,4\}}}{n}=o(1)$. Then
	\begin{align*}
	\sup_{t \in \mathbb{R}} \left| \mathbf{P}^{\ast}_{n} \left( \frac{\hat{T}^{\ast}_{GMM,n}}{\sqrt{d(n)}} \geq t \mid Z^{n}  \right)   -  \mathbf{P} \left( \frac{\hat{T}_{GMM,n}}{\sqrt{d(n)}}\geq t   \right)     \right| = o_{\mathbf{P}}(1),
	\end{align*}
	and \begin{align*}
		\sup_{t \in \mathbb{R}} \left| \mathbf{P}^{\ast}_{n} \left( \frac{\hat{T}^{\ast}_{GEL,n}}{\sqrt{d(n)}} \geq t \mid Z^{n}  \right)   -  \mathbf{P} \left( \frac{\hat{T}_{GEL,n}}{\sqrt{d(n)}}\geq t   \right)     \right| = o_{\mathbf{P}}(1).
	\end{align*}
\end{theorem}

\begin{proof}
	See Appendix \ref{app:appl}.
\end{proof}

 This result allow us to compute bootstrap-based p-values for the MST statistics for the general classes of GMM and GEL estimators, even when the number of moment restrictions increases with the sample size (but not too fast). In particular,  for $\gamma \geq 2$, our condition on rate imposes that $d(n)^{4}/n = o(1)$ which is the one required in Theorem 6.4 in DIN, but at the cost of imposing restrictions on some higher moments of $||g(\cdot,\theta_{0})||_{2}$ (see Assumption \ref{ass:example-1}(i)).

		  \section{Numerical Simulations}
		  \label{sec:simul}
		  
		  In this section we present a Monte Carlo (MC) study to assess the finite sample behavior of our procedure. We perform $5000$ MC repetitions and in each draw we perform $5000$ bootstrap repetitions.
		 		  
%		  The design is as follows: In each MC repetition we draw  $Z_{i} \sim \sqrt{12} U_{i}$ with $U_{i} \sim U(-0.5,0.5)$ for $i=1,...,n$, and we vary the number of observations. Let
%		  \begin{align*}
%		  	Q_{n} = n \left( n^{-1} \sum_{i=1}^{n} Z_{i}  \right)^{T} V \left( n^{-1} \sum_{i=1}^{n} Z_{i}  \right)
%		  \end{align*} 
%		  where $V$ is specified below, and the associated bootstrapped version is given by 
%		  \begin{align*}
%		  	Q^{\ast}_{n} = n \left( n^{-1} \sum_{i=1}^{n} \omega_{i} Z_{i}  \right)^{T} V \left( n^{-1} \sum_{i=1}^{n} \omega_{i} Z_{i}  \right).
%		  \end{align*} 
%		  Throughout the study we use $\omega \sim N(0,1)$.
		  
		  	  The design is as follows: In each MC repetition we draw  $Z_{i} = V^{1/2} \sqrt{12} U_{i}$ with $U_{i} \sim U(-0.5,0.5)$ for $i=1,...,n$, and $V$ is a positive definite symmetric matrix specified below. Let
		  	  \begin{align*}
		  	  Q_{n} = n \left( n^{-1} \sum_{i=1}^{n} Z_{i}  \right)^{T} \left( n^{-1} \sum_{i=1}^{n} Z_{i}  \right)
		  	  \end{align*} 
		  	  and the associated bootstrapped version is given by 
		  	  \begin{align*}
		  	  Q^{\ast}_{n} = n \left( n^{-1} \sum_{i=1}^{n} \omega_{i} Z_{i}  \right)^{T} \left( n^{-1} \sum_{i=1}^{n} \omega_{i} Z_{i}  \right).
		  	  \end{align*} 
		  	  Throughout the study we use $\omega \sim N(0,1)$.

		  We are interested in studying $\mathbb{K}^{B}_{n} = \sup_{a \in \mathbb{A}} \left|  \mathbf{P}_{n} \left(  Q_{n}   \geq t^{B}_{n}(a,Z^{n} )  \right) - (1-a) \right|$		   
%		  \begin{align*}
%		  	\mathbb{K}^{B}_{n} = \sup_{a \in \mathbb{A}} \left|  \mathbf{P}_{n} \left(  Q_{n}   \geq t^{B}_{n}(a,Z^{n} )  \right) - (1-a) \right| 
%		  \end{align*}
		  and, for comparison, $\mathbb{K}_{n} = \sup_{a \in \mathbb{A}} \left|\mathbf{P}_{n} \left(  Q_{n}   \geq t_{n}(a)  \right) - (1-a) \right|$,
%		  and 
%		  \begin{align*}
%		  	\mathbb{K}_{n} = \sup_{a \in \mathbb{A}} \left|\mathbf{P}_{n} \left(  Q_{n}   \geq t_{n}(a)  \right) - (1-a) \right| 
%		  \end{align*}
		  where $t^{B}_{n}(a,Z^{n} ) $ is the $a$-th empirical percentile of $Q^{\ast}_{n}$ and $t_{n}(a)$ is the $a$-th percentile of a chi-square with degrees of freedom $d(n)$.\footnote{In both cases, we approximate $\mathbf{P}_{n}$ using the empirical cdf across MC repetitions.} \footnote{$\mathbb{K}^{B}_{n}$ is in fact the quantity of interest since, by construction, $1-a$ coincides with the empirical quantile of $Q^{\ast}_{n}$, thus $\mathbb{K}^{B}_{n}$ approximates $\sup_{a \in \mathbb{A}} \left|  \mathbf{P}_{n} \left(  Q_{n}   \geq t^{B}_{n}(a,Z^{n} )  \right) - \mathbf{P}^{\ast}_{n}(  Q^{\ast}_{n} \geq   t^{B}_{n}(a,Z^{n} )  \mid Z^{n}  ) \right| $. A similar observation holds for $\mathbb{K}_{n}$.  } The set $\mathbb{A}$ is given by $\{ 0.900,0.950,0.975,0.990  \}$. The typical application for our results is testing --- like in the Section \ref{sec:MST} ---, and with this in mind $\mathbb{A}$ is designed to capture the relevant values of $a$ for which we would like to  assess the performance of the approximation.
		  
		  \bigskip 
		  
		  \textbf{Approximation Error.} Figure \ref{fig:KB_K} shows the $\log (\mathbb{K}_{n}/\mathbb{K}^{B}_{n})$ for different values of the weighting matrix $V$ and for $n=500$ and $d(n) = 3$. When $V=I$ both, the chi-squared-based and boostrap-based procedures yield correct approximations of the limiting distribution, and thus the value is close to one. As expected, for cases where $V = (1+\epsilon/\sqrt{n}) I$ with $\epsilon \ne 0$, the chi-squared-based approximation does not approximate the limiting distribution, whereas the bootstrap-based continues to do so. The simulations shows that even for small values of $\epsilon/\sqrt{n}$, the difference is non-negligible. We note also that the deviations from $V=I$ we consider are \textquotedblleft mild" and we expect that for more complex deviations the results will be even more stark.
		  
		  \begin{figure}[h!]
		  	\centering
		  	\includegraphics[height=2.5in,width=\textwidth]{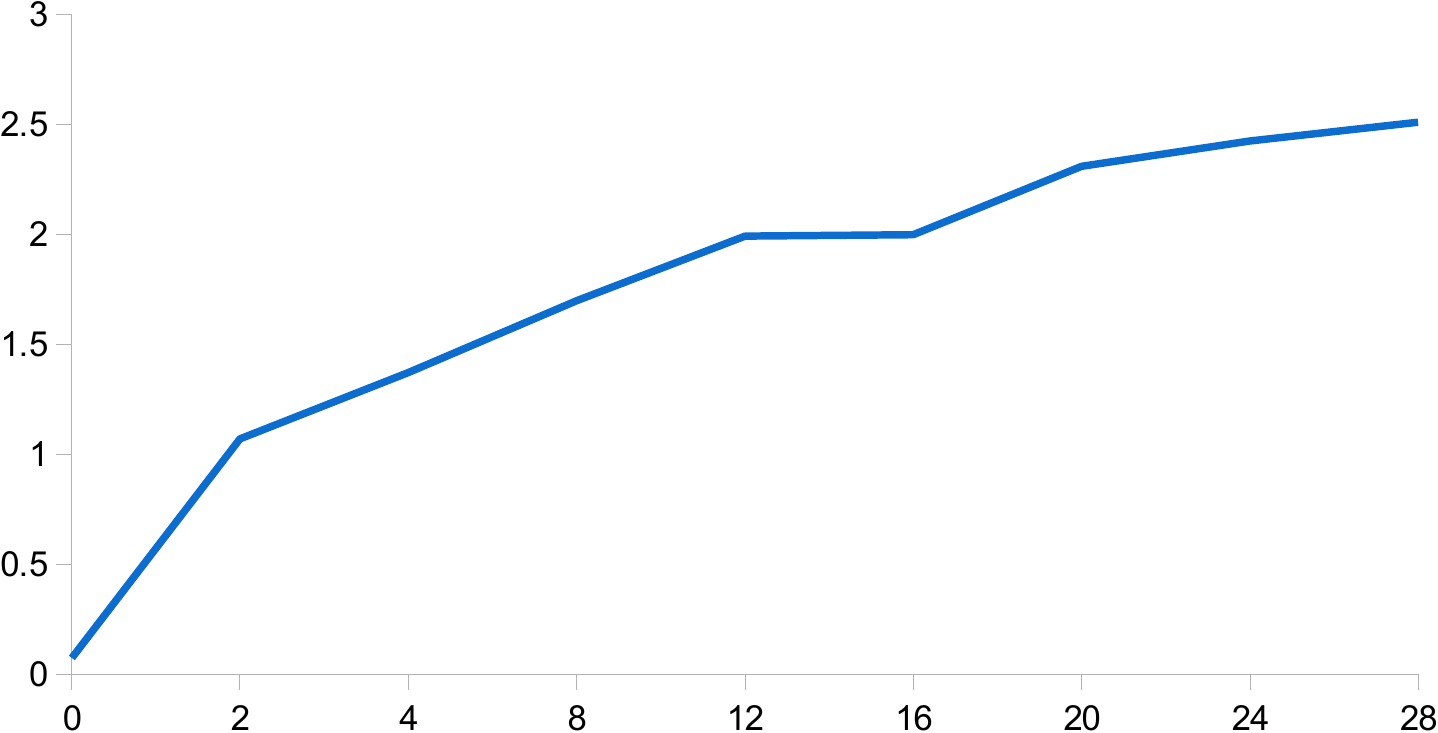}
		  		\caption{Plot of $\log (\mathbb{K}_{n}/\mathbb{K}^{B}_{n})$ for different values of $V=(1+\epsilon/\sqrt{n}) I$ for $\epsilon \in \{0,2,...,28\}$. }
		  			\label{fig:KB_K}
		  \end{figure}
		  
		  Table \ref{tab:approx} shows the value of $100 \times \left|  \mathbf{P}_{n} \left(  Q_{n}   \geq t^{B}_{n}(a,Z^{n} )  \right) - (1-a) \right|$ for each $a \in \mathbb{A}$. We can see that regardless of the value of $V$, the approximation error of our bootstrap procedure remains stable at low values, below 0.5\%. 
		  
		  \begin{table}[h!]
		  	\centering
		  	\begin{tabular}{c|ccccccccc}\hline \hline 
		  		   $\epsilon$         & $0$ & $4 $ & $8$ & $12$ & $16$ & $20$ & $24$ & $28 $ \\ \hline 
		  		   $a = 0.900$       &   0.06  & 0.50 & 0.40 & 0.30 & 0.20 & 0.02 & 0.20 & 0.10 \\
		  		     $a = 0.950$     &  0.06   & 0.30 & 0.22 & 0.02 & 0.04 &  0.02 & 0.28 & 0.14 \\
		  		     $a = 0.975$     &  0.30   & 0.20 & 0.08 & 0.02 & 0.14 & 0.26 & 0.06 & 0.06 \\
		  		       $a = 0.990 $  &  0.24   & 0.30 & 0.04 & 0.04 & 0.06 & 0.18 & 0.02 & 0.02 \\ \hline \hline 
		  	\end{tabular}
		  	\caption{$100 \times \left|  \mathbf{P}_{n} \left(  Q_{n}   \geq t^{B}_{n}(a,Z^{n} )  \right) - (1-a) \right|$  for all $a \in \mathbb{A}$ and different values of $\epsilon$; for $n=500$ and $d(n)=3$. }
		  	\label{tab:approx}
		  \end{table}

			  \begin{table}[h!]
			  	\centering	  	
			  	\begin{tabular}{l|rrrrr} \hline\hline 
			  		  & $n=250$ & $n=500$ & $n=1000$ & $n=2000$ & $n=3000$  \\ \hline 
			  		{\footnotesize{$\mathbb{K}^{B}_{n}/\mathbb{K}_{n}$ for $d(n) = n^{1/5}$}}       &   0.937       &  0.950   &  0.750  &  0.892 &  1.030  \\ 	   \hline\hline 
			  	\end{tabular}
			   	\caption{$\mathbb{K}^{B}_{n}/\mathbb{K}_{n}$ for $d(n) = n^{1/5}$ and different values of $n$.}
			   		\label{tab:convrate-simul}
			 
			  \end{table}	  
		  
%
%		  \begin{table}[h!]
%		  	\centering	  	
%		  	\begin{tabular}{l|c} \hline\hline 
%		  		$n  \setminus d(n)$    &   {\footnotesize{$\mathbb{K}^{B}_{n}/\mathbb{K}_{n}$ for $d(n) = n^{1/5}$}}    \\ \hline 
%		  		$250$     &   0.937      \\
%		  		$500$     &  0.950   \\
%		  		$1000$    &  0.750  \\
%		  		$2000$    &  0.892 \\
%		  		%	  	  $2500$    &  0.923  &    0.234   & 0.319       &   0.920    &  0.820      & 1.005 \\ 
%		  		$3000$    &  1.030  \\ 	   \hline\hline 
%		  	\end{tabular}
%\caption{$\mathbb{K}^{B}_{n}/\mathbb{K}_{n}$ for $d(n) = n^{1/5}$ and different values of $n$.}
%\label{tab:convrate-simul}
%		  \end{table}
		  
		  Table \ref{tab:convrate-simul} shows $\mathbb{K}^{B}_{n}/\mathbb{K}_{n}$ for $d(n) = n^{1/5}$. We see that for all $n$ under consideration the ratio is around one, and in almost all below one. These results suggest that, at least for the current design, the convergence rate of the bootstrap-based approximation is no worse than the one for the chi-squared-based.

		  \bigskip 
		  
		  \textbf{Robustness to $d(n)$ and choice of weights.}  We now assess how robust our procedure is to the choice of $d(n)$. Recall that, for this specification, our theory predicts that is sufficient to have $d(n) = o(n^{1/4})$; for values higher than this our theory is silent about the validity of our bootstrap procedure. We are thus particularly interested on the performance of our procedure for the latter set of values. In this exercise, we set $V=I$ and consider different values of $n$ and $d(n)$. 
		  
		  Table \ref{tab:d-simul} columns 2-5 shows the value of $100 \times \mathbb{K}^{B}_{n}$ for different choices of $d(n)$ and $n$. For values of $n$ less than 1000, the procedure seems to be quite robust to larger choices of $d(n)$ in the range of $n^{1/4}$ to $n^{1/2}$, but not higher. For values of $n$ around 2000-3000, however, our procedure seems to deteriorate for values of $d(n)$ larger than  $n^{1/2}$.

		  \begin{table}[h!]
		  	\centering	  	
		  \begin{tabular}{l||r|r|r|r|r} \hline\hline 
 $n  \setminus d(n)$     &$ n^{1/5} $ & $ n^{1/4} $ & $n^{1/3} $ & $ n^{1/2} $ & $n^{3/4} $  \\ \hline 
              $250$          &   0.300       &    0.440         &    0.332       & 
              0.440     &   2.440      \\
		  	  $500$       &    0.350   & 0.401       &    0.280        &  0.450      &    2.780      \\
		  	  $1000$     &    0.340   & 0.240       &   0.540    &  0.500      & 2.080 \\
		  	  $2000$      &    0.400   & 0.250       &   0.340    &  0.690      & 0.943 \\
	%	  	  $2500$    &  0.923  &    0.234   & 0.319       &   0.920    &  0.820      & 1.005 \\ 
		   	  $3000$     &    0.200   & 0.201       &   0.341    &  0.601      & 1.463 \\ 	   \hline\hline 
		  \end{tabular}
		  \caption{The value of $100 \times \mathbb{K}^{B}_{n}$ for different values of $(n,d(n))$.}
		  	\label{tab:d-simul}
		\end{table}

		 		  We now assess the robustness of our procedure to different choices of weights. We compare the Gaussian weights with two other weights: $\omega_{i} \sim U(-0.5,0.5)$ and $\omega_{i} \sim t-Student(3)$ (properly scaled to have unit variance). These choices are designed to study how different tail behavior of the weight's distribution affect the performance of our bootstrap procedure.
		 		  
		 		  In order to ease the computational burden we lower the bootstrap repetitions to 2000 each. Table \ref{tab:w-simul} presents the results. The overall pattern seems to suggest that the Gaussian and Uniform weights have comparable performances, and perform better than the t-Student weights. This pattern illustrates the discussion in Section \ref{sec:discussion} regarding desirable properties of weights.   
		 		  
		 		  		  \begin{table}[h!]
		 		  		  	\centering	  	
		 		  		  	\begin{tabular}{l|ccc} \hline\hline 
		 		  		  		  $n \setminus$ Weights  &  Gaussian & Uniform & t-Student  \\ \hline 
		 		  		  		  $250$       &  0.560    &  0.440   &   0.960        \\
		 		  		  		  $500$       & 0.500   & 0.370 & 0.980 \\
		 		  		  		  $1000$     &  0.139 &  0.319 & 0.400 \\
		 		  		  		  $2000$    &  0.240  &  0.340 &  0.660 \\
		 		  		  		  $3000$     &  0.180 & 0.200 & 0.400   \\	 \hline\hline 
		 		  		  	\end{tabular}
		 		  		  	\caption{The value of $100 \times \mathbb{K}^{B}_{n}$ for Gaussian, Uniform and t-Student weights with $d(n) = n^{1/5}$ .}
		 		  		  		\label{tab:w-simul}
		 		  		  \end{table}

		  \bigskip
		  
		  \textbf{Remarks.} Overall, the simulations suggest that our procedure has a finite sample performance that is at least as good as, and in some cases better than, the \textquotedblleft standard" chi-squared approach. Weights with \textquotedblleft thin tails" such as Uniform and Gaussian seem to perform better than weights with heavier tails. Additionally, as also discussed in the context of our application in Section \ref{sec:MST}, our bootstrap-based approximation can be applied in situations that go beyond those covered by the chi-square approach.

\section{Proof of Theorem \ref{thm:boot}}
\label{sec:proof-boot}

Recall that $x \in \mathbb{R}^{d(n)} \mapsto ||x||^{2}_{2} \equiv x^{T}x$ and that $\mathcal{C}_{M}$ is the class of functions $f : \mathbb{R} \rightarrow \mathbb{R}$ that are three times continuously differentiable and $ \sup_{x} | \partial^{r} f(x)| \leq (M)^{r}$. 

All the proofs of the lemmas in this section are relegated to Appendix \ref{app:lemmas-main}.

For any two probability measures $Q$ and $P$, let \begin{align}\label{eqn:WN}
	\Delta_{M}(P,Q) \equiv \sup_{f \in \mathcal{C}_{M}} | E_{P}[f(||X||^{2}_{2} )] - E_{Q}[f(||Y||^{2}_{2})]|.
\end{align}

\begin{remark}
	Throughout the text we use this definition for $X = n^{-1/2}\sum_{i=1}^{n} X_{i}$ and    $Y = n^{-1/2} \sum_{i=1}^{n} Y_{i}$, with $(X_{1},...,X_{n}) \sim P$ and $(Y_{1},...,Y_{n})  \sim Q$. For these cases, we abuse notation and use $\Delta_{M}(P,Q)$ to denote\begin{align*}
		\sup_{f \in \mathcal{C}_{M}} | E_{P}[f(||n^{-1/2}\sum_{i=1}^{n} X_{i} ||^{2}_{2} )] - E_{Q}[f(||n^{-1/2}\sum_{i=1}^{n} Y_{i}||^{2}_{2})]|.
	\end{align*} 
	
	Also, in the cases where $X_{i} \sim i.i.d.-P$, we abuse notation and still use $\Delta_{M}(P,Q)$ to denote the same quantity.
\end{remark}

% Since $t \sqrt{d(n)} + d(n) \in \mathbb{R}$ for all $t \in \mathbb{R}$, 
We want to establish the following: For any $\varepsilon'>0$, there exists a $N(\varepsilon')$ such that 
\begin{align*}
	\mathbf{P}_{n} \left(  \sup_{t \in \mathbb{R}} \left| \mathbf{P}^{\ast}_{n} \left( ||\sqrt{n} \mathbb{Z}^{\ast}_{n}||^{2}_{2} \geq t \mid Z^{n}  \right)   -  \mathbf{P}_{n} \left( ||\sqrt{n}\mathbb{Z}_{n}||^{2}_{2}  \geq t   \right)     \right|   \geq \varepsilon'  \right) < \varepsilon'
\end{align*}
for all $n \geq N(\varepsilon')$. Observe that \begin{align*}
&	\mathbf{P}_{n} \left(  \sup_{t \in \mathbb{R}} \left| \mathbf{P}^{\ast}_{n} \left( ||\sqrt{n} \mathbb{Z}^{\ast}_{n}||^{2}_{2} \geq t \mid Z^{n}  \right)   -  \mathbf{P}_{n} \left( ||\sqrt{n}\mathbb{Z}_{n}||^{2}_{2}  \geq t   \right)     \right|   \geq \varepsilon'  \right) \\
	\leq & 	\mathbf{P}_{n} \left(  \{ \sup_{t \in \mathbb{R}} \left| \mathbf{P}^{\ast}_{n} \left( ||\sqrt{n} \mathbb{Z}^{\ast}_{n}||^{2}_{2} \geq t \mid Z^{n}  \right)   -  \mathbf{P}_{n} \left( ||\sqrt{n}\mathbb{Z}_{n}||^{2}_{2}  \geq t   \right)     \right| \geq \varepsilon' \} \cap S_{n}     \right)\\
	& +	\mathbf{P}_{n} \left(  S_{n}^{C} \right)
\end{align*}
where $S_{n} \equiv \{  Z^{n}   : n^{-1} \sum_{i=1}^{n} ||Z_{i}||^{2}_{2}  \leq  (0.5 \varepsilon')^{-1} tr \{ \Sigma_{n}  \}  \}$. By the Markov inequality $\mathbf{P}_{n} \left(  S_{n}^{C} \right) \leq 0.5 \varepsilon'$. 
%\begin{align*}
%	\mathbf{P}_{n} \left(  S_{n}^{C} \right) \leq & 0.5 \varepsilon' (tr\{ \Sigma_{n} \})^{-1} E_{\mathbf{P}_{n}} [ n^{-1} \sum_{i=1}^{n} ||Z_{i}||^{2}_{2} ] \\
%	= & 0.5 \varepsilon' (tr\{ \Sigma_{n} \})^{-1}  tr \{ E_{\mathbf{P}_{n}} [ Z_{1} Z_{1}^{T} ]  \} = 0.5 \varepsilon'. 
%\end{align*}
Thus, it suffices to show that \begin{align}\label{eqn:boot-1}
	\mathbf{P}_{n} \left(  \left\{ \sup_{t \in \mathbb{R}} \left| \mathbf{P}^{\ast}_{n} \left( ||\sqrt{n} \mathbb{Z}^{\ast}_{n}||^{2}_{2} \geq t \mid Z^{n}  \right)   -  \mathbf{P}_{n} \left( ||\sqrt{n}\mathbb{Z}_{n}||^{2}_{2}  \geq t   \right)     \right| \geq \varepsilon' \right\} \cap S_{n}     \right) < 0.5 \varepsilon'.
\end{align}

By the triangle inequality, for all $t \in \mathbb{R}$ and $Z^{n}$
\begin{align*}
	& | E_{\mathbf{P}^{\ast}_{n}} \left[ 1\{  ||\sqrt{n}\mathbb{Z}^{\ast}_{n}||^{2}_{2}   \geq t \} \mid Z^{n}  \right]  -  E_{\mathbf{P}_{n}} \left[ 1\{  ||\sqrt{n}\mathbb{Z}_{n}||^{2}_{2}  \geq t   \}\right]   |\\
	\leq & | E_{\mathbf{P}^{\ast}_{n}} \left[ 1\{  ||\sqrt{n}\mathbb{Z}^{\ast}_{n}||^{2}_{2}   \geq t \} \mid Z^{n}  \right]  -  E_{\boldsymbol{\Phi}_{n}} \left[ 1\{  ||\sqrt{n}\mathbb{V}_{n}||^{2}_{2}   \geq t   \}  \right]   | \\
	& + | E_{\mathbf{P}_{n}} \left[ 1\{  ||\sqrt{n} \mathbb{Z}_{n}||^{2}_{2}   \geq t \}   \right]  -  E_{\boldsymbol{\Phi}_{n}} \left[ 1\{  ||\sqrt{n}\mathbb{V}_{n}||^{2}_{2}  \geq t   \} \right]   | 
\end{align*}
where $\sqrt{n} \mathbb{V}_{n} \sim N(0,\Sigma_{n})$. We use $\boldsymbol{\Phi}_{n}$ to denote this probability.
%where $\mathbb{V}_{n} \equiv n^{-1} \sum_{i=1}^{n} V_{i,n}  $ with $V_{i,n} \sim i.i.d.-N(0,\Sigma_{n})$. We use $\boldsymbol{\Phi}_{n}$ to denote the probability of $V_{i,n}$ for all $i=1,...,n$.

Therefore, in order to obtain display \ref{eqn:boot-1}, it suffices to bound\begin{align}\notag
	& \mathbf{P}_{n} \left(  \{ \sup_{t \in \mathbb{R}} \left| E_{\mathbf{P}^{\ast}_{n}} \left[  1 \{  ||\sqrt{n} \mathbb{Z}^{\ast}_{n}||^{2}_{2} \geq t \}  \mid Z^{n}  \right]   -  E_{\boldsymbol{\Phi}_{n}} \left[ 1 \{  ||\sqrt{n}\mathbb{V}_{n}||^{2}_{2}  \geq t  \} \right]    \right| \geq 0.5 \varepsilon' \} \cap S_{n}     \right)\\ \label{eqn:boot-2a}
	&  < 0.25 \varepsilon'
\end{align}
and 
\begin{align}\label{eqn:boot-2b}
 \lim_{n \rightarrow \infty }   \sup_{t \in \mathbb{R}} \left| E_{\mathbf{P}_{n}} \left[  1 \{  ||\sqrt{n} \mathbb{Z}_{n}||^{2}_{2} \geq t \}  \right]   -  E_{\boldsymbol{\Phi}_{n}} \left[ 1 \{  ||\sqrt{n}\mathbb{V}_{n}||^{2}_{2}  \geq t  \} \right]    \right| =0.
\end{align}

The next two lemmas allow us to  \textquotedblleft replace" the indicator functions by  \textquotedblleft smooth" functions. 

\begin{lemma}\label{lem:Delta-P-Pr} %[{lem:Delta-P-Pr}]
	Suppose Assumption \ref{ass:data-Z}(i) holds. For any $\varepsilon>0$, there exists a $\gamma(\varepsilon)$ and $N(\varepsilon)$ such that for all $n \geq N(\varepsilon)$ and all $h \leq h(\varepsilon,\sqrt{tr\{ \Sigma^{2}_{n} \}} \gamma(\varepsilon))$
	\begin{align}
	& \sup_{t \in \mathbb{R}} \left| E_{\mathbf{P}_{n}} \left[ 1\{ ||\sqrt{n} \mathbb{Z}_{n}||^{2}_{2} \geq t \}  \right] - E_{\boldsymbol{\Phi}_{n}} \left[ 1\{ ||\sqrt{n}\mathbb{V}_{n}||^{2}_{2} \geq t \}   \right] \right| \\
	\leq &   \frac{\varepsilon}{1-\varepsilon} + 3 \varepsilon + \Delta_{h^{-1}}(\mathbf{P}_{n},\boldsymbol{\Phi}_{n}).
	\end{align}
\end{lemma}
\noindent (Recall that, $\Delta_{h^{-1}}(\mathbf{P}_{n},\boldsymbol{\Phi}_{n}) = \sup_{f \in \mathcal{C}_{h^{-1}}} \left|  E_{\mathbf{P}_{n}} \left[ f \left( ||\sqrt{n} \mathbb{Z}_{n}||^{2}_{2}  \right)  \right] - E_{\boldsymbol{\Phi}_{n}} \left[ f \left(||\sqrt{n} \mathbb{V}_{n}||^{2}_{2} \right)   \right]  \right|$). And

\begin{lemma}\label{lem:Delta-Pwz-Pr} %[{lem:Delta-Pwz-Pr}]
	Suppose Assumption \ref{ass:data-Z}(i) holds. For any $\varepsilon>0$, there exists a $\gamma(\varepsilon)$ and $N(\varepsilon)$ such that for all $n \geq N(\varepsilon)$ and all $h \leq h(\varepsilon,\sqrt{tr\{ \Sigma^{2}_{n} \}} \gamma(\varepsilon))$
	\begin{align} \notag
	&  \sup_{t \in \mathbb{R}} \left| E_{\mathbf{P}^{\ast}_{n}} \left[ 1\{ ||\sqrt{n} \mathbb{Z}^{\ast}_{n}||^{2}_{2} \geq t \}  |Z^{n} \right] - E_{Pr} \left[ 1\{ ||\sqrt{n}\mathbb{V}_{n}||^{2}_{2} \geq t \}   \right] \right|  \\ \label{eqn:Delta-Pwz-Pr-1}
	\leq &  \frac{\varepsilon}{1-\varepsilon} + 3 \varepsilon + \Delta_{h^{-1}}(\mathbf{P}^{\ast}_{n}(\cdot|Z^{n}),\boldsymbol{\Phi}_{n}),
	\end{align}
	for any $Z^{n} \in \mathbb{R}^{d(n)}$.
\end{lemma}
%	where, recall that,
%	\begin{align}
%	\Delta_{h^{-1}}(\mathbf{P}^{\ast}_{n}(\cdot|Z^{n}),\boldsymbol{\Phi}_{n}) = \sup_{f \in \mathcal{C}_{h^{-1}}} \left|  E_{\mathbf{P}^{\ast}_{n}} \left[ f \left( ||\sqrt{n} \mathbb{Z}^{\ast}_{n} ||^{2}_{2} \right)  |Z^{n} \right] - E_{\boldsymbol{\Phi}_{n}} \left[ f \left(||\sqrt{n} \mathbb{V}_{n}||^{2}_{2} \right) \right]  \right|
%	\end{align}
	
	\begin{remark}
	The previous lemma holds for any $h$ provided that is below $h \leq h(\varepsilon,\sqrt{tr\{ \Sigma^{2}_{n} \}} \gamma(\varepsilon) )$. The intuition from this restriction is as follows: $h$ and $\delta_{n} \equiv \sqrt{tr\{ \Sigma^{2}_{n} \}} \gamma(\varepsilon)$ index the  \textquotedblleft smooth" function we use to approximate $x \mapsto 1\{ ||x||^{2}_{2} \geq t \}$; see Lemma \ref{lem:Zn-strong-weak} in the Appendix for a precise expression. It turns out that $h$ has to be  \textquotedblleft small" relative to $\delta_{n}$. Therefore, we need the bound $h(\varepsilon,\delta_{n} )$.
		
		It is worth to note that, for the  \textquotedblleft smooth" function to be a good approximation of $1\{ ||\cdot||^{2}_{2} \geq t \}$, we need $\delta_{n}$ to be \textquotedblleft small'' (see the proof of Lemma \ref{lem:Delta-Pwz-Pr} in the Appendix).   What we mean by $\delta_{n}$ to be \textquotedblleft small'' depends on how $||\sqrt{n}\mathbb{V}_{n}||^{2}_{2}$ concentrates mass. Lemma \ref{lem:anticoncentration-V} establishes an anti-concentration result, wherein we obtain that this random variable puts very little mass in any given interval. Therefore $\delta_{n}$ could actually be quite large, of the order of $\sqrt{tr\{ \Sigma^{2}_{n} \}}$.
	\end{remark}

	Therefore, by letting $\varepsilon$ in the lemmas be such that $\frac{\varepsilon}{1-\varepsilon} + 3 \varepsilon  = 0.25 \varepsilon'$ we obtain
	\begin{align} \label{eqn:boot-3}
	& \sup_{t \in \mathbb{R}} \left| E_{\mathbf{P}_{n}} \left[  1 \{  ||\sqrt{n} \mathbb{Z}_{n}||^{2}_{2} \geq t \}  \right]   -  E_{\boldsymbol{\Phi}_{n}} \left[ 1 \{  ||\sqrt{n}\mathbb{V}_{n}||^{2}_{2}  \geq  t  \} \right]    \right| \leq  0.25\varepsilon'  + \Delta_{h^{-1}}(\mathbf{P}_{n},\boldsymbol{\Phi}_{n})  
	\end{align}
	and
	\begin{align}\notag
	& \mathbf{P}_{n} \left(  \{ \sup_{t \in \mathbb{R}} \left| E_{\mathbf{P}^{\ast}_{n}} \left[  1 \{  ||\sqrt{n} \mathbb{Z}^{\ast}_{n}||^{2}_{2} \geq t \}  \mid Z^{n}  \right]   -  E_{\boldsymbol{\Phi}_{n}} \left[ 1 \{  ||\sqrt{n}\mathbb{V}_{n}||^{2}_{2}  \geq t  \} \right]    \right| \geq 0.5 \varepsilon' \} \cap S_{n}     \right) \\ \label{eqn:boot-4}
	& \leq 		\mathbf{P}_{n} \left(  \{ \Delta_{h^{-1}}(\mathbf{P}^{\ast}_{n}(\cdot|Z^{n}),\boldsymbol{\Phi}_{n}) \geq 0.25 \varepsilon' \} \cap S_{n}     \right) 
	\end{align}
	for all $n \geq N(\varepsilon)$ and all $h \leq h(\varepsilon,\delta_{n})$ (note that $\varepsilon$ is a function of $\varepsilon'$).

		  By the triangle inequality and straightforward algebra, it follows that\begin{align*}
		  & \mathbf{P}_{n} \left(  \{ \Delta_{h^{-1}}(\mathbf{P}^{\ast}_{n}(\cdot|Z^{n}),\boldsymbol{\Phi}_{n}) \geq 0.25 \varepsilon' \} \cap S_{n}     \right)  \\
		  & \leq  \mathbf{P}_{n} \left(  \{ \Delta_{h^{-1}}(\mathbf{P}^{\ast}_{n}(\cdot|Z^{n}),\boldsymbol{\Phi}^{\ast}_{n}(\cdot|Z^{n} )) \geq \frac{1}{8} \varepsilon' \} \cap S_{n}     \right) \\
		  &~  + \mathbf{P}_{n} \left(  \{ \Delta_{h^{-1}}( \boldsymbol{\Phi}^{\ast}_{n}(\cdot|Z^{n} ), \boldsymbol{\Phi}_{n} ) \geq \frac{1}{8} \varepsilon' \} \cap S_{n}     \right)   
		  \end{align*}
		  where $\boldsymbol{\Phi}^{\ast}_{n}(\cdot|Z^{n})$ denotes the conditional probability (given the original data $Z^{n}$ )  $\sqrt{n} \mathbb{U}_{n} \sim N(0,n^{-1} \sum_{i=1}^{n} Z_{i} Z^{T}_{i})$
%		  associated to $\mathbb{U}_{n} \equiv n^{-1} \sum_{i=1}^{n} U_{i,n}  $ with $U_{i,n} \sim i.i.d.-N(0,Z_{i}Z_{i}^{T})$. 

		  Hence, by the previous display and Equations \ref{eqn:boot-1}, \ref{eqn:boot-2a}-\ref{eqn:boot-2b}, \ref{eqn:boot-3} and \ref{eqn:boot-4}, in order to show the desired result it suffices to show that: For all $\varepsilon'$, there exists a $N(\varepsilon')$ such that	  
		  \begin{align} \label{eqn:final-bdd-1}
		  &	\mathbf{P}_{n} \left(  \{ \Delta_{h^{-1}}(\mathbf{P}^{\ast}_{n}(\cdot|Z^{n}),\boldsymbol{\Phi}^{\ast}_{n}(\cdot|Z^{n} ) )\geq  \varepsilon' \} \cap S_{n}     \right)  < \varepsilon',\\ \label{eqn:final-bdd-2}
		  &	\mathbf{P}_{n} \left(  \{ \Delta_{h^{-1}}( \boldsymbol{\Phi}^{\ast}_{n}(\cdot|Z^{n} ), \boldsymbol{\Phi}_{n} ) \geq  \varepsilon' \} \cap S_{n}     \right)   < \varepsilon',\\ \label{eqn:final-bdd-3}
		  &~and~\Delta_{h^{-1}}(\mathbf{P}_{n},\boldsymbol{\Phi}_{n}) < \varepsilon' 
		  \end{align}
		  for all $n \geq N(\varepsilon')$ and some $h \leq h(\varepsilon,\sqrt{tr\{ \Sigma^{2}_{n} \}} \gamma(\varepsilon) )$. Theorems \ref{thm:weak-norm-boot} and \ref{thm:weak-norm-ori} establish expressions \ref{eqn:final-bdd-1} and \ref{eqn:final-bdd-3}. 
		  
		  \begin{remark}
		  	From Lemma \ref{lem:Vn-strong-weak}, $h(\varepsilon,\sqrt{tr\{ \Sigma^{2}_{n} \}} \gamma(\varepsilon) ) = \sqrt{tr\{ \Sigma^{2}_{n} \}} \gamma(\varepsilon)  / \Phi^{-1}(\varepsilon) $ and thus $h$ can be taken to be proportional (up to a constant that depends on $\varepsilon$)  to $\sqrt{tr\{ \Sigma_{n}^{2}  \}}$. Hence, under Assumption \ref{ass:data-Z}(i),  $h$ can be taken to be such that $h^{-2} \precsim d(n)^{-1}  $. Therefore, Theorems \ref{thm:weak-norm-boot} and \ref{thm:weak-norm-ori} actually imply a stronger result: $\Delta_{h^{-1}}(\mathbf{P}_{n},\boldsymbol{\Phi}_{n})  = o(d(n)^{-1})$ and $\Delta_{h^{-1}}(\mathbf{P}^{\ast}_{n}(\cdot|Z^{n}),\boldsymbol{\Phi}^{\ast}_{n}(\cdot|Z^{n} ) ) = o_{\mathbf{P}_{n}} (d(n)^{-1})$.
		  \end{remark}

		  We have thus reduced the original problem to a Gaussian approximation problem. That is, it remains to show that 
		  \begin{align}\label{eqn:boot-5}
		  \mathbf{P}_{n} \left( \{ \Delta_{h^{-1}}(\boldsymbol{\Phi}^{\ast}_{n}(\cdot|Z^{n}),\boldsymbol{\Phi}_{n}) \geq \varepsilon' \} \cap S_{n}     \right) < \varepsilon'.
		  \end{align} 
		  
		  Since $\sqrt{n} \mathbb{U}_{n} \sim N(0,\hat{\Sigma}_{n})$ (with $\hat{\Sigma}_{n} = n^{-1} \sum_{i=1}^{n} Z_{i}Z_{i}^{T}$) and $\sqrt{n} \mathbb{V}_{n} \sim N(0,\Sigma_{n})$, the previous display is equivalent to showing that 
		  \begin{align*}
		  \mathbf{P}_{n} \left( \{ \Delta_{h^{-1}}(N(0,\hat{\Sigma}_{n}),N(0,\Sigma_{n})) \geq \varepsilon' \} \cap S_{n}     \right) < \varepsilon'.
		  \end{align*} 
		  
		  Essentially, this expression follows by the fact that $\hat{\Sigma}_{n}$ converges in probability to $\Sigma_{n}$ in a suitable norm. The following lemma formalizes this.
		  %which is analogous to Lemma XX in \cite{Vic} and relies in the Slepian interpolation.
		  
		  \begin{lemma}\label{lem:slep} %[{lem:slep}]
		  	For any $h >0$ and any $n \in \mathbb{N}$
		  	\begin{align*}
		  	\Delta_{h^{-1}}(\boldsymbol{\Phi}^{\ast}_{n}(\cdot| Z^{n} ),\boldsymbol{\Phi}_{n})  \precsim & \max_{j,l} \left| \left\{ n^{-1}
		  	\sum_{i=1}^{n} Z_{[j],i}
		  	Z_{[l],i} - \Sigma_{[l,j]} \right\} \right| \\
		  	& \times  h^{-1} d(n) \left( h^{-1}  tr\{ \Sigma_{n}  \}  + h^{-1} tr\{ \hat{\Sigma}_{n}  \}  + 2   \right).
		  	\end{align*}
		  \end{lemma}

		  Observe that for any $ Z^{n} \in S_{n} = \{  Z^{n}   : n^{-1} \sum_{i=1}^{n} ||Z_{i}||^{2}_{2}  \leq  (0.5 \varepsilon')^{-1} tr \{ \Sigma_{n}  \} \}$, the RHS of the expression in the Lemma is bounded above by\begin{align*}
		  	\frac{d(n)}{h}  \{ \frac{tr\{ \Sigma_{n} \} }{h \varepsilon' } + 2 \}.
		  \end{align*}
		  
		  Thus by Lemma \ref{lem:slep}, in order to establish the desired result, it suffices to show that 
		  \begin{align}\label{eqn:gauss-1}
		  \mathbf{P}_{n} \left( \max_{j,l} \left| n^{-1} \sum_{i=1}^{n} Z_{[l],i}Z_{[j],i} - \Sigma_{[j,l]}  \right|  \geq \frac{ (\varepsilon')^{2}}{ d(n) h^{-2} tr\{ \Sigma_{n} \}  } \cap S_{n}     \right) < \varepsilon'  
		  \end{align}
		  for sufficiently large $n$. Henceforth, let $c_{n} \equiv \frac{ (\varepsilon')^{2}}{ d(n) h^{-2} tr\{ \Sigma_{n} \}  } $ and let $\mathbf{A}_{i,n}[j,l] \equiv Z_{[j],i}
		  Z_{[l],i} $, observe that
		  \begin{align*}
		  E_{\mathbf{P}_{n}}[\mathbf{A}_{i,n}[j,l]] = E_{\mathbf{P}_{n}}[Z_{[j],i}
		  Z_{[l],i}] = \Sigma_{[j,l],n}.
		  \end{align*}
		 
		  Let $\mathbf{A}_{i,n}[j,l] = \mathbf{A}^{L}_{i,n}[j,l] + \mathbf{A}^{U}_{i,n}[j,l] \equiv \mathbf{A}_{i,n}[j,l] 1\{ |\mathbf{A}_{i,n}[j,l]| \leq e_{n} \} + \mathbf{A}_{i,n}[j,l] 1\{ |\mathbf{A}_{i,n}[j,l]| \geq e_{n}   \}$ where $(e_{n})_{n}$ with $e_{n} > 0$ is defined below. Clearly, $\mathbf{A}^{L}_{i,n}[j,l] \leq e_{n}$. So,  by Hoeffding inequality (see \cite{Boucheron-book13} p. 34)
		  %Let $\mathbf{A}_{i,n}[j,l] = \mathbf{A}^{L}_{i,n}[j,l] + \mathbf{A}^{U}_{i,n}[j,l] \equiv \mathbf{A}_{i,n}[j,l] 1\{ ||Z_{i}||^{2}_{1} \leq e_{n} \} + \mathbf{A}_{i,n}[j,l] 1\{ %||Z_{i}||^{2}_{1} \geq e_{n}   \}$ where $(e_{n})_{n}$ with $e_{n} > 0$ is defined below. Since $|Z_{[l],i}Z_{[j],i}| \leq \sum_{l=1}^{d(n)} |Z_{[l],i}| \sum_{j=1}^{d(n)}| %Z_{[j],i}| $, clearly, $\mathbf{A}^{L}_{i,n}[j,l] \leq e_{n}$. So,
		  \begin{align*}
		  & \mathbf{P}_{n} \left( \max_{j,l} |n^{-1}\sum_{i=1}^{n} \{\mathbf{A}^{L}_{i,n}[j,l] - E_{\mathbf{P}_{n}}[\mathbf{A}^{L}_{i,n}[j,l]] \} | \geq c_{n}    \right) \\
		  \leq & \sum_{j,l} \mathbf{P}_{n} \left(  |n^{-1}\sum_{i=1}^{n} \{\mathbf{A}^{L}_{i,n}[j,l] - E_{\mathbf{P}_{n}}[\mathbf{A}^{L}_{i,n}[j,l]] \} | \geq c_{n}    \right)\\
		  \precsim & \exp\left\{ 2\log(d(n)) -  n \frac{c^{2}_{n}}{e^{2}_{n}}    \right\}.
		  \end{align*} 
		 Therefore, by setting $e_{n} = c_{n} \sqrt{\frac{n  0.25}{\log(d(n))}} $, the previous display implies that\begin{align*}
		  	\mathbf{P}_{n} \left( \max_{j,l}  |n^{-1}\sum_{i=1}^{n} \{\mathbf{A}^{L}_{i,n}[j,l] - E_{\mathbf{P}_{n}}[\mathbf{A}^{L}_{i,n}[j,l]] \} | \geq \varepsilon'     \right) \leq \varepsilon',
		  \end{align*}
		  for sufficiently large $n$.

		  Second, by the Markov inequality and the fact that 
		  \begin{align}
		  E_{\mathbf{P}_{n}} \left[ \left( \{\mathbf{A}^{U}_{i,n}[j,l] - E_{\mathbf{P}_{n}}[\mathbf{A}^{U}_{i,n}[j,l]] \} \right) \left( \{\mathbf{A}^{U}_{k,n}[j,l] - E_{\mathbf{P}_{n}}[\mathbf{A}^{U}_{k,n}[j,l]] \} \right)    \right] = 0
		  \end{align} 
		  for all $i\ne k$, it follows that
		  \begin{align*}
		  & \mathbf{P}_{n} \left( \max_{j,l}  |n^{-1}\sum_{i=1}^{n} \{\mathbf{A}^{U}_{i,n}[j,l] - E_{\mathbf{P}_{n}}[\mathbf{A}^{U}_{i,n}[j,l]] \} | \geq c_{n}    \right) \\
		  \leq & \sum_{j,l} (c_{n})^{-2} E_{\mathbf{P}_{n}} \left[ \left( n^{-1}\sum_{i=1}^{n} \{\mathbf{A}^{U}_{i,n}[j,l] - E_{\mathbf{P}_{n}}[\mathbf{A}^{U}_{i,n}[j,l]] \} \right)^{2}      \right]\\
		  = &  (c_{n})^{-2} n^{-1} \sum_{j,l}  E_{\mathbf{P}_{n}} \left[ \left( \{\mathbf{A}^{U}_{1,n}[j,l] - E_{\mathbf{P}_{n}}[\mathbf{A}^{U}_{1,n}[j,l]] \} \right)^{2}      \right] \\
		  \leq & (c_{n})^{-2} n^{-1} \sum_{j,l}  E_{\mathbf{P}_{n}} \left[ \left( \mathbf{A}^{U}_{1,n}[j,l]\right)^{2}      \right].
		  \end{align*}
		  %\begin{align}
		  %& \mathbf{P}_{n} \left( \max_{j,l}  |n^{-1}\sum_{i=1}^{n} \{\mathbf{A}^{U}_{i,n}[j,l] - E_{\mathbf{P}_{n}}[\mathbf{A}^{U}_{i,n}[j,l]] \} | \geq c_{n}    \right) \\
		  %\leq & \sum_{j,l} (c_{n})^{-2} E_{\mathbf{P}_{n}} \left[ \left( n^{-1}\sum_{i=1}^{n} \{\mathbf{A}^{U}_{i,n}[j,l] - E_{\mathbf{P}_{n}}[\mathbf{A}^{U}_{i,n}[j,l]] \} \right)^{2}      \right]\\
		  %= &  (c_{n})^{-2} n^{-1} \sum_{j,l}  E_{\mathbf{P}_{n}} \left[ \left( \{\mathbf{A}^{U}_{1,n}[j,l] - E_{\mathbf{P}_{n}}[\mathbf{A}^{U}_{1,n}[j,l]] \} \right)^{2}      \right] \\
		  %\leq & (c_{n})^{-2} n^{-1} \sum_{j,l}  E_{\mathbf{P}_{n}} \left[ \left( \mathbf{A}^{U}_{1,n}[j,l]\right)^{2}      \right]\\
		  %= & (c_{n})^{-2} n^{-1}   E_{\mathbf{P}_{n}} \left[ \left( \sum_{j=1}^{d(n)} |Z_{[j],1}|^{2} \right)^{2} 1\{ ||Z_{1}||^{2}_{1} \geq e_{n} \} \right].
		  %\end{align}
		  Therefore by the Markov inequality, for $p  > 0$
		  \begin{align*}
		  &  \mathbf{P}_{n} \left( \max_{j,l} |n^{-1}\sum_{i=1}^{n} \{\mathbf{A}^{U}_{i,n}[j,l] - E_{\mathbf{P}_{n}}[\mathbf{A}^{U}_{i,n}[j,l]] \} | \geq c_{n}     \right) \\
		  \leq & \frac{1}{c^{2}_{n} n (e_{n})^{p}} \sum_{j=1}^{d(n)}  \sum_{l=1}^{d(n)} E_{\mathbf{P}_{n}} \left[ \left( Z_{[j],1}
		  Z_{[l],1}  \right)^{2+p}      \right] \\
	%	  = & \frac{1}{c^{2}_{n} n (e_{n})^{p}}  E_{\mathbf{P}_{n}} \left[ \sum_{j=1}^{d(n)} ( Z_{[j],1})^{2+p}
	%	  \sum_{l=1}^{d(n)} (Z_{[l],1} )^{2+p}      \right] \\
		  = & \frac{1}{c^{2}_{n} n (e_{n})^{p}}  E_{\mathbf{P}_{n}} \left[ \left( \sum_{j=1}^{d(n)} ( Z_{[j],1})^{2+p} \right)^{2}  \right].
		  \end{align*}
		  Since $e_{n} =  c_{n} \sqrt{\frac{n 0.25}{\log(d(n))}} $ and $c_{n} \equiv \frac{ (\varepsilon')^{2}}{ d(n) h^{-2} tr\{ \Sigma_{n} \}  } $, it follows that 
		  \begin{align*}
		  & \mathbf{P}_{n} \left( \max_{j,l} |n^{-1}\sum_{i=1}^{n} \{\mathbf{A}^{U}_{i,n}[j,l] - E_{\mathbf{P}_{n}}[\mathbf{A}^{U}_{i,n}[j,l]] \} | \geq c_{n}     \right) \\
		  \precsim & \frac{\log(d(n))^{p/2}}{c^{2+p}_{n} n^{1+p/2} }  E_{\mathbf{P}_{n}} \left[ \left( \sum_{j=1}^{d(n)} ( Z_{[j],1})^{2+p} \right)^{2}  \right] \\
		  \precsim &\frac{(\log(d(n)))^{p/2} d(n)^{2+p} (tr\{ \Sigma_{n} \})^{2+p}   }{h^{4+2p} n^{1+p/2} }  E_{\mathbf{P}_{n}} \left[ \left( \sum_{j=1}^{d(n)} ( Z_{[j],1})^{2+p} \right)^{2}  \right].
		  \end{align*}
		  
		  Since we can set $h \asymp \sqrt{tr\{ \Sigma^{2}_{n}   \}}$, the RHS becomes\\ $\frac{(\log(d(n)))^{p/2} d(n)^{2+p} }{n^{1+p/2} } \left( \frac{ tr\{ \Sigma_{n} \} }{ tr\{ \Sigma_{n}^{2} \} } \right)^{2+p} E_{\mathbf{P}_{n}} \left[ \left( \sum_{j=1}^{d(n)} ( Z_{[j],1})^{2+p} \right)^{2}  \right]$. By choosing $p = \kappa$, by Assumptions \ref{ass:data-Z}(i) and \ref{ass:data-Z}(iii), the term vanishes as $n \rightarrow \infty$.

		  Therefore, Equation \ref{eqn:gauss-1} is established and with that the proof of Theorem \ref{thm:boot}.

		  \section{Discussion} 
		  \label{sec:discussion}
		  
		  \textbf{Applicability of our Results.} The example developed in  Section \ref{sec:MST} illustrates a general feature present in several test statistics, namely that they  behave asymptotically as quadratic forms of (properly scaled) sample averages. These are the main motivational examples to which we can apply our result in Theorem \ref{thm:boot}.

		  This remark is best illustrated in the Wald statistic case.  To formalize this, consider i.i.d. data $(X_{1},...,X_{n})$ drawn from $\mathbf{P} $ and parameter a $k \geq d=d(n)$ dimensional $\theta_{\mathbf{P}}$ and a \textquotedblleft smooth" function $\theta \mapsto c(\theta) \in \mathbb{R}^{d}$ which represent the hypothesis we want to test; i.e., the  null hypothesis is $c(\theta_{\mathbf{P}}) = 0$. \footnote{The notation $\theta_{\mathbf{P}}$ stresses that the parameter is a (known) function of the probability distribution. Thus, an estimator can obtained by \textquotedblleft plugging in" the empirical distribution $P_{n}$.   } The fact that  $d$ grows with the sample size is of potential interest because in certain situations one could have that the dimension of the parameter, although fixed,  is not \textquotedblleft small" relative to $n$. Also, in some other situations, one could have a more explicit model of increasing dimensionality like in the cases discussed in Section \ref{sec:MST} or in series or sieves estimators; see, for example, \cite{ChenPouzo14}.

		  Suppose there exists an estimator $\theta_{P_{n}}$ ($P_{n}$ is the empirical distribution), then the Wald statistic is given by
		  \begin{align*}
		  \mathbb{W}_{n}(P_{n},\mathbf{P})  = n \left( c(\theta_{P_{n}}) - c(\theta_{\mathbf{P}}) \right)^{T}  V_{n} \left( c(\theta_{P_{n}}) - c(\theta_{\mathbf{P}}) \right)
		  \end{align*}
		  where $V_{n} \in \mathbb{R}^{d \times d}$ is some (possibly random) matrix to be determined later.
		  
		     Suppose $\theta_{P_{n}}$ admits an asymptotic linear representation (ALR) of the form \footnote{See \cite{VdV00} and references therein for a discussion regarding ALR and sufficient conditions for it. Here we follow \cite{VdV-Murphy00}. }
		     %\begin{align}\label{eqn:W-LAR}
		  %   	c(\theta_{P_{n}}) = c(\theta_{\mathbf{P}}) + E_{P_{n}} [ \psi(Z,\theta_{\mathbf{P}}) ] + o_{\mathbf{P}}(a_{n})
		  %   \end{align}
		  \begin{align}\label{eqn:W-LAR}
		  \left \Vert \sqrt{n}  (c(\theta_{P_{n}}) - c(\theta_{\mathbf{P}}) ) - \sqrt{n} E_{P_{n}} [ \psi(X,\theta_{\mathbf{P}}) ] \right \Vert _{2} = o_{\mathbf{P}}(1 + \sqrt{n}  || c(\theta_{P_{n}}) - c(\theta_{\mathbf{P}}) ||_{2} )
		  \end{align}
		  with $E_{\mathbf{P}}[\psi(X,\theta_{0})] = 0$ and finite second moment. \footnote{Note that $E_{P_{n}}[ \psi(X,\theta_{\mathbf{P}}) ] = n^{-1} \sum_{i=1}^{n} \psi(X_{i},\theta_{\mathbf{P}}) $.}.  
		  
		  The bootstrap analog of the Wald statistic are of the \textquotedblleft plug-in" type, i.e., \footnote{In principle, one could also \textquotedblleft replace" $V_{n}$ --- which typically is a function of $P_{n}$, $V_{P_{n}}$ --- by $V_{P^{\ast}_{n}}$. Our results could be extended to this case too.} 
		  \begin{align*}
		  \mathbb{W}_{n}(P_{n}^{\ast}, P_{n} ) = n \left( c(\theta_{P_{n}^{\ast}}) - c(\theta_{P_{n}}) \right)^{T}  V_{n} \left( c(\theta_{P_{n}^{\ast}}) - c(\theta_{P_{n}}) \right)
		  \end{align*}
		  where $P^{\ast}_{n}$ is given by $n^{-1} \sum_{i=1}^{n} \omega_{i,n} \delta_{X_{i}}$  (the dependence of $P_{n}^{\ast}$ on $X^{n}$ is omitted to ease the notational burden). The bootstrap ALR (B-ALR) is given by
		  %   \begin{align}\label{eqn:W-BLAR}
		  %   c(\theta_{P^{\ast}_{n}}) = c(\theta_{P_{n}}) + n^{-1} \sum_{i=1}^{n} \omega_{i,n} \psi(Z_{i},\theta_{\mathbf{P}}) + o_{P^{\ast}_{n}}(a_{n}),~wpa1-\mathbf{P}.
		  %   \end{align}
		  \begin{align}\notag
		  & \left \Vert \sqrt{n} (c(\theta_{P^{\ast}_{n}}) - c(\theta_{P_{n}})) - n^{-1/2} \sum_{i=1}^{n} \omega_{i,n} \psi(X_{i},\theta_{\mathbf{P}}) \right \Vert_{2} \\ \label{eqn:W-BLAR}
		  & = o_{\mathbf{P}^{\ast}_{n}}(1 + \sqrt{n} || c(\theta_{P^{\ast}_{n}}) - c(\theta_{P_{n}}) ||_{2} ),~wpa1-\mathbf{P}.
		  \end{align}
		  
		  Given the asymptotic linear representations, we can show that the Wald and Bootstrapped Wald statistics can be represented asymptotically as quadratic forms, and thus fall in the framework studied in this paper. The following proposition formalizes such representation, and thereby allow us to apply our Theorem \ref{thm:boot} with $Z_{i} \equiv \psi(X_{i},\theta_{\mathbf{P}}) V^{1/2}$ to approximate the limiting distribution of $\mathbb{W}_{n}(P_{n},\mathbf{P})$.

		  \begin{proposition}\label{pro:Wald}
		  	  Let $V$ be a matrix such that there exists a $c \geq 1$ such that $c^{-1} \leq \lambda_{l}(V) \leq c$ for all $l=1,...,d$ and $||V_{n} - V ||_{2} = o_{\mathbf{P}}(1)$ . Then,  under the null hypothesis, ALR and B-ALR yield
		  	\begin{align*}
		  	\mathbb{W}_{n}(P_{n},\mathbf{P}) (1 + o_{\mathbf{P}}( 1))  = &  \left( n^{-1/2} \sum_{i=1}^{n} \psi(X_{i},\theta_{\mathbf{P}}) \right)^{T}  V \left( n^{-1/2} \sum_{i=1}^{n} \psi(X_{i},\theta_{\mathbf{P}}) \right) \\
		  	& + o_{\mathbf{P}}( \sqrt{d(n)}  ),
		  	\end{align*}
		  	and 
		  	\begin{align*}
		  	\mathbb{W}_{n}(P_{n}^{\ast}, P_{n}) (1 + o_{\mathbf{P}_{n}^{\ast}}( 1) ) = &  \left( n^{-1/2} \sum_{i=1}^{n} \omega_{i,n} \psi(X_{i},\theta_{\mathbf{P}}) \right)^{T}  V \left( n^{-1/2} \sum_{i=1}^{n} \omega_{i,n} \psi(X_{i},\theta_{\mathbf{P}}) \right) \\
		  	& + o_{\mathbf{P}_{n}^{\ast}}(  \sqrt{d(n)}  )~wpa1-\mathbf{P}.
		  	\end{align*}
		  	%   	\begin{align*}
		  	%   	\mathbb{W}_{n}(P_{n},\mathbf{P}) = &  \left( n^{-1/2} \sum_{i=1}^{n} \psi(Z_{i},\theta_{\mathbf{P}}) \right)^{T}  V_{n} \left( n^{-1/2} \sum_{i=1}^{n} \psi(Z_{i},\theta_{\mathbf{P}}) \right) \\
		  	%   	 & + o_{\mathbf{P}}( n ||a_{n}||^{2}_{2} + n^{-1/2} ||a_{n}||_{2} \sqrt{d(n)}  ),
		  	%   	\end{align*}
		  	%   	and 
		  	%   	   	\begin{align*}
		  	%   	   	\mathbb{W}_{n}(P_{n}^{\ast}, P_{n}) = &  \left( n^{-1/2} \sum_{i=1}^{n} \omega_{i,n} \psi(Z_{i},\theta_{\mathbf{P}}) \right)^{T}  V_{n} \left( n^{-1/2} \sum_{i=1}^{n} \omega_{i,n} \psi(Z_{i},\theta_{\mathbf{P}}) \right) \\
		  	%   	   	& + o_{P_{n}^{\ast}}( n ||a_{n}||^{2}_{2} + n^{-1/2} ||a_{n}||_{2} \sqrt{d(n)}  )~wpa1-\mathbf{P}.
		  	%   	   	\end{align*}
		  \end{proposition}
		  
		  \begin{proof}
		  	See Appendix \ref{app:discussion}.
		  \end{proof}

		    A few remarks are in order. First, and more importantly, we note that, our results can be applied to other test statistic provided that are \emph{asymptotically} equivalent (up to $o(\sqrt{d(n)})$) to $\mathbb{W}(P_{n},\mathbf{P})$ or to a quadratic form as in the proposition. Typically this is the case for the Likelihood ratio and Lagrange Multiplier (or Score) test statistics; see \cite{Newey19942111} Section 9.  
		  
		  Second, for the Chi-square-based approximation to be valid, $V$ must coincide with $(E_{\mathbf{P}}[\psi (X ,\theta_{0} ) \psi (X ,\theta_{0} )^{T} ])^{-1}$. 
		  The bootstrap-based approximation, however, does not require this assumption. This situation may arise, for instance, in Likelihood ratio tests under model misspecification.

%		  This, however, is not the case for bootstrap-based approximations. Thus, the bootstrap approach works on situations where the \textquotedblleft standard" approach does not. For instance, this could be important for Likelihood ratio tests under model misspecification. 
		  
		  \bigskip 
		  
		  \textbf{Choice of Weights.} We now provide some heuristic discussion regarding the weights. 
		  
		   The bootstrap procedure studied in this paper uses independent weights. Such restriction has also been used in several papers; e.g. \cite{CCK-AOS13} and \cite{Ma2005190}. This choice is largely due to the fact that the independent behavior of weights makes many of the proofs easier. It would be of interest still to extend our results to non-iid weights such as Multinomial weights --- which yield non-parametric and m-out-n bootstrap procedures. While such an extension is beyond the scope of the paper, we point out that the key step in order to do this is to extend Theorem \ref{thm:LFQ-bound} and Lemma \ref{lem:LFQ-bound} (in the Appendix) to allow for non-independent data;\footnote{At least for one sequence, either the $A$'s or $B$'s in the Theorem. Independence is, due to the technique of proof, particularly important for establishing Lemma \ref{lem:LFQ-bound}.}
		 
%		  Even within the class of independent weights, one could wonder which properties (other than those listed in the assumption) are desirable for the weights. 

		  Even within the class of independent weights, one could wonder what properties are desirable for the weights to have. Clearly, as indicated by our Assumption \ref{ass:boot-w}, restrictions on the tail behavior of the weights are important for our results. We now present a discussion, which expands on the quantitative explorations in Section \ref{sec:simul}, about what other properties might be desirable to have. 
		  
		  Heuristically, the Lindeberg interpolation result --- Theorem \ref{thm:LFQ-bound} --- relies on \textquotedblleft matching" the first and second moments. By choosing the weights to match higher moment one could expect to improve the approximation rates.\footnote{These observations are related to the four moment Theorem of Tao and Vu in the context of random matrices; see \cite{TaoVu11}.} 	  More precisely, the bounds for $\mathbf{S}_{n}$ (and $\mathbf{R}_{n}$) obtained in Lemma \ref{lem:LFQ-bound} in the Appendix only use restrictions impose the restrictions in the original data and the bootstrap weights present in Assumptions \ref{ass:data-Z}(i)(ii) and \ref{ass:boot-w}. However, it is easy to see that if one would have additional information on the higher moments, one could obtain sharper bounds for $\mathbf{S}_{n}$. For instance, to show Theorem \ref{thm:weak-norm-boot}, we apply Theorem \ref{thm:LFQ-bound} with $A_{i} = n^{-1/2} \omega_{i,n} Z_{i} $ and $B_{i} = n^{-1/2} u_{i} Z_{i}$ with $u_{i} \sim N(0,1)$. If we would have that $(\omega_{i,n})_{i=1}^{n}$ were such that $E[|\omega_{i,n}|^{4}]=E[(Z)^{4}]$ with $z \sim N(0,1)$, then $\mathbf{S}_{1,n} = 0$. A similar observation applies to $\mathbf{S}_{2,n}$ but in this case the relevant moments are $E[\omega_{i,n}^{3}]$ and $E [Z^{3}]$.\footnote{The rate of convergence of the term $\mathbf{R}_{n}$ is regulated by $q$, which is link to the bound on higher moments of the data and weights (see Assumptions \ref{ass:boot-w} and \ref{ass:data-Z}).}

		  Finally, another extension that is linked to the previous discussion, is that of refinements (or lack thereof) of certain choices of bootstrap weights. We leave this for future research.

		 	\bibliographystyle{natbib}
		 	\bibliography{mybib_boot_incr.bib}

\begin{thebibliography}{}

\bibitem[Arellano and Bond(1991)Arellano and Bond]{AB-RESTUD91}
Arellano, M. and Bond, S. (1991).
\newblock Some tests of specification for panel data: Monte carlo evidence and
  an application to employment equations.
\newblock {\em The Review of Economic Studies\/}, {\bf 58}(2), pp. 277--297.

\bibitem[Boucheron {\em et~al.}(2013)Boucheron, Lugosi, and
  Massart]{Boucheron-book13}
Boucheron, S., Lugosi, G., and Massart, P. (2013).
\newblock {\em Concentration Inequalities\/}.
\newblock Oxford Univ. Press.

\bibitem[Chatterjee(2006)Chatterjee]{Chatterjee_AOP06}
Chatterjee, S. (2006).
\newblock A generalization of the {Lindeberg} principle.
\newblock {\em The Annals of Probability\/}, {\bf 34}(6), 2061--2076.

\bibitem[Chatterjee and Meckes(2008)Chatterjee and
  Meckes]{Chatterjee-Meckes-2008}
Chatterjee, S. and Meckes, E. (2008).
\newblock Multivariate normal approximation using exchangeable pairs.
\newblock {\em ALEA Lat. Am. J. Probab. Math. Stat.}, {\bf 4}, 257--283.

\bibitem[Chen(2007)Chen]{Chen20075549}
Chen, X. (2007).
\newblock Chapter 76 large sample sieve estimation of semi-nonparametric
  models.
\newblock volume 6, Part B of {\em Handbook of Econometrics\/}, pages 5549 --
  5632. Elsevier.

\bibitem[Chen and Pouzo(2015)Chen and Pouzo]{ChenPouzo14}
Chen, X. and Pouzo, D. (2015).
\newblock {Sieve Wald} and {QLR} inferences on semi/nonparametric conditional
  moment models.
\newblock {\em Econometrica\/}, {\bf 83}(3), 1013--1079.

\bibitem[Chernozhukov {\em et~al.}(2013a)Chernozhukov, Chetverikov, and
  Kato]{CCK2013}
Chernozhukov, V., Chetverikov, D., and Kato, K. (2013a).
\newblock Comparison and anti-concentration bounds for maxima of {Gaussian}
  random vectors.
\newblock {\em ArXiv 1301.4807\/}.

\bibitem[Chernozhukov {\em et~al.}(2013b)Chernozhukov, Chetverikov, and
  Kato]{CCK-AOS13}
Chernozhukov, V., Chetverikov, D., and Kato, K. (2013b).
\newblock Gaussian approximations and multiplier bootstrap for maxima of sums
  of high-dimensional random vectors.
\newblock {\em The Annals of Statistics\/}, {\bf 41}, 2786--2819.

\bibitem[de~Jong and Bierens(1994)de~Jong and Bierens]{Djong-Bierens-ET94}
de~Jong, R. and Bierens, H. (1994).
\newblock On the limit behavior of a {Chi-Square} type test if the number of
  conditional moments tested approaches infinity.
\newblock {\em Econometric Theory\/}, {\bf 10}(01), 70--90.

\bibitem[Donald {\em et~al.}(2003)Donald, Imbens, and Newey]{DIN_JOE03}
Donald, S., Imbens, G., and Newey, W. (2003).
\newblock Empirical likelihood estimation and consistent tests with conditional
  moment restrictions.
\newblock {\em Journal of Econometrics\/}, {\bf 117}, 55--93.

\bibitem[Efron(1979)Efron]{efron1979}
Efron, B. (1979).
\newblock Bootstrap methods: Another look at the jackknife.
\newblock {\em The Annals of Statistics\/}, {\bf 7}(1), 1--26.

\bibitem[Feller(1971)Feller]{FELLER_book}
Feller, W. (1971).
\newblock {\em An Introduction to Probability Theory and its Applications\/},
  volume~II.
\newblock Wiley, 2nd edition.

\bibitem[Hall(2005)Hall]{Hall}
Hall, A. (2005).
\newblock {\em Generalized Method of Moments\/}.
\newblock Advanced Texts in Econometrics Series. Oxford University Press.

\bibitem[Hall(1986)Hall]{Hall1986}
Hall, P. (1986).
\newblock Methodology and theory for the {Bootstrap}.
\newblock In R.~F. Engle and D.~McFadden, editors, {\em Handbook of
  Econometrics\/}, volume~4, chapter~39, pages 2341--2381. Elsevier, 1 edition.

\bibitem[Hansen {\em et~al.}(1996)Hansen, Heaton, and Yaron]{HHY_JBES96}
Hansen, L., Heaton, J., and Yaron, A. (1996).
\newblock Finite-sample properties of some alternative {GMM} estimators.
\newblock {\em Journal of Business and Economic Statistics\/}, {\bf 14}(3),
  262–280.

\bibitem[Hansen(1982)Hansen]{Hansen-ECMA82}
Hansen, L.~P. (1982).
\newblock Large sample properties of generalized method of moments estimators.
\newblock {\em Econometrica\/}, {\bf 50}(4), 1029--54.

\bibitem[He and Shao(2000)He and Shao]{He-Shao2000}
He, X. and Shao, Q.-M. (2000).
\newblock On parameters of increasing dimensions.
\newblock {\em Journal of Multivariate Analysis\/}, {\bf 73}, 120--135.

\bibitem[Hjort {\em et~al.}(2009)Hjort, McKeague, and Keilegom]{HmcKVK_AOS09}
Hjort, N., McKeague, I., and Keilegom, I.~V. (2009).
\newblock Extending the scope of empirical likelihood.
\newblock {\em The Annals of Statistics\/}, {\bf 37}(3), 1079--1111.

\bibitem[Horowitz(2001)Horowitz]{Horowitz2001}
Horowitz, J.~L. (2001).
\newblock {The Bootstrap}.
\newblock In R.~F. Engle and D.~McFadden, editors, {\em Handbook of
  Econometrics\/}, volume~5, chapter~52, pages 3159--3228. Elsevier, 1 edition.

\bibitem[Imbens(2002)Imbens]{Imbens-JBES02}
Imbens, G.~W. (2002).
\newblock Generalized method of moments and empirical likelihood.
\newblock {\em Journal of Business and Economic Statistics\/}, {\bf 20}(4),
  493--506.

\bibitem[Imbens {\em et~al.}(1998)Imbens, Spady, and Johnson]{ISJ_ECMA98}
Imbens, G.~W., Spady, R.~H., and Johnson, P. (1998).
\newblock Information-theoretic approaches to inference in moment condition
  models.
\newblock {\em Econometrica\/}, {\bf 66}(2), 333--358.

\bibitem[Johnson {\em et~al.}(1985)Johnson, Schechtman, and
  Zinn]{Johnson-Schechtman-Zinn-85}
Johnson, W.~B., Schechtman, G., and Zinn, J. (1985).
\newblock Best constants in moment inequalities for linear combinations of
  independent and exchangeable random variables.
\newblock {\em The Annals of Probability\/}, {\bf 13}(1), pp. 234--253.

\bibitem[Kitamura and Stutzer(1997)Kitamura and Stutzer]{KS_ECMA97}
Kitamura, Y. and Stutzer, M. (1997).
\newblock An information-theoretic alternative to generalized method of moments
  estimation.
\newblock {\em Econometrica\/}, {\bf 65}(4), 861--874.

\bibitem[Koenker and Machado(1999)Koenker and Machado]{K-M_JOE93}
Koenker, R. and Machado, J.~A. (1999).
\newblock {GMM} inference when the number of moment conditions is large.
\newblock {\em Journal of Econometrics\/}, {\bf 93}(2), 327 -- 344.

\bibitem[Ma and Kosorok(2005a)Ma and Kosorok]{Ma2005}
Ma, S. and Kosorok, M.~R. (2005a).
\newblock Robust semiparametric m-estimation and the weighted bootstrap.
\newblock {\em Journal of Multivariate Analysis\/}, {\bf 96}(1), 190 -- 217.

\bibitem[Ma and Kosorok(2005b)Ma and Kosorok]{Ma2005190}
Ma, S. and Kosorok, M.~R. (2005b).
\newblock Robust semiparametric m-estimation and the weighted bootstrap.
\newblock {\em Journal of Multivariate Analysis\/}, {\bf 96}(1), 190 -- 217.

\bibitem[Mammen(1989)Mammen]{Mammen_AOS89}
Mammen, E. (1989).
\newblock Asymptotics with increasing dimension for robust regression with
  applications to the bootstrap.
\newblock {\em The Annals of Statistics\/}, {\bf 17}(61), 382--400.

\bibitem[Mammen(1993)Mammen]{mammen1993}
Mammen, E. (1993).
\newblock Bootstrap and {Wild} bootstrap for high dimensional linear models.
\newblock {\em The Annals of Statistics\/}, {\bf 21}(1), 255--285.

\bibitem[Murphy and der Vaart(2000)Murphy and der Vaart]{VdV-Murphy00}
Murphy, S. and der Vaart, A.~V. (2000).
\newblock On profile likelihood.
\newblock {\bf 95}, 449--485.

\bibitem[Newey and McFadden(1994)Newey and McFadden]{Newey19942111}
Newey, W.~K. and McFadden, D. (1994).
\newblock Chapter 36: Large sample estimation and hypothesis testing.
\newblock volume~4 of {\em Handbook of Econometrics\/}, pages 2111 -- 2245.
  Elsevier.

\bibitem[Owen(1988)Owen]{Owen-BIO88}
Owen, A. (1988).
\newblock Empirical likelihood ratio confidence intervals for a single
  functional.
\newblock {\em Biometrika\/}, {\bf 75}(2), 237--249.

\bibitem[Owen(1990)Owen]{Owen-90}
Owen, A. (1990).
\newblock {\em Empirical Likelihood\/}.
\newblock Chapman and Hall/CRC.

\bibitem[Peng and Schick(2012)Peng and Schick]{Peng-Schick-12}
Peng, H. and Schick, A. (2012).
\newblock Asymptotic normality of quadratic forms with random vectors of
  increasing dimension.
\newblock {\em Working Paper\/}.

\bibitem[Pollard(2001)Pollard]{Pollard_01}
Pollard, D. (2001).
\newblock {\em A User's Guide to Measure Theoretic Probability\/}.
\newblock Cambridge University Press.

\bibitem[Portnoy(1988)Portnoy]{Portnoy_AOS88}
Portnoy (1988).
\newblock Asymptotic behavior of likelihood methods for exponential families
  when the number of parameters tends to infinity.
\newblock {\em The Annals of Statistics\/}, {\bf 16}(1), 356--366.

\bibitem[Radulovic(1998)Radulovic]{Radulovic-98}
Radulovic, D. (1998).
\newblock Can we bootstrap even if {CLT} fails?
\newblock {\em Journal of Theoretical Probability\/}, {\bf 11}(3), 813--830.

\bibitem[Raic(2004)Raic]{Raic2014}
Raic, M. (2004).
\newblock A multivariate {CLT} for decomposable random vectors with finite
  second moments.
\newblock {\em Journal of Theoretical Probability\/}, {\bf 17}(3), 573--603.

\bibitem[Rollin(2013)Rollin]{Rollin2013}
Rollin, A. (2013).
\newblock {Stein's} method in high dimensions with applications.
\newblock {\em ArXiv 1101.4454\/}.

\bibitem[Slepian(1962)Slepian]{Slepian1962}
Slepian, D. (1962).
\newblock The one-sided barrier problem for {Gaussian} noise.
\newblock {\em Bell System Technical Journal\/}, {\bf 41}(2), 463--501.

\bibitem[Smith(1997)Smith]{Smith-EJ97}
Smith, R.~J. (1997).
\newblock Alternative semi-parametric likelihood approaches to generalized
  method of moments estimation.
\newblock {\em The Economic Journal\/}, {\bf 107}(441), 503--519.

\bibitem[Spokoiny and Zhilova(2014)Spokoiny and Zhilova]{Spoikony-Zhilova-14}
Spokoiny, V. and Zhilova, M. (2014).
\newblock Bootstrap confidence sets under a model misspecification.
\newblock {\em arXiv:1410.0347v1\/}.

\bibitem[Stein(1981)Stein]{Stein1981}
Stein, C. (1981).
\newblock Estimation of the mean of a multivariate {Normal} distribution.
\newblock {\em The Annals of Statistics\/}, {\bf 9}(6), 1135--1151.

\bibitem[Tao and Vu(2011)Tao and Vu]{TaoVu11}
Tao, T. and Vu, V. (2011).
\newblock Random matrices: The {Four Moment Theorem} for {Wigner} ensembles.
\newblock {\em ArXiv 1112.1976\/}.

\bibitem[Van~der Vaart(2000)Van~der Vaart]{VdV00}
Van~der Vaart, A. (2000).
\newblock {\em Asymptotic Statistics\/}.
\newblock Cambridge University Press.

\bibitem[Vershynin(2012a)Vershynin]{Vershynin-JTP12}
Vershynin, R. (2012a).
\newblock How close is the sample covariance matrix to the actual covariance
  matrix?
\newblock {\em Journal of Theoretical Probability\/}, {\bf 25}, 655--686.

\bibitem[Vershynin(2012b)Vershynin]{Vershynin-12}
Vershynin, R. (2012b).
\newblock Introduction to the non-asymptotic analysis of random matrices.
\newblock In {\em Compressed sensing\/}, pages 210--268. Cambridge Univ. Press.

\bibitem[Wasserman(2014)Wasserman]{Wasserman_2014}
Wasserman, L. (2014).
\newblock {Stein's} method and the bootstrap in low and high dimensions: A
  tutorial.
\newblock {\em Working Paper\/}.

\bibitem[Xu {\em et~al.}(2014)Xu, Zhang, and Wu]{Xu_Zhang_Wu_14}
Xu, M., Zhang, D., and Wu, W.~B. (2014).
\newblock {$L^{2}$} asymptotics for high-dimensional data.
\newblock {\em arXiv: 1405.7244\/}.

\bibitem[Zhang and Cheng(2014)Zhang and Cheng]{Zhang-Cheng-14}
Zhang, X. and Cheng, G. (2014).
\newblock Bootstrapping high dimensional time series.
\newblock {\em arXiv:1406.1037v2\/}.

\end{thebibliography}
	
	\newpage
	
	\appendix
	
	{\Large{{\bf{Appendix}}}}

	\section{Proof of Theorems \ref{thm:LFQ-bound}, \ref{thm:weak-norm-boot} and  \ref{thm:weak-norm-ori}}
	\label{app:lindeberg}
	
The next lemma provides a bound for $\mathbf{S}_{n}$ and $\mathbf{R}_{n}$ in Theorem \ref{thm:LFQ-bound}. Henceforth, let $\mathbb{S}_{i:n} \equiv \sum_{j=1}^{i-1} A_{j} + 0 + \sum_{j=i+1}^{n} B_{j} \equiv \sum_{j=1}^{n} S_{j}$.
	
	\begin{lemma} %[{lem:LFQ-bound}] 
		\label{lem:LFQ-bound}
		Suppose the same conditions of Theorem \ref{thm:LFQ-bound}. Then, \begin{align*}
			\mathbf{S}_{1,n} \leq &  L_{2}(f) \sum_{i=1}^{n} E[ || B_{i}||^{4}_{2}  + || A_{i}||^{4}_{2}  ] \\ 
			\mathbf{S}_{2,n} \leq & L_{2}(f) \sqrt{ \sum_{j=1}^{n} tr\{ C_{j} \} } \sum_{i=1}^{n} \left( E \left[  || B_{i}||^{3}_{2}  \right] +  E \left[  || A_{i}||^{3}_{2}  \right] \right).
		\end{align*}
		And, for any $q>0$ \begin{align*}
			 \mathbf{R}_{n} \precsim    \sum_{i=1}^{n} \left( E \left[ \left(  \mathbb{S}_{i:n}^{T} B_{i}   \right)^{2+q} + \left(  \mathbb{S}_{i:n}^{T} A_{i}   \right)^{2+q} \right] +  E \left[  ||B_{i}||_{2} ^{4+2q}  \right] +  E \left[  ||A_{i}||_{2}  ^{4+2q}  \right] \right).
		\end{align*}  
		And 
		\begin{align*}
		& \sum_{i=1}^{n} E \left[ \left(  \mathbb{S}_{i:n}^{T} B_{i}   \right)^{2+q}  \right] \\
		%\precsim & \max \left\{ \left( \sum_{j=1}^{n} tr\{ C_{j} \} \right)^{1+0.5q}  \sum_{i=1}^{n}  \left(  E \left[  ||B_{i}||^{2}_{2}  \right] \right)^{1+0.5q} ,\sum_{i=1}^{n} E \left[  ||B_{i}||_{2}   ^{2+q}  \right] \sum_{j=1}^{n} E \left[ \left(  ||S_{j}||_{2}   \right)^{2+q}  \right] \right\} .
		\precsim & \sum_{i=1}^{n}  E[  ||B_{i}||^{2+q}_{2} ] \max \left\{ \left(   \sum_{j=1}^{n}   E \left[  ||S_{j}||^{2}_{2} \right]  \right)^{1+0.5q}  ,     \sum_{j=1}^{n} E \left[  ||S_{j}||_{2}  ^{2+q}  \right]   \right\} .
		\end{align*}
		An analogous expression holds for $\sum_{i=1}^{n} E \left[ \left(  \mathbb{S}_{i:n}^{T} A_{i}   \right)^{2+q}  \right]$.
	\end{lemma}

	\begin{proof}[Proof of Lemma \ref{lem:LFQ-bound}]
		$\mathbf{S}_{1,n}$ is trivially bounded by $ L_{2}(f) \sum_{i=1}^{n} E[ || B_{i}||^{4}_{2}  + || A_{i}||^{4}_{2}  ]$. Regarding $\mathbf{S}_{2,n}$, observe that 
		
		\begin{align*}
		 \sum_{i=1}^{n} |E\left[ f''  \left( || \mathbb{S}_{i:n} ||^{2}_{2}  \right) \mathbb{S}_{i:n}^{T} \right] \left( E [  B_{i} || B_{i}||^{2}_{2} ] \right)| \leq & L_{2}(f)  \sum_{i=1}^{n} E\left[ ||\mathbb{S}_{i:n}||_{2} || B_{i}||^{3}_{2}  \right]\\
		\leq & L_{2}(f) \sum_{i=1}^{n} \sqrt{ E\left[ ||\mathbb{S}_{i:n}||^{2}_{2}  \right]} E \left[  || B_{i}||^{3}_{2}  \right]
		\end{align*}
		by independence of $\mathbb{S}_{i:n}$ and $B_{i}$ and Cauchy-Schwarz. Also, $E[ \mathbb{S}_{i:n}\mathbb{S}_{i:n}^{T} ] = \sum_{j=1}^{n} E[S_{j}S_{j}^{T}]$, so $E\left[ ||\mathbb{S}_{i:n}||^{2}_{2}  \right] = tr \{  E[ \mathbb{S}_{i:n}\mathbb{S}_{i:n}^{T} ]  \} = \sum_{j=1}^{n} tr\{ C_{j} \}$. A similar results holds when $B_{i}$ is replaced by $A_{i}$. Therefore
		\begin{align*}
		\mathbf{S}_{2,n} \leq & L_{2}(f) \sqrt{ \sum_{j=1}^{n} tr\{ C_{j} \} } \sum_{i=1}^{n} \left( E \left[  || B_{i}||^{3}_{2}  \right] +  E \left[  || A_{i}||^{3}_{2}  \right] \right).
		\end{align*}

		Regarding $\mathbf{R}_{n}$. Note that 		
		\begin{align*}
		  \sum_{i=1}^{n}  E \left[ \left(  \mathbb{S}_{i:n}^{T} B_{i}  + ||B_{i}||^{2}_{2}    \right)^{2+q}  \right] 	\precsim   \left( \sum_{i=1}^{n}  E \left[ \left(  \mathbb{S}_{i:n}^{T} B_{i}   \right)^{2+q}  \right] + \sum_{i=1}^{n} E \left[ \left( ||B_{i}||_{2}  \right)^{4+2q}  \right] \right).
		\end{align*}

		Observe that $E \left[ \left(  \mathbb{S}_{i:n}^{T} B_{i}   \right)^{2+q}  \right] = E \left[ E \left[ \left(  \sum_{j=1}^{n} S_{j}^{T}b_{i}   \right)^{2+q}  | B_{i} = b_{i} \right] \right]$. Since $(S_{j})_{j}$ does not contain $B_{i}$, conditioning on $B_{i}=b_{i}$, $(S_{j}^{T}b_{i})_{j}$ is an independent sequence. 
		
		Therefore, by \cite{Johnson-Schechtman-Zinn-85}, for any $q>0$, 
		\begin{align*}
		& E \left[ \left(  \mathbb{S}_{i:n}^{T} b_{i}   \right)^{2+q}  \right] \\
		& \precsim \left( \max \left\{ \sqrt{ E \left[ \left(  \sum_{j=1}^{n} S_{j}^{T}b_{i}   \right)^{2}  \right]} , \left( \sum_{j=1}^{n} E \left[ \left(  S_{j}^{T}b_{i}   \right)^{2+q}  \right] \right)^{1/(2+q)}  \right\} \right)^{2+q}	
		\end{align*}
		(where the expectation is only with respect to $(S_{j})_{j=1}^{n}$, not $b_{i}$). By independence, and the fact that $E[S^{T}_{j} b_{i} ] = 0$,
		\begin{align*}
		E \left[ \left(  \sum_{j=1}^{n} S_{j}^{T}b_{i}   \right)^{2}  \right] = & E \left[ \sum_{j=1}^{n}  \left(  S_{j}^{T}b_{i}   \right)^{2}  \right] %=   E \left[ b_{i}^{T}  \left(  \sum_{j=1}^{n} S_{j} S_{j}^{T}    \right) b_{i}  \right] \\ 
		%= & E \left[ tr \left\{  \left(  \sum_{j=1}^{n}  S_{j} S_{j}^{T}    \right) b_{i} b_{i}^{T}  \right\} \right] \\
		= & tr \left\{  E \left[  \left(  \sum_{j=1}^{n}  S_{j} S_{j}^{T}    \right) \right]   b_{i} b_{i}^{T} \right\} .
		\end{align*}
		
		Also, note that\begin{align*}
		\sum_{j=1}^{n} E \left[ \left(  S_{j}^{T}b_{i}   \right)^{2+q}  \right] \leq \sum_{j=1}^{n} E \left[ \left(  ||S_{j}||_{2} ||b_{i}||_{2}    \right)^{2+q}  \right] =  \left(  ||b_{i}||_{2}   \right)^{2+q}  \sum_{j=1}^{n} E \left[ \left(  ||S_{j}||_{2}   \right)^{2+q}  \right]   .
		\end{align*}
		
		Therefore, using these bounds and taking expectation with respect to $B_{i}$ and after straightforward algebra, 
		\begin{align*}
		&  \sum_{i=1}^{n} E \left[ \left(  \mathbb{S}_{i:n}^{T} B_{i}   \right)^{2+q}  \right]   \\
%		\precsim & \sum_{i=1}^{n} E \left[ \left( \max \left\{ \sqrt{ tr \left\{  E \left[  \left(  \sum_{j=1}^{n}  S_{j} S_{j}^{T}    \right) \right]  B_{i} B_{i}^{T} \right\}} , ||B_{i}||_{2}  \left(    \sum_{j=1}^{n} E \left[  ||S_{j}||_{2}  ^{2+q}  \right]  \right)^{1/(2+q)}  \right\} \right)^{2+q} \right] \\
		% \precsim & \sum_{i=1}^{n} E \left[ \max \left\{ \left(  tr \left\{  E \left[  \left(  \sum_{j=1}^{n}  S_{j} S_{j}^{T}    \right) \right]  B_{i} B_{i}^{T} \right\} \right)^{1+0.5q}, ||B_{i}||^{2+q}_{2}  \left(    \sum_{j=1}^{n} E \left[  ||S_{j}||_{2}  ^{2+q}  \right]  \right) \right\} \right] \\
%		\precsim & \sum_{i=1}^{n}  \max \left\{ \left(  tr \left\{  E \left[  \left(  \sum_{j=1}^{n}  S_{j} S_{j}^{T}    \right) \right]  \right\} \right)^{1+0.5q}  E[ ( tr\{ B_{i} B_{i}^{T}   \} )^{1+0.5q} ], E[ ||B_{i}||^{2+q}_{2} ] \left(    \sum_{j=1}^{n} E \left[  ||S_{j}||_{2}  ^{2+q}  \right]  \right) \right\} \\
		\precsim & \sum_{i=1}^{n}  E[  ||B_{i}||^{2+q}_{2} ] \max \left\{ \left(   \sum_{j=1}^{n}   E \left[  ||S_{j}||^{2}_{2} \right]  \right)^{1+0.5q}  ,     \sum_{j=1}^{n} E \left[  ||S_{j}||_{2}  ^{2+q}  \right]   \right\} .
		\end{align*}
		Analogous steps can be taken to show the same result replacing $B_{i} $ by $A_{i}$; they will be omitted.
	\end{proof}

	\begin{proof}[Proof of Theorem \ref{thm:LFQ-bound} ]
		
		 Observe that $(S_{i})_{i=1} ^{n}$ are independent and $E[S_{i}]=0$, also $E[S_{i}S_{i}^{T}] = E[B_{i} B_{i}^{T}] = C_{i}$. Also, note that $\mathbb{S}_{1:n} \equiv \sum_{i=1}^{n} B_{i} - B_{1}$ and $\mathbb{S}_{n:n} \equiv \sum_{i=1}^{n} A_{i} - A_{n}$. Moreover
		\begin{align}
		\mathbb{S}_{i:n} + A_{i} =\left( \sum_{j=1}^{i} A_{j} + \sum_{j=i+1}^{n} B_{j}   \right) = \mathbb{S}_{i+1:n} + B_{i+1}.
		\end{align}
		Therefore, 
		\begin{align*}
		\sum_{i=1}^{n} E\left[f \left( ||\mathbb{S}_{i:n} + B_{i}||^{2}_{2}  \right) - f \left( || \mathbb{S}_{i:n} + A_{i}||^{2}_{2} \right) \right] =   E\left[f \left( ||\sum_{i=1}^{n} B_{i} ||^{2}_{2}  \right) - f \left(  ||\sum_{i=1}^{n} A_{i} ||^{2}_{2} \right) \right].
		\end{align*}
		
		Observe that $||\mathbb{S}_{i:n} + B_{i}||^{2}_{2} = ||\mathbb{S}_{i:n}||^{2}_{2} + ||B_{i}||^{2}_{2}  + 2 \mathbb{S}_{i:n}^{T} B_{i} $. Therefore, by this fact and three times differentiability of $f$, it follows that
		\begin{align*}
		f \left( ||\mathbb{S}_{i:n} + B_{i}||^{2}_{2}  \right) -f \left( ||\mathbb{S}_{i:n} ||^{2}_{2}  \right) 	=	&  f'  \left( ||\mathbb{S}_{i:n} ||^{2}_{2}  \right) \left( ||B_{i}||^{2}_{2}  + 2 \mathbb{S}_{i:n}^{T} B_{i} \right)\\
		& + 0.5 f''  \left( ||\mathbb{S}_{i:n} ||^{2}_{2}  \right) \left( ||B_{i}||^{2}_{2}  + 2 \mathbb{S}_{i:n}^{T} B_{i} \right)^{2} \\
		& + R_{i,1,n}
		\end{align*}
		where $R_{i,1,n}$ is a reminder term which will be defined later. Similarly
		\begin{align*}
		f \left( ||\mathbb{S}_{i:n} + A_{i}||^{2}_{2}  \right) -f \left( ||\mathbb{S}_{i:n} ||^{2}_{2}  \right) 	=	&  f'  \left( ||\mathbb{S}_{i:n} ||^{2}_{2}  \right) \left( ||A_{i}||^{2}_{2}  + 2 \mathbb{S}_{i:n}^{T} A_{i} \right)\\
		& + 0.5 f''  \left( ||\mathbb{S}_{i:n} ||^{2}_{2}  \right) \left( ||A_{i}||^{2}_{2}  + 2 \mathbb{S}_{i:n}^{T} A_{i} \right)^{2} \\
		& + R_{i,2,n}.
		\end{align*}
		
		Hence\begin{align*}
		& E\left[f \left( ||\mathbb{S}_{i:n} + B_{i}||^{2}_{2}  \right) - f \left( || \mathbb{S}_{i:n} + A_{i}||^{2}_{2} \right) \right] \\
		= & E\left[ f'  \left( ||\mathbb{S}_{i:n} ||^{2}_{2}  \right) \left( ||A_{i}||^{2}_{2} - ||B_{i}||^{2}_{2}   + 2 \mathbb{S}_{i:n}^{T} (A_{i}  - B_{i}) \right)  \right] \\
		& + 0.5 E\left[ f''  \left( ||\mathbb{S}_{i:n} ||^{2}_{2}  \right) \left\{ \left( ||B_{i}||^{2}_{2}  + 2 \mathbb{S}_{i:n}^{T} B_{i} \right)^{2}  - \left( ||A_{i}||^{2}_{2}  + 2 \mathbb{S}_{i:n}^{T} A_{i} \right)^{2} \right\}  \right]\\
		& + E \left[ R_{i,1,n} - R_{i,2,n} \right]\\
		\equiv & F_{i,n} + S_{i,n} + E \left[ R_{i,1,n} - R_{i,2,n} \right].
		\end{align*}
		
		Therefore, it suffices to bound the \emph{first order terms} $F_{n} \equiv \sum_{i=1}^{n} F_{i,n}$, \emph{second order terms} $S_{n} \equiv \sum_{i=1}^{n} S_{i,n}$ and \emph{the remainder terms} $E \left[ R_{i,1,n} - R_{i,2,n} \right]$.

		\bigskip

		\textsc{The First order terms, $F_{n}$.} Since $\mathbb{S}_{i:n}$ is independent with $A_{i}$ and $B_{i}$ and $E[A_{i}]=E[B_{i}]=0$ and $E[A_{i}A_{i}^{T}]=E[B_{i}B_{i}^{T}]$ it readily follows that
		\begin{align*}
		\sum_{i=1}^{n} E\left[ f'  \left( || \mathbb{S}_{i:n} ||^{2}_{2}  \right) \mathbb{S}_{i:n}^{T} \left( B_{i}  - A_{i}   \right)  \right] = \sum_{i=1}^{n} E\left[ f'  \left( || \mathbb{S}_{i:n} ||^{2}_{2}  \right) \mathbb{S}_{i:n}^{T} \right] E \left[ \left( B_{i}  - A_{i}   \right)  \right] = 0
		\end{align*}
		and
		\begin{align*}
		\sum_{i=1}^{n} E\left[ f'  \left( || \mathbb{S}_{i:n} ||^{2}_{2}  \right) \left( ||B_{i}||^{2}_{2}  - ||A_{i}||^{2}_{2}   \right)  \right] = \sum_{i=1}^{n} E\left[ f'  \left( || \mathbb{S}_{i:n} ||^{2}_{2}  \right) \right] E \left[ \left( ||B_{i}||^{2}_{2}  - ||A_{i}||^{2}_{2}   \right)  \right] = 0.
		\end{align*}

		\bigskip

		\textsc{The term Second order terms, $S_{n}$.} For this term it suffices to study the following terms:
		\begin{align*}
		S_{1,n} \equiv &\sum_{i=1}^{n} E\left[ f''  \left( || \mathbb{S}_{i:n} ||^{2}_{2}  \right) \left( || B_{i}||^{4}_{2}  - || A_{i}||^{4}_{2}    \right)  \right]\\
		S_{2,n} \equiv &\sum_{i=1}^{n} E\left[ f''  \left( || \mathbb{S}_{i:n} ||^{2}_{2}  \right) 4 \left( (\mathbb{S}_{i:n}^{T}B_{i})^{2}   - (\mathbb{S}_{i:n}^{T}A_{i})^{2}    \right)  \right]\\
		S_{3,n} \equiv & \sum_{i=1}^{n} E\left[ f''  \left( || \mathbb{S}_{i:n} ||^{2}_{2}  \right) 4 \mathbb{S}_{i:n}^{T}\left( B_{i} || B_{i}||^{2}_{2}  -  A_{i} || A_{i}||^{2}_{2}    \right)  \right].
		\end{align*}
		
		By independence of $\mathbb{S}_{i:n}$ with $A_{i}$ and $B_{i}$, it follows that \begin{align*}
			S_{1,n} = \sum_{i=1}^{n} E\left[ f''  \left( || \mathbb{S}_{i:n} ||^{2}_{2}  \right) \right]  E \left[  || B_{i}||^{4}_{2}  - || A_{i}||^{4}_{2}   \right] .
		\end{align*}
		%and thus \begin{align*}
		%	|S_{1,n} | \leq L_{2}(f) \sum_{i=1}^{n} |E [  || B_{i}||^{4}_{2} ] - E [ || A_{i}||^{4}_{2}   ] |.
		%\end{align*}
		Regarding $S_{2,n}$, because $\mathbb{S}_{i:n}$ is independent to $A_{i}$ and $B_{i}$ and $E[A_{i} A_{i}^{T}] = E[B_{i} B_{i}^{T}]$, it follows that $E\left[ \mathbb{S}_{i:n}^{T}B_{i}B_{i}^{T}\mathbb{S}_{i:n} \right]  = E\left[ \mathbb{S}_{i:n}^{T}A_{i}A_{i}^{T}\mathbb{S}_{i:n} \right]$ and thus $S_{2,n} = 0$.

		Finally, regarding $S_{3,n}$, observe that by independence of $\mathbb{S}_{i:n}$ and $B_{i}$ and $A_{i}$
		\begin{align*}
			|S_{3,n}| \leq 4 \sum_{i=1}^{n} |E\left[ f''  \left( || \mathbb{S}_{i:n} ||^{2}_{2}  \right) \mathbb{S}_{i:n}^{T} \right] \left( E [  B_{i} || B_{i}||^{2}_{2} ] -  E [ A_{i} || A_{i}||^{2}_{2}   ] \right)|.
		\end{align*}

		\bigskip

		\textsc{The remainder terms, $R_{1,n}$ and $R_{2,n}$.} By Taylor's Theorem it follows that: For any $q>0$
			\begin{align*}
			\sum_{i=1}^{n}  E \left[ |R_{i,1,n} |  \right] \precsim & L_{2}(f)^{1-q} L_{3}(f)^{q} \sum_{i=1}^{n}  E \left[ \left(  \mathbb{S}_{i:n}^{T} B_{i}  + ||B_{i}||^{2}_{2}    \right)^{2+q}  \right]. % \\
%			\precsim & L_{2}(f)^{1-q} L_{3}(f)^{q}  \left( \sum_{i=1}^{n}  E \left[ \left(  \mathbb{S}_{i:n}^{T} B_{i}   \right)^{2+q}  \right] + \sum_{i=1}^{n} E \left[ \left( ||B_{i}||_{2}  \right)^{4+2q}  \right] \right).
			\end{align*}
		
	\end{proof}

	\subsection{Proof of Theorem \ref{thm:weak-norm-boot}}
	\label{app:proof-weak-boot}

	\begin{proof}[Proof of Theorem \ref{thm:weak-norm-boot}]

		We first note that is enough to bound \begin{align*}
		\mathbf{P}_{n} \left( \left\{	\sup_{f \in \mathcal{C}_{h^{-1}}} \left|  E_{\mathbf{P}^{\ast}_{n}} \left[ f \left( || \sqrt{n} \mathbb{Z}^{\ast}_{n}||^{2}_{2}  \right)  |Z^{n} \right] - E_{\boldsymbol{\Phi}^{\ast}_{n}} \left[ f \left( || \sqrt{n} \mathbb{U}_{n} ||^{2}_{2} \right)  | Z^{n} \right]  \right| \geq \varepsilon  \right\} \cap K_{n}  \right)
		\end{align*}
		where $K_{n} \equiv \{  Z^{n}   : n^{-1} \sum_{i=1}^{n} ||Z_{i}||^{2}_{2}  \leq  (0.5 \varepsilon)^{-1} tr \{ \Sigma_{n}  \} \equiv M_{n} \}$.

		The strategy of proof consists of applying the results in  Theorem \ref{thm:LFQ-bound} and Lemma \ref{lem:LFQ-bound}, with $A_{i} = n^{-1/2} \omega_{in} Z_{i}$ and $B_{i} = n^{-1/2} u_{i} Z_{i}$ where $u_{i} \sim N(0,1)$.\footnote{Note that $\mathbb{U}_{n}$ can be cast as $n^{-1} \sum_{i=1}^{n} u_{i} Z_{i}$. } Then use the Markov inequality and show that the expectation (under $	\mathbf{P}_{n}$) of the terms in the RHS of the main expression in Theorem \ref{thm:LFQ-bound}, $\mathbf{S}_{n}$ and $\mathbf{R}_{n}$,  vanishes as $n \rightarrow \infty$.

		\bigskip

		\textsc{The leading terms, $\mathbf{S}_{n}$.} For this case $\sum_{i=1}^{n} E[(||B_{i}||_{2} )^{4}] \precsim  n^{-2} \sum_{i=1}^{n} ||Z_{i}||^{4}_{2} $ and $\sum_{i=1}^{n} E[(||A_{i}||_{2} )^{4}] \precsim  n^{-2} \sum_{i=1}^{n} ||Z_{i}||^{4}_{2} $, under Assumption \ref{ass:boot-w}. Therefore, $\mathbf{S}_{1,n}$ in Theorem \ref{thm:LFQ-bound} is bounded above (up to a constant) by $n^{-1}  \left( n^{-1} \sum_{i=1}^{n} ||Z_{i}||^{4}_{2}  \right)$. 
		
		Therefore, since $L_{2}(f) = h^{-2}$,  $E_{\mathbf{P}_{n} } [\mathbf{S}_{1,n} ] \precsim h^{-2} n^{-2} \sum_{i=1}^{n} E_{\mathbf{P}_{n} }[ ||Z_{i}||^{4}_{2} ] = h^{-2} n^{-1} E_{\mathbf{P}_{n} }[ ||Z_{1}||^{4}_{2} ] $
%		By Jensen inequality, for any $q\geq 0$, $n^{-1/2} E_{\mathbf{P}_{n} }[ ||Z_{1}||^{4}_{2} ] \leq n^{-1/2} ( E_{\mathbf{P}_{n} }[ ||Z_{1}||^{4+2q}_{2} ] )^{\frac{4}{4+2q}} =  (n^{-0.5(1+0.5q)}  E_{\mathbf{P}_{n} }[ ||Z_{1}||^{4+2q}_{2} ] )^{\frac{4}{4+2q}}$ 
which is of order $o(h^{-2})$ by Assumption \ref{ass:data-Z}(i).

		Observe that in this case $E[S_{i}S_{i}^{T}] =n^{-1}  Z_{i}Z_{i}^{T}$ and thus 
		\begin{align*}
		\mathbf{S}_{2,n} \precsim & h^{-2} \sqrt{ n^{-1} \sum_{i=1}^{n} ||Z_{i}||^{2}_{2} } n^{-3/2} \sum_{i=1}^{n} E[|\omega_{in}|^{3} + |u_{i,n}|^{3} ]  ||Z_{i}||^{3}_{2} \\
		\precsim &  h^{-2} \sqrt{ n^{-1} \sum_{i=1}^{n} ||Z_{i}||^{2}_{2} } n^{-3/2} \sum_{i=1}^{n} ||Z_{i}||^{3}_{2}.
		\end{align*}
		
		%================
		%Typically: $\sqrt{ n^{-1} \sum_{i=1}^{n} ||Z_{i}||^{2}_{2} } n^{-3/2} \sum_{i=1}^{n} ||Z_{i}||^{3}_{2} \approx n^{-1/2} d(n)^{1/2} d(n)^{3/2}$
		%
		%================
		
		For any $Z^{n} \in K_{n}$,  $\mathbf{S}_{2,n} \precsim h^{-2} \sqrt{M_{n}} n^{-3/2} \sum_{i=1}^{n} ||Z_{i}||^{3}_{2}$. Therefore, $E_{\mathbf{P}_{n} } [\mathbf{S}_{2,n} 1\{ K_{n} \} ] \precsim h^{-2} \sqrt{M_{n}}  n^{-3/2} \sum_{i=1}^{n} E_{\mathbf{P}_{n} }[ ||Z_{i}||^{3}_{2} ] = h^{-2} \sqrt{M_{n}} n^{-1/2}  E_{\mathbf{P}_{n} }[ ||Z_{1}||^{3}_{2} ] $, which is of order $o(h^{-2})$ by Assumption \ref{ass:data-Z}(i).

		\bigskip

		\textsc{The remainder terms,  $\mathbf{R}_{n}$.} To bound the remainder term in the expression of Theorem \ref{thm:LFQ-bound} we use Lemma \ref{lem:LFQ-bound} and the fact that $L_{2}(f) = h^{-2}$. Observe that $\left(  tr \left\{  \sum_{j=1}^{n}  E \left[  \left( S_{j}^{T} S_{j}    \right) \right] \right\} \right)^{1+0.5q}  = \left(  tr \left\{  n^{-1} \sum_{j=1}^{n} Z_{j} Z_{j}^{T}  \right\} \right)^{1+0.5q}  =  \left(   n^{-1} \sum_{j=1}^{n} ||Z_{j}||^{2}_{2} \right)^{1+0.5q}  $ . Also, \begin{align*}
			\sum_{i=1}^{n} E \left[ \left(  ||B_{i}||_{2}   \right)^{2+q}  \right] = n^{-(1+0.5q)}\sum_{i=1}^{n} E \left[ |u_{i,n}|^{2+q}  \right]  ||Z_{i}||_{2}^{2+q}  \precsim  n^{-(1+0.5q)}\sum_{i=1}^{n}  ||Z_{i}||_{2}^{2+q} 
		\end{align*}
		because of the fact that $E[|u_{i,n}|^{2+q}] \leq C <\infty$ with $q=\gamma$. Similarly, under Assumption \ref{ass:boot-w},
		\begin{align*}
		\sum_{j=1}^{n} E \left[ \left(  ||S_{j}||_{2}   \right)^{2+q}  \right] \precsim  n^{-(1+0.5q)}\sum_{i=1}^{n}  ||Z_{i}||_{2}^{2+q} .
		\end{align*}
		
		Therefore,
		\begin{align*}
		& \sum_{i=1}^{n} E \left[ \left(  \mathbb{S}_{i:n}^{T} B_{i}   \right)^{2+q}  \right] \\
		\precsim &   n^{-(1+0.5q)} \sum_{i=1}^{n} ||Z_{i}||^{2+q}_{2}  \max \left\{ \left(   n^{-1} \sum_{j=1}^{n} ||Z_{j}||^{2}_{2} \right)^{1+0.5q}  , n^{-(1+0.5q)} \sum_{i=1}^{n} ||Z_{i}||_{2}^{2+q} \right\} \\
		\precsim & \max \left\{  n^{-(1+0.5q)} \sum_{i=1}^{n} ||Z_{i}||^{2+q}_{2}  \left(   n^{-1} \sum_{j=1}^{n} ||Z_{j}||^{2}_{2} \right)^{1+0.5q}  , n^{-(1+q)} \sum_{i=1}^{n} ||Z_{i}||_{2}^{4+2q} \right\} 
		\end{align*}
		where the last line follows from Jensen inequality. And, also note that $\sum_{i=1}^{n} E[(||B_{i}||_{2} )^{4+2q}] \precsim n^{-(2+q)} \sum_{i=1}^{n} ||Z_{i}||^{4+2q}_{2} $. 
		
		It is straightforward to check that analogous expressions hold for $\sum_{i=1}^{n} E \left[ \left(  \mathbb{S}_{i:n}^{T} A_{i}   \right)^{2+q}  \right]$ and $\sum_{i=1}^{n} E[(||A_{i}||_{2} )^{4+2q}]$.
		
		Recall that $q=\gamma$. Thus, $E_{\mathbf{P}_{n}}[ n^{-(2+q)} \sum_{i=1}^{n} ||Z_{i}||^{4+2q}_{2}  ] = n^{-(1+q)} E_{\mathbf{P}_{n}}[  ||Z_{1}||^{4+2q}_{2} ]$ which vanishes as $n \rightarrow \infty$ under Assumption \ref{ass:data-Z}(ii). Similarly,
		\begin{align*}
			E_{\mathbf{P}_{n}} \left[ \sum_{i=1}^{n} E \left[ \left(  \mathbb{S}_{i:n}^{T} B_{i}   \right)^{2+q}  \right] 1\{ Z^{n} \in K_{n} \} \right] 
		\end{align*}
		 (and $E_{\mathbf{P}_{n}} \left[ \sum_{i=1}^{n} E \left[ \left(  \mathbb{S}_{i:n}^{T} A_{i}   \right)^{2+q}  \right] 1\{ Z^{n} \in K_{n} \} \right] $) are bounded above (up to a constant) by $(M_{n})^{1+0.5q} n^{-(0.5q)} E_{\mathbf{P}_{n}}[  ||Z_{1}||^{2+q}_{2} ] + n^{-q} E_{\mathbf{P}_{n}}[  ||Z_{1}||^{4+q}_{2} ]$; both terms vanish as $n \rightarrow \infty$ under Assumption \ref{ass:data-Z}(ii) with $q= \gamma$.
		
		The desired result follows by the Markov inequality, since we proven that $E_{\mathbf{P}_{n}}[\mathbf{S}_{n} 1\{ K_{n}  \}]$ and $E_{\mathbf{P}_{n}}[\mathbf{R}_{n} 1\{ K_{n} \}]$ are of order $o(h^{-2})$.
	\end{proof}

			\subsection{Proof of Theorem \ref{thm:weak-norm-ori}}
			\label{app:proof-weak-ori}

			For the proof of Theorem \ref{thm:weak-norm-ori} we need the following simple lemma.

			\begin{lemma}\label{lem:mom-Gau} %[{lem:mom-Gau}]
				Let $d \geq 1$ and let $X \in \mathbb{R}^{d}$ such that $X \sim N(0,A)$ for some $A$ positive definite. Then for any $q>0$
				\begin{align*}
				E[ ||X||^{2q}_{2} ] \leq C(q) (tr\{ A \})^{q}
				\end{align*}
				for some $C(q) \in (0,\infty)$.
			\end{lemma}
			
			\begin{proof}[Proof of Lemma \ref{lem:mom-Gau}]

				Let $U \sim N(0,I_{d})$ and let $\Lambda$ be the diagonal matrix of eigenvalues of $A$ and $V$ the eigenvector matrix. For any $q>0$
				\begin{align*}
				E[ ||X||^{2q}_{2} ] & = E[( X^{T} X  )^{q} ] = E[( U^{T} A U  )^{q} ]  \\
				 & = E[( \xi^{T} \Lambda \xi  )^{q} ],~where~\xi=V^{T} U \\
				 & = tr\{ A \}^{q} E\left[  \left( \sum_{j=1}^{d} c_{j}(A) |\xi_{j}|^{2}  \right)^{q} \right]
				\end{align*}
				where $c_{j}(A) \equiv \frac{\lambda_{j}(A)}{\sum_{j=1}^{d}\lambda_{j}(A)}$. Since 
				\begin{align*}
				E\left[  \left( \sum_{j=1}^{d} c_{j}(A) |\xi_{j}|^{2}  \right)^{q} \right] = & \int_{0}^{\infty}  \Pr \left( \sum_{j=1}^{d} c_{j}(A) |\xi_{j}|^{2} \geq t^{1/q} \right) dt\\ 
				= & q \int_{0}^{\infty} u^{q-1} \Pr \left( \sum_{j=1}^{d} c_{j}(A) |\xi_{j}|^{2} \geq u\right) du\\
				\leq & q \int_{0}^{\infty} u^{q-1} e^{- 0.25 u } d u E \left[  e^{ 0.25 \sum_{j=1}^{d} c_{j}(A) |\xi_{j}|^{2} }\right]\\
				\leq & q \int_{0}^{\infty} u^{q-1} e^{- 0.25 u } d u  \sum_{j=1}^{d} c_{j}(A) E \left[  e^{ 0.25 |\xi_{j}|^{2} }\right] 
				\end{align*}
				where the third line follows from the Markov inequality and the fourth from Jensen inequality. The  result follows from the fact that $q \int_{0}^{\infty} u^{q-1} e^{- 0.25 u } d u \leq C < \infty$ and $|\xi_{j}|^{2} \sim \chi^{2}$ and  $\sum_{j=1}^{d} c_{j}(A) = 1$. 
			\end{proof}

			\bigskip 
			
			\begin{proof}[Proof of Theorem \ref{thm:weak-norm-ori}]
				
				Firs note that we can always write $\mathbb{V}_{n} \equiv n^{-1} \sum_{i=1}^{n} V_{i,n} $ with $V_{i,n} \sim i.i.d.-N(0,\Sigma_{n})$. 
				
				The strategy of proof consists of applying the results in  Theorem \ref{thm:LFQ-bound} and Lemma \ref{lem:LFQ-bound}, with $A_{i} = n^{-1/2} Z_{i}$ and $B_{i} = n^{-1/2} V_{i,n}$. Observe that $E[A_{i}A_{i}^{T}] = E[B_{i}B_{i}^{T}] = \Sigma_{n}$.

				\bigskip

				\textsc{The term $\mathbf{S}_{n}$.} For this case $\sum_{i=1}^{n} E[(||B_{i}||_{2} )^{4}] =  n^{-2} \sum_{i=1}^{n} E[||V_{i,n}||^{4}_{2}] = n^{-1}  E[||V_{1,n}||^{4}_{2}]$ and $\sum_{i=1}^{n} E[(||A_{i}||_{2} )^{4}] = n^{-2} \sum_{i=1}^{n} E[ ||Z_{i}||^{4}_{2} ] = n^{-1} E[||Z_{1}||^{4}_{2}]$. Therefore, $\mathbf{S}_{1,n}$ in Theorem \ref{thm:LFQ-bound} is bounded above (up to a constant) by $h^{-2} n^{-1}  \left(  E[ ||Z_{1}||^{4}_{2} ] + E[ ||V_{1,n}||^{4}_{2} ] \right)$, and by Lemma \ref{lem:mom-Gau}, this implies that\begin{align*}
				\mathbf{S}_{1,n} \precsim h^{-2} n^{-1}  \left(  E[ ||Z_{1}||^{4}_{2} ] + (tr\{ \Sigma_{n} \})^{2} \right)
				\end{align*}
				both terms are of order $o(h^{-2})$ under Assumption \ref{ass:data-Z}(ii). 
				
				Observe that in this case $E[S_{j}S_{j}^{T}] =n^{-1} \Sigma_{n}$ and thus 
				\begin{align*}
				\mathbf{S}_{3,n} \precsim & h^{-2} \sqrt{ tr\{ \Sigma_{n}  \} } n^{-3/2} \sum_{i=1}^{n} ( E[ ||Z_{i}||^{3}_{2} ] + E[ ||V_{i,n}||^{3}_{2} ] ) \\
				= &  h^{-2} \sqrt{ tr\{ \Sigma_{n}  \} } n^{-1/2} ( E[ ||Z_{1}||^{3}_{2} ] + E[ ||V_{1,n}||^{3}_{2} ]) .
				\end{align*}
				
				By Lemma \ref{lem:mom-Gau}, $E[ ||V_{1,n}||^{3}_{2} ] = (tr \{ \Sigma_{n}  \})^{3/2} $. Thus, by Assumption \ref{ass:data-Z}(i), $\mathbf{S}_{2,n}$ is of order $o(h^{-2})$. 
				
				We thus have established that $\mathbf{S}_{n}$ in Theorem \ref{thm:LFQ-bound} vanishes. We now establish that $\mathbf{R}_{n}$ also vanishes.

				%================
				%Typically: $\approx n^{-1/2} d(n)^{1/2} d(n)^{3/2}$
				%
				%================

				\bigskip

				\textsc{The remainder terms,  $\mathbf{R}_{n}$.} To bound the remainder term in the expression of Theorem \ref{thm:LFQ-bound} we use Lemma \ref{lem:LFQ-bound}, $L_{2}(f) = h^{-2}$ and also set $q=\gamma$. Observe that $\left(  tr \left\{  \sum_{j=1}^{n}  E \left[  \left( S_{j}^{T} S_{j}    \right) \right] \right\} \right)^{1+0.5q}  = \left(  tr \left\{  \Sigma_{n} \right\} \right)^{1+0.5q}$ . Also,\begin{align*}
					\sum_{i=1}^{n} E \left[ \left(  ||B_{i}||_{2}   \right)^{2+q}  \right] = n^{-0.5q} E[||V_{1}||^{2+q}_{2} ] \precsim n^{-0.5q} (tr\{ \Sigma_{n} \})^{1+0.5q}
				\end{align*}
				by Lemma \ref{lem:mom-Gau}. Therefore,\begin{align*}
					\sum_{i=1}^{n} E \left[ \left(  \mathbb{S}_{i:n}^{T} B_{i}   \right)^{2+q}  \right] \precsim &  n^{-0.5q} (tr\{ \Sigma_{n} \})^{1+0.5q} \max \left\{  (tr\{ \Sigma_{n} \})^{1+0.5q} , \sum_{j=1}^{n} E[||S_{j}||^{2+q}_{2} ]        \right\}.
				\end{align*}
				
					Observe that\begin{align*}
						\sum_{j=1}^{n} E \left[ \left(  ||S_{j}||_{2}   \right)^{2+q}  \right] \precsim n^{-(1+0.5q)} \left( \sum_{j=1}^{i-1} E \left[ \left(  ||Z_{j}||_{2}   \right)^{2+q}  \right] +  (n-i) tr\{ \Sigma_{n} \}^{1+0.5q} \right)
					\end{align*}
					 by Lemma \ref{lem:mom-Gau}. Under Assumption \ref{ass:data-Z}(ii),
					\begin{align*}
					\sum_{j=1}^{n} E \left[ \left(  ||S_{j}||_{2}   \right)^{2+q}  \right] \precsim & n^{-(1+0.5q)} \left( i  E \left[ \left(  ||Z_{1}||_{2}   \right)^{2+q}  \right] +  (n-i) tr\{ \Sigma_{n} \}^{1+0.5q} \right) \\
					\leq  & n^{-(0.5q)} \left(  E \left[ \left(  ||Z_{1}||_{2}   \right)^{2+q}  \right] +   tr\{ \Sigma_{n} \}^{1+0.5q} \right) \rightarrow 0,~as~n \rightarrow 0
					\end{align*}
					because, $n^{-(0.5q)} tr\{ \Sigma_{n} \}^{1+0.5q}  = \left( n^{-1/2} tr\{ \Sigma_{n} \}^{0.5+1/q}  \right)^{q}$ and with $q=\gamma>2$ is implied by Assumption \ref{ass:data-Z}(ii); and due to Jensen inequality\\ $n^{-(0.5q)} E \left[ \left(  ||Z_{1}||_{2}   \right)^{2+q}  \right]  \leq \sqrt{n^{-q} E \left[ \left(  ||Z_{1}||_{2}   \right)^{4+2q}  \right] }$ which vanishes for $q = \gamma$. 
%					Analogous calculations hold for $\sum_{j=1}^{n} E \left[ \left(  ||S_{j}||_{2}   \right)^{2+q}  \right] $.
					
					Also, by Assumption \ref{ass:data-Z}(ii), $n^{-(0.5q)} (tr\{ \Sigma_{n} \})^{2+q} \rightarrow 0$ as $ n \rightarrow \infty$. Finally, note that, by Lemma \ref{lem:mom-Gau},
					 $\sum_{i=1}^{n} E[(||B_{i}||_{2} )^{4+2q}] \precsim n^{-(2+q)} \sum_{i=1}^{n} E[||V_{i,n}||^{4+2q}_{2} ] \precsim n^{-(1+q)} (tr\{ \Sigma_{n} \})^{2+q}$.
%					 $\sum_{i=1}^{n} E[(||B_{i}||_{2} )^{4+2q}] \precsim n^{-(2+q)} \sum_{i=1}^{n} E[||V_{i,n}||^{4+2q}_{2} ] = n^{-(1+q)} E[||V_{1,n}||^{4+2q}_{2} ] \precsim n^{-(1+q)} (tr\{ \Sigma_{n} \})^{2+q}$. 
					 By Assumption \ref{ass:data-Z}(ii) and the previous calculations, $n^{-(1+q)} (tr\{ \Sigma_{n} \})^{2+q} = o(1)$. Similarly,\\ $\sum_{i=1}^{n} E[(||A_{i}||_{2} )^{4+2q}] \precsim n^{-(2+q)} \sum_{i=1}^{n} E[||Z_{i}||^{4+2q}_{2} ] = n^{-(1+q)} E[||Z_{1}||^{4+2q}_{2} ] = o(1)$ by Assumption \ref{ass:data-Z}(ii).\\

					We have established that the remainder term $\mathbf{R}_{n}$ in Theorem \ref{thm:LFQ-bound} vanishes, and thus the desired result follows.	
				\end{proof}

	\section{Proofs of Lemmas in Section \ref{sec:proof-boot}}
	\label{app:lemmas-main}

	In order to prove the lemmas in Section \ref{sec:proof-boot} we need the following lemmas. 
	
	\subsection{Supplementary Lemmas}
	
	Let for any $t \in \mathbb{R} $, $\delta> 0$, $ n \in \mathbb{N}$, and $h > 0$
		\begin{align*}
		\mathcal{P}_{t,\delta,h}(||x||^{2}_{2}) = \int p_{t,\delta}(||x||^{2}_{2} + h z) \phi(z)dz,~\forall x \in \mathbb{R}^{d(n)}
		\end{align*}
		where $\mathbb{R} \ni u \mapsto p_{t,\delta}(u) = 1\{ u \geq t\} +
		\frac{u-t+\delta}{\delta} 1\{ u \in (t-\delta,t)\}$ and $\phi$ is the standard Gaussian pdf.
		
		The next three lemmas show that we can use $\mathcal{P}_{t,\delta,h}(\cdot)$ to approximate the indicator function $1\{ \cdot \geq t  \}$ in expectation for the variables $||\sqrt{n} \mathbb{Z}^{\ast}_{n}||_{2}$, $||\sqrt{n} \mathbb{V}_{n}||_{2}$ and $||\sqrt{n} \mathbb{Z}_{n}||_{2}$, respectively.
		
		\begin{lemma}\label{lem:Zn-strong-weak} %[{lem:Zn-strong-weak}]
			For any $\varepsilon \in (0,1)$, $\delta > 0$ and $n \in \mathbb{N}$, there exists $ h(\delta,\varepsilon) = \frac{\delta}{\Phi^{-1}(\varepsilon)}$ such that for all $h \leq  h(\delta,\varepsilon)$:
			
			(i) 
			\begin{align}
			E_{\mathbf{P}^{\ast}_{n}} \left[ 1\{ ||\sqrt{n} \mathbb{Z}^{\ast}_{n}||^{2}_{2} \geq t \}  |Z^{n} \right] \leq \frac{1}{1-\varepsilon} E_{\mathbf{P}^{\ast}_{n}} \left[ \mathcal{P}_{t-\delta,\delta,h}(||\sqrt{n} \mathbb{Z}^{\ast}_{n}||^{2}_{2}) |Z^{n} \right].
			\end{align}
			
			(ii)
			\begin{align}
			E_{\mathbf{P}^{\ast}_{n}} \left[ 1\{ ||\sqrt{n} \mathbb{Z}^{\ast}_{n}||^{2}_{2} \geq t \}  |Z^{n} \right] \geq \frac{1}{1-\varepsilon} E_{\mathbf{P}^{\ast}_{n}} \left[ \mathcal{P}_{t+2\delta,\delta,h}(||\sqrt{n} \mathbb{Z}^{\ast}_{n}||^{2}_{2}) |Z^{n} \right] - \frac{\varepsilon}{1-\varepsilon}.
			\end{align}
		\end{lemma}
		
			\begin{lemma}\label{lem:Vn-strong-weak} %[{lem:Vn-strong-weak}]
				For any $\varepsilon \in (0,1)$, $\delta > 0$ and $n \in \mathbb{N}$, there exists $ h(\delta,\varepsilon) = \frac{\delta}{\Phi^{-1}(\varepsilon)}$ such that for all $h \leq  h(\delta,\varepsilon)$:
				
				(i) 
				\begin{align}
				E_{\boldsymbol{\Phi}_{n}} \left[ 1\{ ||\sqrt{n}\mathbb{V}_{n}||^{2}_{2} \geq t \}   \right] \leq \frac{1}{1-\varepsilon} E_{\boldsymbol{\Phi}_{n}} \left[ \mathcal{P}_{t-\delta,\delta,h}(||\sqrt{n} \mathbb{V}_{n}||^{2}_{2})  \right].
				\end{align}
				
				(ii)
				\begin{align}
				E_{\boldsymbol{\Phi}_{n}} \left[ 1\{ ||\sqrt{n}\mathbb{V}_{n}||^{2}_{2} \geq t \}   \right] \geq \frac{1}{1-\varepsilon} E_{\boldsymbol{\Phi}_{n}} \left[ \mathcal{P}_{t+2\delta,\delta,h}(||\sqrt{n} \mathbb{V}_{n}||^{2}_{2})  \right] - \frac{\varepsilon}{1-\varepsilon}.
				\end{align}
			\end{lemma}
			
				\begin{lemma}\label{lem:ZZ-strong-weak} %[{lem:ZZ-strong-weak}]
					For any $\varepsilon \in (0,1) $, $\delta > 0$ and $n \in \mathbb{N}$, there exists $ h(\delta,\varepsilon) = \frac{\delta}{\Phi^{-1}(\varepsilon)}$ such that for all $h \leq  h(\delta,\varepsilon)$:
					
					(i) 
					\begin{align}
					E_{\mathbf{P}_{n}} \left[ 1\{ ||\sqrt{n} \mathbb{Z}_{n}||^{2}_{2} \geq t \}   \right] \leq \frac{1}{1-\varepsilon} E_{\mathbf{P}_{n}} \left[ \mathcal{P}_{t-\delta,\delta,h}(||\sqrt{n} \mathbb{Z}_{n}||^{2}_{2})  \right] .
					\end{align}
					
					(ii)
					\begin{align}
					E_{\mathbf{P}_{n}} \left[ 1\{ ||\sqrt{n} \mathbb{Z}_{n}||^{2}_{2} \geq t \}   \right] \geq \frac{1}{1-\varepsilon} E_{\mathbf{P}_{n}} \left[ \mathcal{P}_{t+2\delta,\delta,h}(||\sqrt{n} \mathbb{Z}_{n}||^{2}_{2})  \right] - \frac{\varepsilon}{1-\varepsilon}.
					\end{align}
				\end{lemma}
				
				\begin{lemma}\label{lem:anticoncentration-V} %[{lem:anticoncentration-V}]
					Suppose Assumption \ref{ass:data-Z}(i) holds. For any $\varepsilon>0$, there exists a $N(\varepsilon)$ and $\gamma(\varepsilon)$ such that for all $\gamma \leq  \gamma(\varepsilon)$ and all $ n \geq N(\varepsilon)$:
					\begin{align}
					\sup_{t} \boldsymbol{\Phi}_{n} \left(   | ||\sqrt{n}\mathbb{V}_{n}||^{2}_{2} - t | \leq  \sqrt{ tr\{ \Sigma_{n}^{2} \} } \gamma  \right) \leq \varepsilon .
					\end{align}
				\end{lemma}
				
				\begin{remark}\label{rem:anticon-V} %[{rem:anticon-V}]
					It is easy to see that from this lemma it follows that: For any $\varepsilon>0$, there exists a $N(\varepsilon)$ and $\gamma(\varepsilon)$ such that for all $\gamma \leq  \gamma(\varepsilon)$ and all $ n \geq N(\varepsilon)$:
					\begin{align}
					\boldsymbol{\Phi}_{n} \left(  ||\sqrt{n}\mathbb{V}_{n}||^{2}_{2} \geq t \right) \leq \varepsilon + \boldsymbol{\Phi}_{n}\left(  ||\sqrt{n}\mathbb{V}_{n}||^{2}_{2} \geq t + \sqrt{ tr\{ \Sigma_{n}^{2} \} }  \gamma  \right)
					\end{align}
					for all $t \geq 0$.
				\end{remark}
				
					\begin{proof}[Proof of Lemma \ref{lem:Zn-strong-weak}]

						\textbf{Part (i)} By definition of $\mathcal{P}_{t,\delta,h}$, for any $||x||^{2}_{2} \geq t+\delta$
						\begin{align*}
						\mathcal{P}_{t,\delta,h}(||x||^{2}_{2}) \geq & \int 1\{ z: ||x||^{2}_{2} + hz \geq  t  \}\phi(z)dz \geq \int 1\{ z: hz \geq  -\delta  \}\phi(z)dz \\
						 & = 1- \Phi(-\delta/h).
						\end{align*}
						Thus, for any $h \leq \frac{\delta}{\Phi^{-1}(\varepsilon)} \equiv h(\delta,\varepsilon)$, $\mathcal{P}_{t,\delta,h}(||x||^{2}_{2}) \geq (1-\varepsilon)1\{ ||x||^{2}_{2} \geq t+ \delta  \}$. Thus 	\begin{align*}
						E_{\mathbf{P}^{\ast}_{n}} \left[ 1\{ ||\sqrt{n} \mathbb{Z}^{\ast}_{n}||^{2}_{2} \geq t \}  | Z^{n} \right] \leq \frac{1}{1-\varepsilon} E_{\mathbf{P}^{\ast}_{n}} \left[ \mathcal{P}_{t-\delta,\delta,h}(||\sqrt{n} \mathbb{Z}^{\ast}_{n}||^{2}_{2}) |Z^{n}  \right]
						\end{align*}
						for any $h \leq  h(\delta,\varepsilon)$.\\
						
						\textbf{Part (ii)} Observe that for any $x : ||x||^{2}_{2} \leq t- 2\delta $,
						\begin{align*}
						\mathcal{P}_{t,\delta,h}(||x||^{2}_{2}) \leq \int 1\{ z : ||x||^{2}_{2} +hz) \geq t-\delta  \} \phi(z)dz \leq  \int 1\{ z :  hz \geq \delta  \} \phi(z)dz  .
						\end{align*}
						Thus $	\mathcal{P}_{t,\delta,h}(||x||^{2}_{2}) \leq  \varepsilon $ for any $x : ||x||^{2}_{2} \leq t- 2\delta $ and $h \leq  h(\delta,\varepsilon)$. Thus, for all $x \in \mathbb{R}^{d}$, $\mathcal{P}_{t,\delta,h}(||x||^{2}_{2})  \leq (1-\varepsilon)1\{ ||x||^{2}_{2} \geq t-2 \delta \} + \varepsilon $. The result follows by taken expectations at both sides. 
					\end{proof}

					\begin{proof}[Proof of Lemma \ref{lem:Vn-strong-weak}]
						The proof is identical to that of Lemma \ref{lem:Zn-strong-weak} and will be omitted.
					\end{proof}

					\begin{proof}[Proof of Lemma \ref{lem:ZZ-strong-weak}]
						The proof is identical to that of Lemma \ref{lem:Zn-strong-weak} and will be omitted.
					\end{proof}
					
					\begin{proof}[Proof of Lemma \ref{lem:anticoncentration-V}]
						Observe that $\xi_{n} \equiv \sqrt{n} \mathbb{V}_{n} \sim N(0,\Sigma_{n})$ (recall $\Sigma_{n} = E[ Z_{1,n}Z_{1,n}^{T}]$). Note that
						\begin{align*}
						\xi_{n}^{T} \xi_{n} = (\Sigma_{n}^{-1/2}\xi_{n})^{T}\Sigma_{n}(\Sigma_{n}^{-1/2}\xi_{n}) = & (U_{n}\Sigma_{n}^{-1/2}\xi_{n})^{T}\Lambda_{n}(U_{n}\Sigma_{n}^{-1/2}\xi_{n})\\
						\equiv & (\zeta_{n})^{T}\Lambda_{n}(\zeta_{n}) = \sum_{l=1}^{d(n)} \lambda_{l} \zeta_{l,n}^{2}
						\end{align*}
						where the third inequality follows from the diagonalization of $\Sigma_{n}$, where $\Lambda_{n}$ is a diagonal matrix of eigenvalues and $U_{n}$ is an unitary matrix. Observe that $\zeta_{n} = U_{n}\Sigma_{n}^{-1/2}\xi_{n} \sim N(0,I_{d(n)})$ and thus its components are iid standard Gaussian, so $\zeta_{l}^{2} \sim \chi^{2}_{1}$ and $\lambda_{l} \zeta_{l}^{2} \sim \Gamma(1/2,2\lambda_{l})$. Moreover, it is easy to see that
						\begin{align*}
						E[\lambda_{l} \zeta_{l,n}^{2}] = \lambda_{l}~and~Var(\lambda_{l} \zeta_{l,n}^{2}) = 2 \lambda^{2}_{l}
						\end{align*}
						which implies that $Var(\sum_{l=1}^{d(n)} \lambda_{l} \zeta^{2}_{l,n} ) = 2 tr\{ \Sigma^{2}_{n}  \}$. Also, $E[|\lambda_{l} \zeta_{l,n}^{2}|^{3}] = \lambda^{3}_{l} E[|\zeta_{l,n}|^{6}] \leq C \left(\lambda_{max}(\Sigma_{n}) \right)^{3} $ where $\lambda_{max}(A)$ is the largest eigen value of  a matrix $A$.
						
						If $d(n) \leq d < \infty$, the proof follows from the fact that $\Gamma(1/2,2\lambda_{l})$ does not have mass points and is straight forward to show that the statement holds for any $n$.
						
						Suppose that $d(n) \rightarrow \infty$ as $n \rightarrow \infty$. \footnote{The relevant cases for us are: (i) $d(n) \leq d <\infty$ or (ii) $d(n) \uparrow \infty$, that is why we implicitly assume the limit of $(d(n))_{n}$ exist.} Therefore,
						\begin{align*}
						& \sup_{t} \boldsymbol{\Phi}_{n} \left(   | ||\sqrt{n}\mathbb{V}_{n}||^{2}_{2} - t | \leq  \sqrt{ tr\{ \Sigma_{n}^{2} \} } \gamma   \right) \\
						=& \sup_{t} \boldsymbol{\Phi}_{n} \left(   | \frac{ ||\xi_{n}||^{2}_{2}}{\sqrt{ 2 tr\{ \Sigma_{n}^{2} \} }} - \frac{t}{\sqrt{2 tr\{ \Sigma_{n}^{2} \} }} | \leq   \gamma /  \sqrt{2}  \right)\\
						=&  \sup_{t'} \boldsymbol{\Phi}_{n} \left(   | \frac{ ||\xi_{n}||^{2}_{2}}{\sqrt{ 2 tr\{ \Sigma_{n}^{2} \} }} - t' | \leq   \gamma /  \sqrt{2}  \right)\\
						= & \sup_{t'} \boldsymbol{\Phi}_{n} \left(   | \frac{\sum_{l=1}^{d(n)} \lambda_{l} (\zeta_{l,n}^{2}-1)}{\sqrt{ 2 tr\{ \Sigma_{n}^{2} \} }} - t' + tr\{\Sigma_{n}\} | \leq   \gamma /  \sqrt{2}  \right)\\
						= & \sup_{t''} \boldsymbol{\Phi}_{n} \left(   | \frac{\sum_{l=1}^{d(n)} \lambda_{l} (\zeta_{l,n}^{2}-1)}{\sqrt{ 2 tr\{ \Sigma_{n}^{2} \} }} - t'' | \leq   \gamma /  \sqrt{2}  \right).
						\end{align*}
%						where the second line and fourth line follow from the fact that if $t \in \mathbb{R}$, then $\frac{t}{\sqrt{ 2 tr\{ \Sigma_{n}^{2} \} }} \in \mathbb{R}$.
						
						Then, by Berry-Essen bound (Theorem 2, p. 544 feller \cite{FELLER_book}).
						\begin{align*}
						\sup_{t} \left| \boldsymbol{\Phi}_{n} \left(    \frac{\sum_{l=1}^{d(n)} \lambda_{l} (\zeta_{l,n}^{2}-1)}{\sqrt{ 2 tr\{ \Sigma_{n}^{2} \} }} \leq t'  \right) - \Phi(t') \right| \leq 6 C \frac{\sum_{l=1}^{d(n)} \lambda^{3}_{l} }{\left( 2 tr\{ \Sigma_{n}^{2} \} \right)^{3/2}}
						\end{align*}
						where $\Phi$ is the standard Gaussian cdf. Since  $\frac{\sum_{l=1}^{d(n)} \lambda^{3}_{l} }{\left( 2 tr\{ \Sigma_{n}^{2} \} \right)^{3/2}} = \frac{ tr\{ \Sigma_{n}^{3}  \} }{    \left( 2 tr\{ \Sigma_{n}^{2} \} \right)^{3/2} }$, by Assumption \ref{ass:data-Z}(i), for any $\varepsilon>0$, there exists a $N(\varepsilon)$ such that $\frac{tr\{ \Sigma_{n}^{3} \}}{\left(tr\{ \Sigma_{n}^{2} \} \right)^{3/2}} < 0.5 \varepsilon$ for all $n \geq N(\varepsilon)$. Thus,
						\begin{align*}
						\sup_{t \in \mathbb{R}}  \boldsymbol{\Phi}_{n} \left( | ||\xi_{n}||^{2}_{2} - t| \leq \sqrt{ tr\{ \Sigma_{n}^{2} \} } \gamma \right) = &\sup_{t \in \mathbb{R}}  \boldsymbol{\Phi}_{n} \left(\sqrt{ tr\{ \Sigma_{n}^{2} \} } \gamma - t \leq ||\xi_{n}||^{2}_{2}  \leq t + \sqrt{ tr\{ \Sigma_{n}^{2} \} } \gamma \right)\\
						\leq &  \sup_{t \in \mathbb{R}} \left|\Phi \left(t+ \gamma/\sqrt{2} \right) - \Phi \left(t-\gamma/\sqrt{2} \right) \right| + 0.5\varepsilon.
						\end{align*}
						
						Since for any $\varepsilon>0$, there exists a $\gamma(\varepsilon)$ such that $\left|\Phi \left(t+ \gamma/\sqrt{2} \right) - \Phi \left(t-\gamma/\sqrt{2} \right) \right| < 0.5 \varepsilon$, the desired result follows.
					\end{proof}

				\subsection{Proofs of Lemmas in Section \ref{sec:proof-boot}}
	
		\begin{proof}[Proof of Lemma \ref{lem:Delta-P-Pr}]
			The proof is analogous to that of Lemma \ref{lem:Delta-Pwz-Pr} and will not be repeated here.
		\end{proof}
	
		\begin{proof}[Proof of Lemma \ref{lem:Delta-Pwz-Pr}]
			Throughout the proof, let $\delta_{n} \equiv \sqrt{tr\{ \Sigma^{2}_{n} \}} \gamma(\varepsilon)$, where $\gamma(\varepsilon)$ as in Lemma \ref{lem:anticoncentration-V}. By remark \ref{rem:anticon-V} (applied thrice), 
			\begin{align}
			E_{\boldsymbol{\Phi}_{n}} \left[ 1\{ ||\sqrt{n}\mathbb{V}_{n}||^{2}_{2} \geq t  \}   \right] \geq 	E_{\boldsymbol{\Phi}_{n}} \left[ 1\{ ||\sqrt{n}\mathbb{V}_{n}||^{2}_{2} \geq t - 3 \delta_{n} \}   \right] - 3\varepsilon
			\end{align}
			for all $n \geq N(\varepsilon)$. By Lemma \ref{lem:Vn-strong-weak}(ii),
			\begin{align}
			E_{\boldsymbol{\Phi}_{n}} \left[ 1\{ ||\sqrt{n}\mathbb{V}_{n}||^{2}_{2} \geq t  \}   \right] \geq \frac{1}{1-\varepsilon} 	E_{\boldsymbol{\Phi}_{n}} \left[ \mathcal{P}_{t-\delta_{n},\delta_{n},h} ( ||\sqrt{n}\mathbb{V}_{n}||^{2}_{2} )   \right] - \frac{\varepsilon}{1-\varepsilon} - 3\varepsilon
			\end{align}
			for all $h \leq h(\varepsilon,\delta_{n})$ and all $n \geq N(\varepsilon)$. By Lemma \ref{lem:Zn-strong-weak}(i), for all $h \leq h(\varepsilon,\delta_{n})$
			\begin{align}
			E_{\mathbf{P}^{\ast}_{n}} \left[ 1\{ ||\sqrt{n} \mathbb{Z}^{\ast}_{n}||^{2}_{2} \geq t \}  |Z^{n} \right] \leq \frac{1}{1-\varepsilon} E_{\mathbf{P}^{\ast}_{n}} \left[ \mathcal{P}_{t-\delta_{n},\delta_{n},h}(||\sqrt{n} \mathbb{Z}^{\ast}_{n}||^{2}_{2}) |Z^{n} \right].
			\end{align}
			
			Hence, for all $h \leq h(\varepsilon,\delta_{n})$ and all $n \geq N(\varepsilon)$,
			\begin{align}\notag
			& E_{\mathbf{P}^{\ast}_{n}} \left[ 1\{ ||\sqrt{n} \mathbb{Z}^{\ast}_{n}||^{2}_{2} \geq t \}  |Z^{n} \right] - E_{\boldsymbol{\Phi}_{n}} \left[ 1\{ ||\sqrt{n}\mathbb{V}_{n}||^{2}_{2} \geq t \}   \right]\\ \notag
			\leq & \frac{1}{1-\varepsilon} \left(E_{\mathbf{P}^{\ast}_{n}} \left[ \mathcal{P}_{t-\delta_{n},\delta_{n},h}(||\sqrt{n} \mathbb{Z}^{\ast}_{n}||^{2}_{2}) |Z^{n} \right]  - E_{\boldsymbol{\Phi}_{n}} \left[ \mathcal{P}_{t-\delta_{n},\delta_{n},h} ( ||\sqrt{n}\mathbb{V}_{n}||^{2}_{2} )   \right] \right) \\
			&  + \frac{\varepsilon}{1-\varepsilon} + 3\varepsilon. \label{eqn:bound-1}
			\end{align}
			
			Similarly, by Lemma \ref{lem:Zn-strong-weak}(ii), for all $h \leq h(\varepsilon,\delta_{n})$
			\begin{align}
			E_{\mathbf{P}^{\ast}_{n}} \left[ 1\{ ||\sqrt{n} \mathbb{Z}^{\ast}_{n}||^{2}_{2} \geq t \}  |Z^{n} \right] \geq \frac{1}{1-\varepsilon} E_{\mathbf{P}^{\ast}_{n}} \left[ \mathcal{P}_{t+2\delta_{n},\delta_{n},h}(||\sqrt{n} \mathbb{Z}^{\ast}_{n}||^{2}_{2}) |Z^{n} \right] - \frac{\varepsilon}{1-\varepsilon}.
			\end{align}
			By Remark \ref{rem:anticon-V} (applied thrice), 
			\begin{align}
			E_{\boldsymbol{\Phi}_{n}} \left[ 1\{ ||\sqrt{n}\mathbb{V}_{n}||^{2}_{2} \geq t  \}   \right] \leq E_{\boldsymbol{\Phi}_{n}} \left[ 1\{ ||\sqrt{n}\mathbb{V}_{n}||^{2}_{2} \geq t + 3 \delta_{n} \}   \right] + 3\varepsilon
			\end{align}
			for all $n \geq N(\varepsilon)$. By Lemma \ref{lem:Vn-strong-weak}(ii),
			\begin{align}
			E_{\boldsymbol{\Phi}_{n}} \left[ 1\{ ||\sqrt{n}\mathbb{V}_{n}||^{2}_{2} \geq t  \}   \right] \leq \frac{1}{1-\varepsilon} 	E_{\boldsymbol{\Phi}_{n}} \left[ \mathcal{P}_{t+2\delta_{n},\delta_{n},h} ( ||\sqrt{n}\mathbb{V}_{n}||^{2}_{2} )   \right] + 3\varepsilon
			\end{align}
			for all $h \leq h(\varepsilon,\delta_{n})$ and all $n \geq N(\varepsilon)$.
			
			Hence,\begin{align}\notag
			& E_{\mathbf{P}^{\ast}_{n}} \left[ 1\{ ||\sqrt{n} \mathbb{Z}^{\ast}_{n}||^{2}_{2} \geq t \}  |Z^{n} \right] - E_{\boldsymbol{\Phi}_{n}} \left[ 1\{ ||\sqrt{n}\mathbb{V}_{n}||^{2}_{2} \geq t \}   \right]\\ \notag
			\geq & \frac{1}{1-\varepsilon} \left(E_{\mathbf{P}^{\ast}_{n}} \left[ \mathcal{P}_{t+2\delta_{n},\delta_{n},h}(||\sqrt{n} \mathbb{Z}^{\ast}_{n}||^{2}_{2}) |Z^{n} \right]  - E_{\boldsymbol{\Phi}_{n}} \left[ \mathcal{P}_{t+2\delta_{n},\delta_{n},h} ( ||\sqrt{n}\mathbb{V}_{n}||^{2}_{2} )   \right] \right) \\
			&  - \frac{\varepsilon}{1-\varepsilon} - 3\varepsilon. \label{eqn:bound-2}
			\end{align}
			
			By displays \ref{eqn:bound-1} and \ref{eqn:bound-2}, %and the fact that $|a|-|b| \leq |a-b|$, it follows that: for all $h \leq h(\varepsilon,\delta_{n})$ and all $n \geq N(\varepsilon)$,
%			\begin{align}\notag
%			& |E_{\mathbf{P}^{\ast}_{n}} \left[ 1\{ ||\sqrt{n} \mathbb{Z}^{\ast}_{n}||^{2}_{2} \geq t \}  |Z^{n} \right] - E_{\boldsymbol{\Phi}_{n}} \left[ 1\{ ||\sqrt{n}\mathbb{V}_{n}||^{2}_{2} \geq t  \}   \right] | \\ \notag
%			\leq & \frac{1}{1-\varepsilon} \left | E_{\mathbf{P}^{\ast}_{n}} \left[ \mathcal{P}_{t+2\delta_{n},\delta_{n},h}(||\sqrt{n} \mathbb{Z}^{\ast}_{n}||^{2}_{2}) |Z^{n} \right]  - E_{\boldsymbol{\Phi}_{n}} \left[ \mathcal{P}_{t+2\delta_{n},\delta_{n},h} ( ||\sqrt{n}\mathbb{V}_{n}||^{2}_{2} )   \right] \right| \\
%			& + \frac{\varepsilon}{1-\varepsilon} + 3\varepsilon.
%			\end{align}
%			
%			Thus,
 in order to obtain the desired result it suffices to verify that $a \in \mathbb{R} \mapsto \mathcal{P}_{t,\delta,h}(a) \in \mathcal{C}_{h^{-1}}$. It is straight forward to check that $\mathcal{P}_{t,\delta,h} $ is three times continuously differentiable. Moreover, for any $a \in \mathbb{R}$,
			\begin{align*}
			|\partial \mathcal{P}_{t,\delta,h} (a)| \leq h^{-1}.
			\end{align*}
			
			To show this expression, observe that by the Dominated Convergence Theorem, for any $a \in \mathbb{R}$,
			\begin{align*}
			|\partial \mathcal{P}_{t,\delta,h} (a)| = & h^{-1} \left| \int p_{t,\delta}(u) (u-a) h^{-2} \phi((u - a)h^{-1})du  \right| \\
			= &  h^{-1}  \int \left| u-a   \right| h^{-2} \phi((u - a)h^{-1})du\\ 
			\leq & h^{-2} \sqrt{ \int \left| u-a  \right|^{2} h^{-1} \phi((u-a)h^{-1})du    }\\
			= & h^{-1}
			\end{align*}
			where the second line follows from the fact that $0 \leq p_{t,\delta}(u) \leq 1$. Similarly calculations yield
			\begin{align*}
			|\partial^{r} \mathcal{P}_{t,\delta,h} (a)| \leq h^{-r}
			\end{align*}
			which holds uniformly in $a \in \mathbb{R}$, $\delta$, and $t$. 
		\end{proof}

	\begin{proof}[Proof of Lemma \ref{lem:slep}]

		Establishing the result is analogous to establishing a bound for $\Delta_{h^{-1}}(\mathbf{Q}^{\ast}_{n}(\cdot| Z^{n} ),\mathbf{Q}_{n})$ where $\mathbf{Q}^{\ast}_{n}(\cdot| Z^{n} )$ is $N(0,\hat{\Sigma}_{n})$ and $\mathbf{Q}_{n}$ is $N(0,\Sigma_{n})$ . Let $\tilde{\xi}_{n} \sim \mathbf{Q}^{\ast}_{n}(\cdot|Z^{n})$ and $\xi_{n} \sim \mathbf{Q}_{n}$.

		For any $x \in \mathbb{R}^{d}$, let $f(x) \equiv g(||x||^{2}_{2})$. Observe that for any $g \in \mathcal{C}_{h^{-1}}$, $\partial_{i} f(x) = g'(||x||^{2}_{2}) 2 x_{i} $ and $\partial_{ij} f(x) = g''(||x||^{2}_{2}) 4 x_{i} x_{j} +  2 g'(||x||^{2}_{2})  1\{ i=j \} $. 
		
		By the Slepian interpolation  (\cite{Rollin2013} p. 4 --- there the construction itself is slightly different, using $\sqrt{t}$ instead of $\cos(t)$ ---),
		\begin{align*}
		E_{\mathbf{Q}^{\ast}_{n}(\cdot|Z^{n}) \cdot \mathbf{Q}_{n} } \left[ f\left( \tilde{\xi}_{n} \right) - f\left( \xi_{n} \right) \right] = \sum_{j=1}^{d(n)} \int_{0}^{\pi/2} E_{\mathbf{Q}^{\ast}_{n}(\cdot|Z^{n}) \cdot \mathbf{Q}_{n}} \left[ \partial_{j} f\left( \xi_{n}(t) \right) \dot{\xi}_{[j],n}(t) \right]  dt
		\end{align*}
		where $\xi_{n}(t) = \cos(t) \xi_{n} + \sin(t) \tilde{\xi}_{n}$ and $\dot{\xi}_{[j],n}(t)$ denotes the $j$-th coordinate of $\dot{\xi}_{n}(t)$ (the same holds for $\xi_{n}$, etc). Observe that $\dot{\xi}_{[j],n}(t) = - \sin(t) \xi_{[j],n} + \cos(t) \tilde{\xi}_{[j],n}$. Hence $(\dot{\xi}_{[j],n}(t),\xi_{n}(t))$ are jointly Gaussian with mean 0 a.s.-$\mathbf{P}_{n}$, for any $t$. Hence, by Stein's Identity (\cite{Stein1981} and \cite{CCK-AOS13} Lemma H.2),
		\begin{align*}
		& E_{\mathbf{Q}^{\ast}_{n}(\cdot|Z^{n}) \cdot \mathbf{Q}_{n}} \left[ \partial_{j} f\left( \xi_{n}(t) \right) \dot{\xi}_{[j],n}(t) \right]\\
		= & \sum_{l=1}^{d(k(n))} E_{\mathbf{Q}^{\ast}_{n}(\cdot|Z^{n}) \cdot \mathbf{Q}_{n}} \left[ \partial_{jl} f\left( \xi_{n}(t) \right) \right] E_{\mathbf{Q}^{\ast}_{n}(\cdot|Z^{n}) \cdot \mathbf{Q}_{n}} \left[ \xi_{[l],n}(t)  \dot{\xi}_{[j],n}(t) \right].
		\end{align*}
		It follows that
		\begin{align*}
		E \left[ \xi_{[l],n}(t)  \dot{\xi}_{[j],n}(t) \right] = 	E \left[ (\tilde{\xi}_{[l],n}\tilde{\xi}_{[j],n} - \xi_{[l],n}\xi_{[j],n}) \right] \sin(t)\cos(t).
		\end{align*}
		
		Therefore,
		\begin{align*}
		& E_{\mathbf{Q}^{\ast}_{n}(\cdot|Z^{n}) \cdot \mathbf{Q}_{n}} \left[ f\left( \tilde{\xi}_{n} \right) - f\left( \xi_{n} \right) \right] =  \sum_{j=1}^{d(n)}  \sum_{l=1}^{d(n)} E_{\mathbf{Q}^{\ast}_{n}(\cdot|Z^{n}) \cdot \mathbf{Q}_{n}} \left[ (\tilde{\xi}_{[l],n}\tilde{\xi}_{[j],n} - \xi_{[l],n}\xi_{[j],n}) \right] \\
		& \times \int_{0}^{\pi/2} E_{\mathbf{Q}^{\ast}_{n}(\cdot|Z^{n}) \cdot \mathbf{Q}_{n}} \left[ \partial_{jl} f\left( \xi_{n}(t) \right) \right]  \sin(t)\cos(t)dt \\
		=&   \sum_{j=1}^{d(n)}  \sum_{l=1}^{d(n)} \left\{ n^{-1} \sum_{i=1}^{n} Z_{[l],i,n}Z_{[j],i,n} - \Sigma_{[j,l],n}  \right\} \\
		& \times \int_{0}^{\pi/2} E_{\mathbf{Q}^{\ast}_{n}(\cdot|Z^{n}) \cdot \mathbf{Q}_{n}} \left[ \partial_{jl}
		f\left( \xi_{n}(t) \right) \right]  \sin(t)\cos(t)dt
		\end{align*}
		where the second line follows from the fact that $\tilde{\xi}_{n}
		\sim N(0, n^{-1} \sum_{i=1}^{n} Z_{i}Z_{i}^{T}) $, under $\mathbf{Q}^{\ast}_{n}(\cdot|Z^{n})$. 
		
		Therefore, 
		\begin{align*}
		& E_{\mathbf{Q}^{\ast}_{n}(\cdot|Z^{n}) \cdot \mathbf{Q}_{n}} \left[ f\left( \tilde{\xi}_{n} \right) - f\left( \xi_{n} \right) \right] \\
		\leq &   \max_{j,l} \left| n^{-1} \sum_{i=1}^{n} Z_{[l],i}Z_{[j],i} - \Sigma_{[j,l],n}  \right|\\
		& \times \sum_{j=1}^{d(n)}  \sum_{l=1}^{d(n)} \int_{0}^{\pi/2} E_{\mathbf{Q}^{\ast}_{n}(\cdot|Z^{n}) \cdot \mathbf{Q}_{n}} \left[ |\partial_{jl}
		f\left( \xi_{n}(t) \right)|  \right]  |\sin(t)\cos(t)|dt.
		\end{align*}	
		
		Observe that, by Cauchy-Schwarz inequality and the fact that $\partial_{ij} f(x) = g''(||x||^{2}_{2}) 4 x_{i} x_{j} +  2 g'(||x||^{2}_{2})  1\{ i=j \} $
		\begin{align*}
		\sum_{j=1}^{d(n)}  \sum_{l=1}^{d(n)} E_{\mathbf{Q}^{\ast}_{n}(\cdot|Z^{n}) \cdot \mathbf{Q}_{n}} \left[ |\partial_{jl}
		f\left( \xi_{n}(t) \right)| \right] \leq & 4 h^{-2} \sum_{j=1}^{d(n)}  \sum_{l=1}^{d(n)} E_{\mathbf{Q}^{\ast}_{n}(\cdot|Z^{n}) \cdot \mathbf{Q}_{n}} \left[ |\xi_{[j],n}(t) | |\xi_{[l],n}(t) |  \right]\\
		&  + 2 h^{-1} d(n) \\
		\leq & 4 h^{-2} \left( \sum_{j=1}^{d(n)}  \sqrt{ E_{\mathbf{Q}^{\ast}_{n}(\cdot|Z^{n}) \cdot \mathbf{Q}_{n}} \left[ |\xi_{[j],n}(t) |^{2} \right] } \right)^{2} \\
		& + 2 h^{-1} d(n) \\
		\leq & 4 h^{-2} d(n)   E_{\mathbf{Q}^{\ast}_{n}(\cdot|Z^{n}) \cdot \mathbf{Q}_{n}} \left[ ||\xi_{n}(t) ||^{2}_{2} \right]   + 2 h^{-1} d(n) .
		\end{align*} 
		
		Therefore, since $||\xi_{n}(t)||^{2}_{2} \precsim \{ ||\xi_{n}||^{2}_{2} + ||\tilde{\xi}_{n}||^{2}_{2} \}$,
		\begin{align*}
		\sum_{j=1}^{d(n)}  \sum_{l=1}^{d(n)} E_{\mathbf{Q}^{\ast}_{n}(\cdot|Z^{n}) \cdot \mathbf{Q}_{n}} \left[ |\partial_{jl}
		f\left( \xi_{n}(t) \right)|  \right] \precsim & d(n) h^{-1}  \{ h^{-1} E_{\mathbf{Q}^{\ast}_{n}(\cdot|Z^{n}) \cdot \mathbf{Q}_{n}} \left[ ||\xi_{n}||^{2}_{2} + ||\tilde{\xi}_{n}||^{2}_{2}  \right] + 2 \} \\
		= & d(n) h^{-1}  \{ h^{-1} \left( tr\{ \Sigma_{n} \} + tr\{ \hat{\Sigma}_{n} \} \right) + 2 \}.
		\end{align*}
		
		The desired result from the fact that $\int_{0}^{\pi/2} |\sin(t)\cos(t)|dt < \infty$.
	\end{proof}
	
	\section{Proofs for Section \ref{sec:MST}}
	\label{app:appl}
	
	We first introduce some notation and lemmas needed in the proofs of the results in Section \ref{sec:MST} (the proofs of these lemmas are relegated to the end of the section). Let $\hat{g}^{\ast}_{n} \equiv n^{-1}\sum_{i=1}^{n} \omega_{in} g(X_{i},\hat{\theta}^{\ast}_{GMM,n})$ and $\bar{g}^{\ast}_{n} \equiv n^{-1}\sum_{i=1}^{n} \omega_{in} g(X_{i},\theta_{0})$. Let $\bar{G}^{\ast}_{n}(\theta) = n^{-1} \sum_{i=1}^{n} \omega_{in} \nabla_{\theta} g(X_{i}, \theta) \in \mathbb{R}^{d(n) \times q}$.
	
	\begin{lemma}\label{lem:GMM-1}
		Suppose Assumption \ref{ass:example-1}(ii)(iii)(iv) holds. Then:\\
		
		(1) $\sqrt{n} || \bar{g}^{\ast}_{n} ||_{2} = O_{\mathbf{P}^{\ast}_{n}(\cdot|Z^{n})}(\sqrt{n^{-1}\sum_{i=1}^{n} ||g(X_{i},\theta_{0})||^{2}_{2} })$.\\

		(2) Uniformly over $\theta \in \{ \theta \in \Theta : ||\theta - \theta_{0}||_{2} \precsim \Delta_{n} \}$ with $\Delta_{n} = o(1)$,\begin{align*}
			||\bar{G}^{\ast}_{n}(\theta)||_{2} =O_{\mathbf{P}^{\ast}_{n}(\cdot|Z^{n})}(\sqrt{d(n)} (n^{-1/2} + \Delta_{n})),~wpa1-\mathbf{P}.
		\end{align*}
	\end{lemma}
	
	Let 
	\begin{align*}
	R^{\ast}_{n}(\theta,\lambda) \equiv \lambda^{T} \left(  n^{-1}\sum_{i=1}^{n} \int_{0}^{1} s_{2}(t\lambda^{T} \omega_{in}g(X_{i},\theta) ) dt\omega^{2}_{in}g(X_{i},\theta)g(X_{i},\theta)^{T}  - s_{2}(0)\Omega \right) \lambda.
	\end{align*}
	
			\begin{lemma}\label{lem:GEL-1}
				Suppose Assumption \ref{ass:example-1}(i)(ii)(iii) holds and $d(n)^{4}/n = o(1)$. Then: \\
				
				(1) For all $\theta \in \mathcal{N}$, $\{ \lambda : || \lambda||_{2} \precsim  \sqrt{d(n)/n}  \} \subseteq  \Lambda(\theta)$.\footnote{The set $\mathcal{N}$ is the one in Assumption \ref{ass:example-1}.}\\
				
				(2)  Uniformly over $\lambda \in \{ \lambda \in \Lambda(\hat{\theta}^{\ast}_{GEL,n}) :  ||\lambda||_{2} \precsim \sqrt{d(n)/n}  \}$ and $||\theta - \theta_{0}||_{2} \precsim \Delta_{n}$ with $\Delta_{n} = o(1)$,
				\begin{align*}
				n R^{\ast}_{n}(\theta,\lambda) =  O_{\mathbf{P}^{\ast}_{n}(\cdot|Z^{n})} \left( \sqrt{d(n)} ( o(1) +  d(n)^{3/2} \Delta_{n} )    \right),~wpa1-\mathbf{P}.
				\end{align*}
				
			\end{lemma}
			
	The following lemma is a general result that provides a relationship between $O_{\mathbf{P}^{\ast}_{n}(\cdot|Z^{n})}$ (and $o_{\mathbf{P}^{\ast}_{n}(\cdot|Z^{n})}$) and $O_{\mathbf{P}}$ variables that we use throughout.
	
	\begin{lemma}\label{lem:prb-pr}
	 Let $(W_{i})_{i}$ and $(X_{i})_{i}$ be sequences of random variables such that $W_{n}$ is $(\omega_{in},Z_{i})_{i \leq n}$ measurable and $X_{n}$ is $(Z_{i})_{i\leq n}$ measurable and $X_{n} \ne 0 $ a.s.-$\mathbf{P}$. Let $(c_{n})_{n}$ be a sequence of positive real numbers. Then:
	 
	 \smallskip
	 
	 (1) If $W_{n} = O_{\mathbf{P}^{\ast}_{n}(\cdot|Z^{n})}(|X_{n}|)$ and $X_{n} = O_{\mathbf{P}}(c_{n})$, then $W_{n} = O_{\mathbf{P}^{\ast}_{n}(\cdot|Z^{n})}(c_{n})$ wpa1-$\mathbf{P}$.
	 
	 \smallskip
	 
	 (2) If $W_{n} = O_{\mathbf{P}^{\ast}_{n}(\cdot|Z^{n})}(|X_{n}|)$ and $X_{n} = o_{\mathbf{P}}(c_{n})$, then $W_{n} = o_{\mathbf{P}^{\ast}_{n}(\cdot|Z^{n})}(c_{n})$  wpa1-$\mathbf{P}$.
	 
	\end{lemma}

%	The result of Lemma \ref{lem:quad-approx} for the non-bootstrapped statistics can be found in DIN; the results regarding the bootstrap statistics although not contained in there are derived by using similar arguments. In particular, the derivations are analogous to those in the proof of Theorems 6.3 and 6.4 in DIN. 

	\begin{proof}[Proof of Lemma \ref{lem:quad-approx}]
		The proof for $\hat{T}_{GMM,n}$ is in Lemma 6.1 in DIN and also analogous to that of $\hat{T}^{\ast}_{GMM,n}$, so it will be omitted. 
		
		\medskip
		
		We now establish the result for $\hat{T}^{\ast}_{GMM,n}$. It follows that $n |(\bar{g}^{\ast}_{n})^{T} \hat{W}_{n} \bar{g}^{\ast}_{n} -  (\bar{g}^{\ast}_{n})^{T} W_{n} \bar{g}^{\ast}_{n}| \leq ||\hat{W}_{n} - W_{n}||_{2} \times || \sqrt{n} \bar{g}^{\ast}_{n}||^{2}_{2} $. By Lemma \ref{lem:GMM-1}(1),
		\begin{align*}
		n |(\bar{g}^{\ast}_{n})^{T} \hat{W}_{n} \bar{g}^{\ast}_{n} -  (\bar{g}^{\ast}_{n})^{T} W_{n} \bar{g}^{\ast}_{n}| = O_{\mathbf{P}^{\ast}_{n}(\cdot|Z^{n})} \left(  n^{-1}\sum_{i=1}^{n} ||g(X_{i},\theta_{0})||^{2}_{2} ||\hat{W}_{n} - W_{n}||_{2} \right).
		\end{align*}
		Under Assumption \ref{ass:example-W} and since $E_{\mathbf{P}}[n^{-1}\sum_{i=1}^{n} ||g(X_{i},\theta_{0})||^{2}_{2}] = tr\{ \Omega \} = O(d(n))$, it follows by the Markov inequality that $n^{-1}\sum_{i=1}^{n} ||g(X_{i},\theta_{0})||^{2}_{2} ||\hat{W}_{n} - W_{n}||_{2} = o_{\mathbf{P}}(\sqrt{d(n)})$. Thus, by Lemma \ref{lem:prb-pr},  $n |(\bar{g}^{\ast}_{n})^{T} \hat{W}_{n} \bar{g}^{\ast}_{n} -  (\bar{g}^{\ast}_{n})^{T} W_{n} \bar{g}^{\ast}_{n}| = o_{\mathbf{P}^{\ast}_{n}(\cdot|Z^{n})}(\sqrt{d(n)})$ wpa1-$\mathbf{P}$.

		Given this, it suffices to show that $n |(\hat{g}^{\ast}_{n})^{T} \hat{W}_{n} \hat{g}^{\ast}_{n} -  (\bar{g}^{\ast}_{n})^{T} \hat{W}_{n} \bar{g}^{\ast}_{n}|= o_{\mathbf{P}^{\ast}_{n}(\cdot|Z^{n})}(\sqrt{d(n)})$ wpa1-$\mathbf{P}$. Note that
		\begin{align*}
		 n |(\hat{g}^{\ast}_{n})^{T} \hat{W}_{n} \hat{g}^{\ast}_{n} -  (\bar{g}^{\ast}_{n})^{T} \hat{W}_{n} \bar{g}^{\ast}_{n}| \leq &  2 n |(\hat{g}^{\ast}_{n} - \bar{g}^{\ast}_{n})^{T} \hat{W}_{n} \bar{g}^{\ast}_{n} | + n| (\hat{g}^{\ast}_{n} - \bar{g}^{\ast}_{n})^{T} \hat{W}_{n} (\hat{g}^{\ast}_{n} - \bar{g}^{\ast}_{n})|\\
		= & 2 n |(\hat{\theta}^{\ast}_{GMM,n}-\theta_{0})^{T}(\Gamma^{\ast}_{n})^{T} \hat{W}_{n} \bar{g}^{\ast}_{n} | \\
		& + n| (\hat{\theta}^{\ast}_{GMM,n}-\theta_{0})^{T}( \Gamma^{\ast}_{n} )^{T} \hat{W}_{n} (\Gamma^{\ast}_{n}) (\hat{\theta}^{\ast}_{GMM,n}-\theta_{0}) | \\
		\equiv & T^{\ast}_{1,n} + T^{\ast}_{2,n}
		\end{align*}
		where the second line follows Assumption \ref{ass:example-1}(i) and the mean value Theorem; here $\Gamma^{\ast}_{n} \equiv \int_{0}^{1} \bar{G}^{\ast}_{n}(\hat{\theta}^{\ast}_{n}(t))dt$ with $\hat{\theta}^{\ast}_{n}(t) \equiv \theta_{0} + t (\hat{\theta}^{\ast}_{GMM,n}-\theta_{0})  $. The desired result follows by establishing that $T^{\ast}_{i,n} = o_{\mathbf{P}^{\ast}_{n}(\cdot|Z^{n})}(\sqrt{d(n)})$ wpa1-$\mathbf{P}$ for $i=1,2$. We do this next.
		
		We note that, 
		\begin{align*}
		T^{\ast}_{1,n} \precsim \sqrt{n} ||\hat{\theta}^{\ast}_{GMM,n}-\theta_{0}||_{2} || (\Gamma^{\ast}_{n})^{T} \hat{W}_{n} \sqrt{n} \bar{g}^{\ast}_{n} ||_{2}
		\end{align*}
		wpa1-$\mathbf{P}$. 
		
		By assumption $\sqrt{n}||\hat{\theta}^{\ast}_{GMM,n}-\theta_{0}||_{2} = O_{\mathbf{P}^{\ast}_{n}(\cdot|Z^{n})}(\sqrt{d(n)})$ wpa1-$\mathbf{P}$. Moreover, under Assumption \ref{ass:example-W}, $\lambda_{max}(\hat{W}_{n}) \leq C$ wpa1-$\mathbf{P}$ and thus \\ $|| (\Gamma^{\ast}_{n})^{T} \hat{W}_{n} \sqrt{n} \bar{g}^{\ast}_{n} ||_{2} \precsim || \Gamma^{\ast}_{n}||_{2} ||\sqrt{n} \bar{g}^{\ast}_{n} ||_{2}$. We can apply Lemma \ref{lem:GMM-1}(2) with $\Delta_{n} = n^{-1/2} \sqrt{d(n)} $ and obtain $|| \Gamma^{\ast}_{n}||_{2} = O_{\mathbf{P}^{\ast}_{n}(\cdot|Z^{n})}(d(n)/\sqrt{n}) $ wpa1-$\mathbf{P}$. By Lemma \ref{lem:GMM-1}, and since $E_{\mathbf{P}}[n^{-1}\sum_{i=1}^{n} ||g(X_{i},\theta_{0})||^{2}_{2}] = tr\{ \Omega \} = O(d(n))$, it follows by Lemma \ref{lem:prb-pr} that $||\sqrt{n} \bar{g}^{\ast}_{n} ||_{2} = O_{\mathbf{P}^{\ast}_{n}(\cdot|Z^{n})}(\sqrt{(d(n)})$. Thus $T^{\ast}_{1,n} = O_{\mathbf{P}^{\ast}_{n}(\cdot|Z^{n})}(\sqrt{d(n)} d(n)/\sqrt{n} \sqrt{d(n)}) =  O_{\mathbf{P}^{\ast}_{n}(\cdot|Z^{n})}(\sqrt{d(n)} \frac{d(n)^{3/2}}{\sqrt{n} })   $ wpa1-$\mathbf{P}$ since $\frac{d(n)^{3}}{ n } \rightarrow 0$ the result follows.

				Finally, by our assumption  $||\hat{\theta}^{\ast}_{GMM,n}-\theta_{0}||_{2} = O_{\mathbf{P}^{\ast}_{n}(\cdot|Z^{n})}(\sqrt{d(n)} n^{-1/2})$, Lemma \ref{lem:GMM-1}, and Assumption \ref{ass:example-W}, it follows that $T_{2,n} = O_{\mathbf{P}^{\ast}_{n}(\cdot|Z^{n})}(d(n)^{1/2}   \frac{d(n)^{5/2} }{n} ) =  o_{\mathbf{P}^{\ast}_{n}(\cdot|Z^{n})}(d(n)^{1/2})$ wpa1-$\mathbf{P}$ because $d(n)^{3}/n \rightarrow 0$.

		Therefore, we conclude that
		\begin{align*}
		n |(\hat{g}^{\ast}_{n})^{T} \hat{W}_{n} \hat{g}^{\ast}_{n} -  (\bar{g}^{\ast}_{n})^{T} W_{n} \bar{g}^{\ast}_{n}| = o_{\mathbf{P}^{\ast}_{n}(\cdot|Z^{n})}(\sqrt{d(n)}) 
		\end{align*}
		wpa1-$\mathbf{P}$.
		
		\medskip

		We now establish the result for $\hat{T}^{\ast}_{GEL,n}$. The proof for  $\hat{T}_{GEL,n}$ is completely analogous and therefore omitted. Abusing notation, we denote $\hat{g}^{\ast}_{n} \equiv n^{-1}\sum_{i=1}^{n} \omega_{in} g(X_{i},\hat{\theta}^{\ast}_{GEL,n})$. Let $s_{1}(\cdot)$ and $s_{2}(\cdot)$ denote the first and second derivatives of $s$. Define the following function\begin{align*}
		\lambda \mapsto F^{\ast}_{n}(\lambda) = s_{1}(0)  \lambda^{T}  \bar{g}^{\ast}_{n}  + 0.5 s_{2}(0) \lambda^{T} \Omega \lambda.	
		\end{align*}
		
		Since $s_{2}(0)<0$, the maximum of this function is achieved at $\lambda_{0} = - \frac{s_{1}(0)}{s_{2}(0)} \Omega^{-1} \bar{g}^{\ast}_{n}$ and $F^{\ast}_{n}(\lambda_{0}) = 0.5 \frac{(s_{1}(0))^{2}}{s_{2}(0)} (\bar{g}^{\ast}_{n})^{T} \Omega^{-1} \bar{g}^{\ast}_{n}$. By Lemma \ref{lem:GMM-1}(1) and the fact that $\Omega$ has eigenvalues uniformly bounded away from zero (Assumption \ref{ass:Omega}), $||\lambda_{0}||_{2} = O_{\mathbf{P}^{\ast}_{n}(\cdot|Z^{n})}(\sqrt{d(n)/n})$ wpa1-$\mathbf{P}$. Hence, $\lambda_{0} \in \Lambda(\hat{\theta}^{\ast}_{GEL,n})$  wpa1-$\mathbf{P}$ by Lemma \ref{lem:GEL-1}(1).
		
		By definition of $\hat{T}^{\ast}_{GEL,n} $ and the mean value Theorem
		\begin{equation*}
		\hat{T}^{\ast}_{GEL,n} \geq  2 \sum_{i=1}^{n} \left( s(\lambda^{T}\omega_{in} g(X_{i},\hat{\theta}^{\ast}_{GEL,n}) ) - s(0) \right) =  2 n F^{\ast}_{n}(\lambda) + n R^{\ast}_{n}(\hat{\theta}^{\ast}_{GEL,n},\lambda)
		\end{equation*}
		for all $ \lambda \in \Lambda(\hat{\theta}^{\ast}_{GEL,n})$ with $R^{\ast}_{n}$ defined in Lemma \ref{lem:GEL-1}.
		
		By Lemma \ref{lem:GEL-1}(2) with $\Delta_{n} = n^{-1/2} \sqrt{d(n)} $, it follows that $nR^{\ast}_{n}(\hat{\theta}^{\ast}_{GEL,n},\lambda_{0}) =  o_{\mathbf{P}^{\ast}_{n}(\cdot|Z^{n})}(d(n)^{1/2}) $ wpa1-$\mathbf{P}$ since $d(n)^{4}/n = o(1) $ by assumption. Moreover,  $\lambda_{0} \in \Lambda(\hat{\theta}^{\ast}_{GEL,n})$, so $\hat{T}^{\ast}_{GEL,n} \geq  2 n F^{\ast}_{n}(\lambda_{0})+ o_{\mathbf{P}^{\ast}_{n}(\cdot|Z^{n})}(d(n)^{1/2})$ wpa1-$\mathbf{P}$.

		By definition of $\lambda_{0}$ it also follows that $F^{\ast}_{n}(\lambda_{0}) \geq F^{\ast}_{n}(\hat{\lambda}^{\ast}_{n})$ (recall that $\hat{\lambda}^{\ast}_{n}$ is the maximizer of $\sum_{i=1}^{n} s(\lambda^{T} \omega_{in}g(X_{i},\hat{\theta}^{\ast}_{GEL,n}) ) $; see Assumption \ref{ass:lambda-GEL}).
		
		Therefore,
		\begin{align*}
		2n	F^{\ast}_{n}(\lambda_{0}) \geq 2nF^{\ast}_{n}(\hat{\lambda}^{\ast}_{n}) = \hat{T}^{\ast}_{GEL,n} - n R^{\ast}_{n}(\hat{\theta}^{\ast}_{GEL,n},\hat{\lambda}^{\ast}_{n}).
		\end{align*}
		Observe that, since $||\hat{\lambda}^{\ast}_{n}||_{2} = O_{\mathbf{P}^{\ast}_{n}(\cdot|Z^{n})}(\sqrt{d(n)/n})$ (by Assumption \ref{ass:lambda-GEL}), by Lemma \ref{lem:GEL-1}(2) with $\Delta_{n} = n^{-1/2} \sqrt{d(n)}$,  $nR^{\ast}_{n}(\hat{\theta}^{\ast}_{GEL,n},\hat{\lambda}^{\ast}_{n}) =  o_{\mathbf{P}^{\ast}_{n}(\cdot|Z^{n})}(\sqrt{d(n)}) $ wpa1-$\mathbf{P}$, and thus $2n	F^{\ast}_{n}(\lambda_{0}) \geq 2nF^{\ast}_{n}(\hat{\lambda}^{\ast}_{n}) \geq \hat{T}^{\ast}_{GEL,n} + o_{\mathbf{P}^{\ast}_{n}(\cdot|Z^{n})}(\sqrt{d(n)})$ wpa1-$\mathbf{P}$.
		
		Therefore, it follows that
		\begin{align*}
		\hat{T}^{\ast}_{GEL,n} = & 2nF^{\ast}_{n}(\lambda_{0}) + o_{\mathbf{P}^{\ast}_{n}(\cdot|Z^{n})}(d(n)^{1/2}) \\
		= & \frac{(s_{1}(0))^{2}}{s_{2}(0)} n(\bar{g}^{\ast}_{n})^{T} \Omega^{-1} \bar{g}^{\ast}_{n} + o_{\mathbf{P}^{\ast}_{n}(\cdot|Z^{n})}(\sqrt{d(n)})
		\end{align*}
		wpa1-$\mathbf{P}$.
		\end{proof}
		
		Throughout the proof, for any matrix $M$, let $||t||^{2}_{M} \equiv t' M t$. 
		
		\begin{proof}[Proof of Theorem \ref{thm:MST}]
			We only establish the result for the GMM estimator; the one for the GEL estimator is completely analogous. We divide the proof into several steps.\\
			
			\textsc{Step 1.} By Lemma \ref{lem:quad-approx}, for any $\varepsilon>0$,
			\begin{align*}
			\mathbf{P}^{\ast}_{n} \left( \frac{\hat{T}^{\ast}_{GMM,n}}{\sqrt{d(n)}} \geq t \mid Z^{n}  \right) \leq (\geq)  & \mathbf{P}^{\ast}_{n} \left( \frac{ \left \Vert n^{-1/2} \sum_{i=1}^{n} \omega_{in} g(X_{i},\theta_{0})   \right \Vert^{2}_{W} }{\sqrt{d(n)}} \geq t -(+) \varepsilon \mid Z^{n}  \right) \\
			& + o_{\mathbf{P}}(1)
			\end{align*}
			and similarly, 
			\begin{align*}
			\mathbf{P} \left( \frac{\hat{T}_{GMM,n}}{\sqrt{d(n)}} \geq t \right) \geq (\leq) & \mathbf{P} \left( \frac{ \left \Vert n^{-1/2} \sum_{i=1}^{n} g(X_{i},\theta_{0})   \right \Vert^{2}_{W} }{\sqrt{d(n)}} \geq t +(-) \varepsilon   \right) - o(1).
			\end{align*}
			
			\bigskip

			\textsc{Step 2.} We now verify Assumptions \ref{ass:boot-w} and \ref{ass:data-Z} for $Z_{i} \equiv W^{1/2} g(X_{i},\theta_{0})$. The former is directly imposed, so we only need to verify the latter. 
			
			Note that  $\Sigma_{n} = W^{1/2} E_{\mathbf{P}}[g(X,\theta_{0})g(X,\theta_{0})^{T}] W^{1/2} = W^{1/2} \Omega W^{1/2}$. Thus, Assumption \ref{ass:data-Z}(i), the first part, follows by the fact that $C^{-1} \leq \lambda_{l}(\Omega) \leq C$ for all $l=1,...,d(n)$ (Assumption \ref{ass:Omega}) and Assumption \ref{ass:example-W}. Regarding the second part of Assumption \ref{ass:data-Z}(i), note that under Assumption \ref{ass:example-1}(i), for any $l \leq 2(2+\gamma)$, 
			\begin{align*}
			E_{\mathbf{P}_{n}}[||Z_{1}||_{2}^{l}] = E_{\mathbf{P}}[||W^{1/2} g(X,\theta_{0})||_{2}^{l}] \precsim E_{\mathbf{P}}[||g(X,\theta_{0})||_{2}^{l}]~(by~Assumption~\ref{ass:example-W})
			\end{align*}
			and by Assumption \ref{ass:example-1}(i), $E_{\mathbf{P}_{n}}[||Z_{1}||_{2}^{l}] \precsim d(n)^{l/2} = d(n)^{l/2}$. Thus $E_{\mathbf{P}_{n}}[||Z_{1}||_{2}^{4}] \precsim d(n)^{2}$ and $(E_{\mathbf{P}_{n}}[||Z_{1}||_{2}^{3}] )^{2} \precsim d(n)^{3}$. Hence, the expression in  the second part of Assumption \ref{ass:data-Z}(i) is of order $d(n)^{4}/n$ which is $o(1)$ by assumption. 
			
		    Assumption \ref{ass:data-Z}(ii) follows because $E_{\mathbf{P}}[||g(X,\theta_{0})||^{2(2+\gamma)}_{2}] \precsim d(n)^{2+\gamma}$ and $\frac{d(n)^{4+2\gamma}}{n^{\gamma}} = \left( \frac{ d(n)^{2+4/\gamma}}{n} \right)^{\gamma} = o(1) $ by assumption. Finally, part (iii) of the Assumption \ref{ass:data-Z} follows with $\kappa = 0$ and $ \frac{d(n)^{4} }{n} = o(1) $.\\
			
%			We can now apply our results to the RHS to each equation. In particular, let $Z_{i} \equiv W^{1/2} \frac{g(X_{i},\theta_{0})}{d(n)^{1/4}}$ and $\Sigma_{n} = W^{1/2} E_{\mathbf{P}}[g(X,\theta_{0})g(X,\theta_{0})^{T}] W^{1/2} = W^{1/2} \Omega W^{1/2}$. Thus, assumption \ref{ass:data-Z}(i) follows by the fact that $C^{-1} \leq \lambda_{l}(\Omega) \leq C$ for all $l=1,...,d(n)$ (assumption \ref{ass:Omega}) and assumption \ref{ass:example-W}, also
%			\begin{align*}
%			E_{\mathbf{P}_{n}}[||Z_{1}||_{2}^{l}] = d(n)^{-l/4} E_{\mathbf{P}}[||W^{1/2} g(X,\theta_{0})||_{2}^{l}] \precsim d(n)^{-l/4} E_{\mathbf{P}}[||g(X,\theta_{0})||_{2}^{l}]
%			\end{align*} 
%			and by assumption \ref{ass:data-Z}(ii), $E_{\mathbf{P}_{n}}[||Z_{1}||_{2}^{l}] \precsim d(n)^{l/2-l/4} = d(n)^{l/4}$ for any $l \leq 2(2+\gamma)$. Therefore the second part of assumption \ref{ass:data-Z}(i) follows since $d(n)^{4}/n = o(1)  $.
%			
%			Assumption \ref{ass:data-Z}(ii) follows because $E_{\mathbf{P}}[||g(X,\theta_{0})||^{2(2+\gamma)}_{2}] \precsim d(n)^{1+\gamma/2}$ by assumption \ref{ass:example-1}(i) and $\frac{d(n)^{3+\gamma \frac{3}{2}}}{n^{\gamma}} = \left( \frac{ d(n)^{\frac{3}{2}+3/\gamma}}{n} \right)^{\gamma} = o(1) $. Finally, part (iii) of the assumption follows with $\kappa = 0$ and $ \frac{d(n)^{4} }{n} = o(1) $.

\textsc{Step 3.} We now show that: For any $\varepsilon>0$, there exists a $T(\varepsilon)$ 	such that
%\begin{align*}
%\mathbf{P}\left( \frac{\hat{T}_{GMM,n}}{\sqrt{d(n)}} \geq t   \right) \geq & \mathbf{P}_{n} \left( \frac{ \left \Vert n^{-1/2} \sum_{i=1}^{n} g(X_{i},\theta_{0})   \right \Vert^{2}_{W} }{\sqrt{d(n)}} \geq t - \varepsilon   \right) - \varepsilon - o(1)
%\end{align*}
\begin{align*}
\mathbf{P}\left(   \frac{ \left \Vert n^{-1/2} \sum_{i=1}^{n} g(X_{i},\theta_{0})   \right \Vert^{2}_{W} }{\sqrt{d(n)}}  \geq t + \varepsilon \right) \geq & \mathbf{P}_{n} \left( \frac{ \left \Vert n^{-1/2} \sum_{i=1}^{n} g(X_{i},\theta_{0})   \right \Vert^{2}_{W} }{\sqrt{d(n)}} \geq t - \varepsilon   \right) \\
& - \varepsilon - o(1)
\end{align*}
for all $ t \geq T(\varepsilon)$. 
			
			By the Expression \ref{eqn:asym-Gauss},\begin{align*}
				\mathbf{P} \left( \frac{ \left \Vert n^{-1/2} \sum_{i=1}^{n} g(X_{i},\theta_{0})   \right \Vert^{2}_{W} }{\sqrt{d(n)}} \geq t + \varepsilon   \right)  \geq&  \mathbf{P} \left( \frac{ \left \Vert  W^{1/2} \sqrt{n} \mathbb{V}_{n} \right \Vert^{2}_{2} }{\sqrt{d(n)}} \geq t + \varepsilon   \right) \\
				& - o(1)
			\end{align*} where $\sqrt{n} \mathbb{V}_{n} \sim N(0,\Omega)$. 
Under our assumptions $d(n) \asymp tr\{(W^{1/2}\Omega W^{1/2})^{2}\}$ and thus by Lemma  \ref{lem:anticoncentration-V} (and its Remark \ref{rem:anticon-V}), it follows that for sufficiently small $\varepsilon$, $\mathbf{P} \left( \frac{ \left \Vert  W^{1/2} \sqrt{n} \mathbb{V}_{n} \right \Vert^{2}_{2} }{\sqrt{d(n)}} \geq t + \varepsilon   \right)  \geq \mathbf{P} \left( \frac{ \left \Vert  W^{1/2} \sqrt{n} \mathbb{V}_{n} \right \Vert^{2}_{2} }{\sqrt{d(n)}} \geq t - \varepsilon   \right) - 0.5 \varepsilon$. Invoking again Expression \ref{eqn:asym-Gauss}, the desired result follows. 
%			it thus follows that
%			\begin{align*}
%			\mathbf{P}\left( \frac{\hat{T}_{GMM,n}}{\sqrt{d(n)}} \geq t   \right) \geq & \mathbf{P}_{n} \left( \frac{ \left \Vert n^{-1/2} \sum_{i=1}^{n} g(X_{i},\theta_{0})   \right \Vert^{2}_{W} }{\sqrt{d(n)}} \geq t - \varepsilon   \right) - \varepsilon - o(1).
%			\end{align*}

\bigskip

	\textsc{Step 4.} For any $t \in \mathbb{R}$, 
		\begin{align*}
		& \mathbf{P}^{\ast}_{n} \left( \frac{\hat{T}^{\ast}_{GMM,n}}{\sqrt{d(n)}} \geq t \mid Z^{n}  \right)   -  \mathbf{P} \left( \frac{\hat{T}_{GMM,n}}{\sqrt{d(n)}}\geq t   \right)    \\
		\leq &   \mathbf{P}^{\ast}_{n} \left( \frac{ \left \Vert n^{-1/2} \sum_{i=1}^{n} \omega_{i,n} g(X_{i},\theta_{0})   \right \Vert^{2}_{W} }{\sqrt{d(n)}} \geq t - \varepsilon \mid Z^{n}  \right)   \\
		& -  \mathbf{P} \left( \frac{ \left \Vert n^{-1/2} \sum_{i=1}^{n} g(X_{i},\theta_{0})   \right \Vert^{2}_{W} }{\sqrt{d(n)}} \geq t + \varepsilon   \right)     + o_{\mathbf{P}}(1),~(by ~Step~1) \\
				\leq &  \mathbf{P}^{\ast}_{n} \left( \frac{ \left \Vert n^{-1/2} \sum_{i=1}^{n} \omega_{i,n} g(X_{i},\theta_{0})   \right \Vert^{2}_{W} }{\sqrt{d(n)}} \geq t - \varepsilon \mid Z^{n}  \right)  \\
				& -  \mathbf{P} \left( \frac{ \left \Vert n^{-1/2} \sum_{i=1}^{n} g(X_{i},\theta_{0})   \right \Vert^{2}_{W} }{\sqrt{d(n)}} \geq t - \varepsilon   \right)    - \varepsilon + o_{\mathbf{P}}(1),~(by ~Step~3).
		\end{align*}
		An analogous result holds for $-\left(\mathbf{P}^{\ast}_{n} \left( \frac{\hat{T}^{\ast}_{GMM,n}}{\sqrt{d(n)}} \geq t \mid Z^{n}  \right)   -  \mathbf{P} \left( \frac{\hat{T}_{GMM,n}}{\sqrt{d(n)}}\geq t   \right)     \right)$. 	Therefore, for any $\varepsilon>0$,
			{\small{\begin{align*}
			& \sup_{t \in \mathbb{R}} \left| \mathbf{P}^{\ast}_{n} \left( \frac{\hat{T}^{\ast}_{GMM,n}}{\sqrt{d(n)}} \geq t \mid Z^{n}  \right)   -  \mathbf{P} \left( \frac{\hat{T}_{GMM,n}}{\sqrt{d(n)}}\geq t   \right)     \right| \\
			\leq & \sup_{t \in \mathbb{R}} \left| \mathbf{P}^{\ast}_{n} \left( \frac{ \left \Vert n^{-1/2} \sum_{i=1}^{n} \omega_{i,n} g(X_{i},\theta_{0})   \right \Vert^{2}_{W} }{\sqrt{d(n)}} \geq t \mid Z^{n}  \right)   -  \mathbf{P} \left( \frac{ \left \Vert n^{-1/2} \sum_{i=1}^{n} g(X_{i},\theta_{0})   \right \Vert^{2}_{W} }{\sqrt{d(n)}} \geq t   \right)     \right| \\
			& + \varepsilon + o_{\mathbf{P}}(1) \\
			= & \sup_{t \in \mathbb{R}} \left| \mathbf{P}^{\ast}_{n} \left( \left \Vert n^{-1/2} \sum_{i=1}^{n} \omega_{i,n} g(X_{i},\theta_{0})   \right \Vert^{2}_{W} \geq t \mid Z^{n}  \right)   -  \mathbf{P} \left( \left \Vert n^{-1/2} \sum_{i=1}^{n} g(X_{i},\theta_{0})   \right \Vert^{2}_{W} \geq t   \right)     \right| \\
			& + \varepsilon + o_{\mathbf{P}}(1) 
			\end{align*}}}
		where the last line follows from the fact that $ \sqrt{d(n)} t \in \mathbb{R}$ for any $t \in \mathbb{R}$. The desired result thus follows from Theorem \ref{thm:boot} with $Z_{i} \equiv W^{1/2} g(X_{i},\theta_{0})$ for all $i=1,...,n$. 
		\end{proof}
		
		\subsection{Proofs of Lemmas \ref{lem:GMM-1}, \ref{lem:GEL-1} and \ref{lem:prb-pr}}
		
		\begin{proof}[Proof of Lemma \ref{lem:GMM-1}]
			
			(1) Note that\begin{align*}
			E_{\mathbf{P}^{\ast}_{n}(\cdot|Z^{n})}[|| \sqrt{n} \bar{g}^{\ast}_{n}||^{2}_{2}] = & tr\{ E_{\mathbf{P}^{\ast}_{n}(\cdot|Z^{n})}[ (n^{-1/2} \sum_{i=1}^{n} \omega_{in} g(X_{i},\theta_{0}) )(n^{-1/2} \sum_{i=1}^{n} \omega_{in} g(X_{i},\theta_{0}))^{T} ] \} \\
			= &  tr\{ n^{-1} \sum_{i=1}^{n} E_{\mathbf{P}^{\ast}_{n}(\cdot|Z^{n})}[ \omega_{in}^{2} ] g(X_{i},\theta_{0}) g(X_{i},\theta_{0})^{T} \}
			\end{align*}
			because under Assumption \ref{ass:boot-w}, the weights are centered and independent. Thus $E_{\mathbf{P}^{\ast}_{n}(\cdot|Z^{n})}[|| \sqrt{n} \bar{g}^{\ast}_{n}||^{2}_{2}] = n^{-1}\sum_{i=1}^{n} ||g(X_{i},\theta_{0})||^{2}_{2}$, and the desired result follows by the Markov inequality.\\
			
			(2) By the triangle inequality	
			\begin{align*}
			||\bar{G}^{\ast}_{n}(\theta)||_{2} \leq & ||n^{-1} \sum_{i=1}^{n} \omega_{in} \{\nabla_{\theta} g(X_{i}, \theta) - \nabla_{\theta} g(X_{i}, \theta_{0})\} ||_{2} \\
			& + n^{-1/2} ||n^{-1/2} \sum_{i=1}^{n} \omega_{in} \nabla_{\theta} g(X_{i}, \theta_{0}) ||_{2}\\
			\equiv & T_{1,n} + T_{2,n}
			\end{align*}
			where $\nabla_{\theta} g(X, \theta_{0}) \in \mathbb{R}^{d(n) \times q}$. Recall that for matrices, $||A||_{2}$ is the spectral norm. Let $||A|| \equiv tr\{ A^{T}A \}$; it is clear that $||A||_{2} \leq ||A||$. Moreover,
			\begin{align*}
			& ||n^{-1/2} \sum_{i=1}^{n} \omega_{in} \nabla_{\theta} g(X_{i}, \theta_{0}) ||^{2} \\
			= & tr \left\{ \left( n^{-1/2} \sum_{i=1}^{n} \omega_{in} \nabla_{\theta} g(X_{i}, \theta_{0}) \right)^{T} \left( n^{-1/2} \sum_{i=1}^{n} \omega_{in} \nabla_{\theta} g(X_{i}, \theta_{0}) \right)  \right\} \\
			= & tr \left\{ n^{-1} \sum_{i=1}^{n} \omega^{2}_{in} \left( \nabla_{\theta} g(X_{i}, \theta_{0}) \right)^{T}  \nabla_{\theta} g(X_{i}, \theta_{0}) \right\}\\
			& + tr \left\{ n^{-1} \sum_{i\ne j} \omega_{in}\omega_{jn} \left( \nabla_{\theta} g(X_{i}, \theta_{0}) \right)^{T}  \nabla_{\theta} g(X_{j}, \theta_{0}) \right\}.
			\end{align*}
			Applying $E_{\mathbf{P}^{\ast}_{n}(\cdot|Z^{n})}$ the second term in the RHS vanishes because of independence of the weights and zero mean. Thus, since $E_{\mathbf{P}^{\ast}_{n}(\cdot|Z^{n})}[\omega^{2}_{in}]=1$, it follows by the Markov inequality that \begin{align*}
			T_{2,n} =  O_{\mathbf{P}^{\ast}_{n}(\cdot|Z^{n})} \left( n^{-1/2} \sqrt{n^{-1} \sum_{i=1}^{n}  || \nabla_{\theta} g(X_{i}, \theta_{0})||^{2} }   \right).
			\end{align*}
			Also, note that $n^{-1} \sum_{i=1}^{n}  || \nabla_{\theta} g(X_{i}, \theta_{0})||^{2} = O_{\mathbf{P}}(d(n))$ by the Markov inequality, Assumption \ref{ass:example-1}(iii), and the fact that $||A|| \leq \sqrt{q} ||A||_{2}$. Therefore by Lemma \ref{lem:prb-pr}, $T_{2,n} =  O_{\mathbf{P}^{\ast}_{n}(\cdot|Z^{n})} \left( \sqrt{d(n)/n}   \right)$ wpa1-$\mathbf{P}$.
			
			Regarding $T_{1,n}$, note that $T_{1,n} \leq n^{-1} \sum_{i=1}^{n} |\omega_{in}| \times ||\nabla_{\theta} g(X_{i}, \theta) - \nabla_{\theta} g(X_{i}, \theta_{0}) ||_{2} $. Under Assumption \ref{ass:example-1}(iv),
			\begin{align*}
			T_{1,n} \leq n^{-1} \sum_{i=1}^{n} |\omega_{in}| \delta_{n}(X_{i})  || \theta  - \theta_{0} ||_{2} \leq n^{-1} \sum_{i=1}^{n} |\omega_{in}| \delta_{n}(X_{i})  \Delta_{n}. 
			\end{align*}
			Since weights are uniformly bounded, $T_{1,n} \precsim \Delta_{n} n^{-1} \sum_{i=1}^{n}  \delta_{n}(X_{i}) $ a.s-$\mathbf{P}^{\ast}_{n}$. Thus under Assumption \ref{ass:example-1}(iv), the Markov inequality and Lemma \ref{lem:prb-pr}, $T_{1,n} = O_{\mathbf{P}^{\ast}_{n}(\cdot|Z^{n})}(\Delta_{n} \sqrt{d(n)})$ wpa1-$\mathbf{P}$. 		
		\end{proof}
		
			\begin{proof}[Proof of Lemma \ref{lem:GEL-1}]
				
				(1) Observe that $|\lambda^{T} \omega_{in} g(X_{i},\theta)| \leq  || \lambda||_{2} |\omega_{in}| ||g(X_{i},\theta)||_{2} \precsim \sqrt{d(n)/n} |\omega_{in}| ||g(X_{i},\theta)||_{2}$. It suffices to show that
				\begin{align*}
					\sqrt{d(n)/n} \max_{i\leq n } |\omega_{in}| ||g(X_{i},\theta)||_{2} = o_{\mathbf{P}^{\ast}_{n}(\cdot|Z^{n})}(1)
				\end{align*}
				 wpa1-$\mathbf{P}$, uniformly in $\theta \in \mathcal{N}$. Since weights are uniformly bounded, it suffices to show that $\sqrt{d(n)/n}  \max_{i \leq n} \sup_{\theta \in \mathcal{N}} ||g(X_{i},\theta)||_{2} = o_{\mathbf{P}}(1)$. By the Markov inequality\begin{align*}
				 	\mathbf{P}( \max_{i\leq n } \sup_{\theta \in \mathcal{N}} ||g(X_{i},\theta)||_{2} \geq K_{n} ) \leq \frac{n}{K^{2\alpha}_{n}} E_{\mathbf{P}} [\sup_{\theta \in \mathcal{N}} ||g(X_{i},\theta)||^{2\alpha}_{2}] .
				 \end{align*}  Thus by Assumption \ref{ass:example-1}(i) and $K_{n} = n^{1/(2\alpha)} \sqrt{d(n)}$ it follows that\begin{align*}
				 	\sqrt{d(n)/n}  \max_{i \leq n} \sup_{\theta \in \mathcal{N}} ||g(X_{i},\theta)||_{2} \precsim \frac{d(n)}{n^{0.5(1-1/\alpha)}}
				 \end{align*} since $d(n)^{4}/n = o(1)$ and $\alpha \geq 2$ this implies the desired result.
				
				\medskip
				
				(2) It follows that \begin{align*}
				R^{\ast}_{n}(\theta,\lambda) \leq & ||\lambda ||^{2}_{2} \left \Vert \int_{0}^{1} n^{-1} \sum_{i=1}^{n} s_{2}(t\lambda^{T} \omega_{in}g(X_{i},\theta) )\omega^{2}_{in}g(X_{i},\theta)g(X_{i},\theta)^{T}  dt - s_{2}(0)\Omega \right \Vert_{e} \\
				\leq & ||\lambda ||^{2}_{2} s_{2}(0) \left \Vert n^{-1} \sum_{i=1}^{n} \omega^{2}_{in}g(X_{i},\theta)g(X_{i},\theta)^{T}  - \Omega \right \Vert_{e} \\
				& + ||\lambda ||^{2}_{2} \left \Vert n^{-1} \sum_{i=1}^{n} \int_{0}^{1} (s_{2}(t\lambda^{T} \omega_{in}g(X_{i},\theta)) - s_{2}(0) ) dt \omega^{2}_{in}g(X_{i},\theta)g(X_{i},\theta)^{T}  \right \Vert_{e}\\
				\leq & ||\lambda ||^{2}_{2} s_{2}(0) \left \Vert n^{-1} \sum_{i=1}^{n} \omega^{2}_{in} \{ g(X_{i},\theta)g(X_{i},\theta)^{T}  - g(X_{i},\theta_{0})g(X_{i},\theta_{0})^{T} \} \right \Vert_{e}\\
				& + ||\lambda ||^{2}_{2} s_{2}(0) \left \Vert n^{-1} \sum_{i=1}^{n} (\omega^{2}_{in} - 1)g(X_{i},\theta_{0})g(X_{i},\theta_{0})^{T}  \right \Vert_{e} \\
				& + ||\lambda ||^{2}_{2} s_{2}(0) \left \Vert n^{-1} \sum_{i=1}^{n} g(X_{i},\theta_{0})g(X_{i},\theta_{0})^{T}  - \Omega \right \Vert_{e} \\
				& + ||\lambda ||^{2}_{2} \left \Vert n^{-1} \sum_{i=1}^{n} \int_{0}^{1} (s_{2}(t\lambda^{T} \omega_{in}g(X_{i},\theta)) - s_{2}(0) ) dt \omega^{2}_{in}g(X_{i},\theta)g(X_{i},\theta)^{T}  \right \Vert_{e}\\
				\equiv & ||\lambda ||^{2}_{2} \{ s_{2}(0) (T_{1,n} + T_{2,n} + T_{3,n}) + T_{4,n}(\lambda) \}.			
				\end{align*}

				Regarding $T_{1,n}$, it is easy to see that
				\begin{align*}
				T_{1,n} = O_{\mathbf{P}^{\ast}_{n}(\cdot|Z^{n})} \left( n^{-1} \sum_{i=1}^{n} \left \Vert   g(X_{i},\theta)g(X_{i},\theta)^{T}  - g(X_{i},\theta_{0})g(X_{i},\theta_{0})^{T} \right \Vert_{e} \right).
				\end{align*}
				% ||AA^{T} - BB^{T}||_{2} \leq ||A(A-B)^{T} + (A-B)B^{T}||_{2} \leq ||A-B||_{2} (||A||_{2} + ||B||_{2}) \leq ||A-B||_{2} (||A-B||_{2} + 2||B||_{2})
				Hence, by Lemma \ref{lem:prb-pr} and after some algebra it follows that it suffices to show that $n^{-1} \sum_{i=1}^{n}|| g(X_{i},\theta) - g(X_{i},\theta_{0}) ||^{2}_{2} = O_{\mathbf{P}}(\Delta_{n}^{2} d(n))$ and\begin{align*}
				n^{-1} \sum_{i=1}^{n}|| g(X_{i},\theta) - g(X_{i},\theta_{0}) ||_{2}||g(X_{i},\theta_{0}) ||_{2} \leq & \sqrt{ n^{-1} \sum_{i=1}^{n}|| g(X_{i},\theta) - g(X_{i},\theta_{0}) ||^{2} _{e}} \\
				& \times \sqrt{ n^{-1} \sum_{i=1}^{n}|| g(X_{i},\theta_{0}) ||^{2}_{2} }\\
				= & O_{\mathbf{P}}( \Delta_{n} d(n)  ).
				\end{align*}
				These two results follow because under Assumption \ref{ass:example-1}(ii), $||g(X_{i},\theta) - g(X_{i},\theta_{0})||_{2} \leq \int_{0}^{1} ||\nabla_{\theta}g(X_{i},\theta_{0} + t(\theta-\theta_{0}))||_{2} dt  ||\theta - \theta_{0}||_{2} \leq \sup_{\theta \in \mathcal{N} } ||\nabla_{\theta}g(X_{i},\theta )||_{2} \Delta_{n}$. And under Assumption \ref{ass:example-1}(iii) and the Markov inequality,\\ $n^{-1} \sum_{i=1}^{n} \sup_{\theta \in \mathcal{N} } ||\nabla_{\theta}g(X_{i},\theta )||_{2} = O_{\mathbf{P}}(d(n)^{1/2})$. Finally, under Assumption \ref{ass:example-1}(i) and the Markov inequality, $n^{-1} \sum_{i=1}^{n} ||g(X_{i},\theta_{0})||^{2}_{2} = O_{\mathbf{P}}(d(n))$. Therefore $n ||\lambda||^{2}_{2} T_{1,n} = O_{\mathbf{P}^{\ast}_{n}(\cdot|Z^{n})} \left(  d(n)^{2} \Delta_{n}     \right)$ wpa1-$\mathbf{P}$.
				
				Regarding $T_{2,n}$ and $T_{3,n}$ it can be shown that are $O_{\mathbf{P}^{\ast}_{n}(\cdot|Z^{n})}(d(n)/\sqrt{n})$ wpa1-$\mathbf{P}$ and $O_{\mathbf{P}}(d(n)/\sqrt{n})$ resp.; the calculations are analogous to those in the proof of Lemma A.6 in DIN and thus omitted. It thus follows,	$n ||\lambda ||^{2}_{2} (T_{2,n} + T_{3,n})  = O_{\mathbf{P}^{\ast}_{n}(\cdot|Z^{n})}(\frac{d(n)^{2}}{\sqrt{n}}) = o_{\mathbf{P}^{\ast}_{n}(\cdot|Z^{n})}(\sqrt{d(n)}) $  wpa1-$\mathbf{P}$, since $ (d(n))^{3/2}/\sqrt{n} = o(1)$ by assumption.
				
%				$n ||\lambda ||^{2}_{2} (T_{2,n} + T_{3,n})  = O_{\mathbf{P}^{\ast}_{n}(\cdot|Z^{n})}(\frac{d(n)^{2}}{\sqrt{n}}) = O_{\mathbf{P}^{\ast}_{n}(\cdot|Z^{n})}(\frac{\sqrt{d(n)} (d(n))^{3/2}}{\sqrt{n}}) = o_{\mathbf{P}^{\ast}_{n}(\cdot|Z^{n})}(\sqrt{d(n)}) $  wpa1-$\mathbf{P}$, since $ (d(n))^{3/2}/\sqrt{n} = o(1)$ by assumption.
				
				Regarding the term $T_{4,n}$, since $s_{2}$ is Lipschitz at 0, it follows that\begin{align*}
					|\int_{0}^{1} (s_{2}(t\lambda^{T} \omega_{in}g(X_{i},\theta)) - s_{2}(0) ) dt| \precsim |\lambda^{T} \omega_{in}g(X_{i},\theta)|
				\end{align*} for all $t \in [0,1]$. Therefore,
				\begin{align*}
				T_{4,n}(\lambda) \leq & \left \Vert n^{-1} \sum_{i=1}^{n} |\omega_{in}|^{3}  |\lambda^{T} g(X_{i},\theta)|  g(X_{i},\theta)g(X_{i},\theta)^{T}  \right \Vert_{e}\\
				\precsim &  \left \Vert n^{-1} \sum_{i=1}^{n}  |\lambda^{T} g(X_{i},\theta)|  g(X_{i},\theta)g(X_{i},\theta)^{T}  \right \Vert_{e}\\
				\leq &   n^{-1} \sum_{i=1}^{n}  |\lambda^{T} g(X_{i},\theta)| \left \Vert g(X_{i},\theta)g(X_{i},\theta)^{T}  \right \Vert_{e}\\
				\leq &  \sqrt{ \lambda^{T} n^{-1} \sum_{i=1}^{n}  g(X_{i},\theta) g(X_{i},\theta)^{T}  \lambda }  \sqrt{n^{-1}\sum_{i=1}^{n} \left \Vert g(X_{i},\theta)g(X_{i},\theta)^{T}  \right \Vert^{2}_{2}}
				\end{align*}
				where the second line follows from the weights being uniformly bounded. By analogous arguments to those in Lemma A.6 in DIN it can be shown that $\lambda_{max}(n^{-1} \sum_{i=1}^{n}  g(X_{i},\theta) g(X_{i},\theta)^{T} ) \leq C < \infty$ wpa1-$\mathbf{P}$ and thus\\ $\sqrt{ \lambda^{T} n^{-1} \sum_{i=1}^{n}  g(X_{i},\theta) g(X_{i},\theta)^{T}  \lambda }  \precsim \sqrt{\frac{d(n)}{n}}$ wpa1-$\mathbf{P}$. It follows that\\ $n^{-1}\sum_{i=1}^{n} \left \Vert g(X_{i},\theta)g(X_{i},\theta)^{T}  \right \Vert^{2}_{2} \leq n^{-1}\sum_{i=1}^{n} \left \Vert g(X_{i},\theta)  \right \Vert^{4}_{2} = O_{\mathbf{P}}(d(n)^{2})$ by Assumption \ref{ass:example-1}(i) (observe that $\theta \in \mathcal{N}$ eventually). Therefore, by Lemma \ref{lem:prb-pr}, $n ||\lambda ||^{2}_{2}  T_{4,n}(\lambda) = O_{\mathbf{P}^{\ast}_{n}(\cdot|Z^{n})}(d(n) \sqrt{\frac{d(n)}{n}} d(n) ) = O_{\mathbf{P}^{\ast}_{n}(\cdot|Z^{n})}(\sqrt{d(n)} \frac{d(n)^{2}}{\sqrt{n}} ) =o_{\mathbf{P}^{\ast}_{n}(\cdot|Z^{n})}(\sqrt{d(n)} )  $ since $\frac{d(n)^{2}}{\sqrt{n}}  = o(1)$ by assumption.				
			\end{proof}
			
				\begin{proof}[Proof of Lemma \ref{lem:prb-pr}]
					(1) We want to establish that for any $\epsilon>0$, there exists a $M=M(\epsilon)$ and $N(\epsilon)$ such that
					\begin{align*}
					\mathbf{P} \left( \mathbf{P}^{\ast}_{n} \left( |W_{n}| \geq c_{n} M  \mid Z^{n} \right) \leq \epsilon  \right) \geq 1 - \epsilon,~\forall n \geq N(\epsilon). 
					\end{align*}
					This is equivalent to establishing that $\mathbf{P} \left( \mathbf{P}^{\ast}_{n} \left( |W_{n}| \geq c_{n} M  \mid Z^{n} \right) \geq \epsilon  \right) \leq \epsilon $. Let $A_{n} \equiv \{ Z^{n} : \mathbf{P}^{\ast}_{n} \left( |W_{n}| \geq c_{n} M  \mid Z^{n} \right) \geq \epsilon   \}$ and $B_{n} \equiv \{ X_{n} : |X_{n}| \leq \sqrt{M} c_{n}  \}$. Given $X_{n} \in B_{n}$, then $\{ W_{n} :  |W_{n}| \geq c_{n} M  \} \subseteq \{ W_{n} :  |W_{n}| \geq |X_{n}| \sqrt{M}  \}$, therefore \begin{align*}
					\mathbf{P}(A_{n}) \leq \mathbf{P}_{n}(A_{n} \cap B_{n}) + \mathbf{P}(B^{C}_{n}) \leq & \mathbf{P} \left( \mathbf{P}^{\ast}_{n} \left( |W_{n}| \geq |X_{n}| \sqrt{M}  \mid Z^{n} \right) \geq \epsilon  \right) \\
					& + \mathbf{P}(\{X_{n} : |X_{n}| \geq \sqrt{M} c_{n}\}). 
					\end{align*}
					Since $W_{n} = O_{\mathbf{P}^{\ast}_{n}(\cdot|Z^{n})}(|X_{n}|)$, the first term in the RHS can be made less than $\epsilon$ for sufficiently large $M$; similarly since $X_{n} = O_{\mathbf{P}}(c_{n})$ the second term can also be made arbitrary small. 
					
					\smallskip
					
					(2) The proof for this result is analogous to (1) and thus omitted.
				\end{proof}	
	
	\section{Proof of Proposition \ref{pro:Wald}}
	\label{app:discussion}
	
	\begin{proof}[Proof of Proposition \ref{pro:Wald}]
		
		By assumption over $V_{n}$, it follows that $\mathbb{W}_{n}(P_{n},\mathbf{P}) =  \left \Vert   c(\theta_{P_{n}}) - c(\theta_{\mathbf{P}})  \right \Vert^{2}_{V} (1 + o_{\mathbf{P}}(1)) $. Henceforth in the proof, we abuse notation and use $\mathbb{W}_{n}(P_{n},\mathbf{P}) $ to denote $\left \Vert   c(\theta_{P_{n}}) - c(\theta_{\mathbf{P}})  \right \Vert^{2}_{V} $. 
		
		Let $a_{n} \equiv 1 + \sqrt{n} || c(\theta_{P_{n}}) - c(\theta_{\mathbf{P}})  ||_{2} $. Under the null $c(\theta_{\mathbf{P}}) = 0$, the representation \ref{eqn:W-LAR} and our assumption over eigenvalues of $V$, it follows that $a_{n} \precsim 1 + \sqrt{\mathbb{W}_{n}(P_{n},\mathbf{P})}$ and 
		\begin{align*}
		\left \Vert \sqrt{n} (c(\theta_{P_{n}}) - c(\theta_{\mathbf{P}}) )  - \sqrt{n} E_{P_{n}}[\psi(X,\theta_{\mathbf{P}})]  \right \Vert_{V} = o_{\mathbf{P}}( 1 +  \sqrt{\mathbb{W}_{n}(P_{n},\mathbf{P})} ).
		\end{align*}
		
		Note that for any $x$ and $y$, $||x-y|| = o(1+||y||)$, implies $| ||x|| - ||y|| | \leq o(1+||y||)$ and thus $||y||(1+o(1)) \leq ||x|| +o(1)$ and $||x|| \leq ||y||(1+o(1)) + o(1)$. Thus applying this to $x= \sqrt{n} E_{P_{n}}[\psi(X,\theta_{\mathbf{P}})]  $ and $y = \sqrt{n} (c(\theta_{P_{n}}) - c(\theta_{\mathbf{P}}) )  $, it follows that
		\begin{align*}
		\sqrt{\mathbb{W}_{n}(P_{n},\mathbf{P})} (1+ o_{\mathbf{P}}(1) ) \leq \left \Vert \sqrt{n} E_{P_{n}}[\psi(X,\theta_{\mathbf{P}})] \right \Vert _{V } + o_{\mathbf{P}}(1)
		\end{align*}
		and 
		\begin{align*}
		\sqrt{\mathbb{W}_{n}(P_{n},\mathbf{P})} (1+ o_{\mathbf{P}}(1) ) \geq \left \Vert \sqrt{n} E_{P_{n}}[\psi(X,\theta_{\mathbf{P}})] \right \Vert _{V} - o_{\mathbf{P}}(1).
		\end{align*}
		
		Therefore, \begin{align*}
		\mathbb{W}_{n}(P_{n},\mathbf{P}) (1+ o_{\mathbf{P}}(1) ) = \left \Vert \sqrt{n} E_{P_{n}}[\psi(X,\theta_{\mathbf{P}})] \right \Vert _{V}^{2} + o_{\mathbf{P}}(1 + \left \Vert \sqrt{n} E_{P_{n}}[\psi(X,\theta_{\mathbf{P}})] \right \Vert _{V } ).
		\end{align*}
		
		It thus remains to show that $\left \Vert \sqrt{n} E_{P_{n}}[\psi(X,\theta_{\mathbf{P}})] \right \Vert _{V } \precsim \left \Vert \sqrt{n} E_{P_{n}}[\psi(X,\theta_{\mathbf{P}})] \right \Vert _{2} = O_{\mathbf{P}}(\sqrt{d(n)})$. Note that $E[\left \Vert E_{P_{n}}[\psi(X,\theta_{\mathbf{P}})] \right \Vert^{2} _{2} ] = \sum_{j=1}^{d(n)} 
		E \left[\left( E_{P_{n}}[ \psi_{[j]} (X,\theta_{\mathbf{P}})] \right)^{2} \right]$, and 
		\begin{align*}
		E \left[ \left( n^{-1} \sum_{i=1}^{n} \psi_{[j]}(X_{i},\theta_{\mathbf{P}})  \right)^{2}  \right]  = & E \left[  n^{-2} \sum_{i=1}^{n} \left( \psi_{[j]}(X_{i},\theta_{\mathbf{P}})  \right)^{2}  \right] \\
		\leq & n^{-1}  E \left[  | \psi_{[j]}(X,\theta_{\mathbf{P}}) |^{2}  \right] \\
		\precsim & n^{-1} 
		\end{align*}
		where the first equality follows because \begin{align*}
		E \left[   \psi_{[j]}(X_{i},\theta_{\mathbf{P}}) \psi_{[j]}(X_{l},\theta_{\mathbf{P}}) \right] = E \left[   \psi_{[j]}(X_{i},\theta_{\mathbf{P}})  \right] E \left[   \psi_{[j]}(X_{l},\theta_{\mathbf{P}})  \right] = 0.
		\end{align*} Thus, by the Markov inequality the result follows.

		The proof of the representation for $\mathbb{W}_{n}(P_{n}^{\ast}, P_{n})$ is analogous and omitted; it is worth pointing out, however, that for this the null hypothesis is not imposed.
	\end{proof}

\end{document}